\documentclass[journal,9p]{IEEEtran}
\usepackage{amsmath,amssymb,amsfonts}
\usepackage{algorithm}
\usepackage{graphicx}
\usepackage{textcomp}
\usepackage{color}
\usepackage{tikz}

\usetikzlibrary{
    arrows.meta,
    positioning,
    calc,
    fit,
    backgrounds
}
\usepackage{xcolor}
\usepackage{mathtools}
\usepackage[font=small,labelfont=bf]{caption}
\usepackage{subcaption}
\usepackage{booktabs}
\usepackage{hyperref}
\usepackage{mathtools}
\usepackage{subcaption} 
\usepackage[
  backend=biber,
  bibstyle=ieee,         
  citestyle=numeric-comp,
  sorting=none, 
  doi=false,     
  isbn=false,    
  url=false,     
  eprint=false   
]{biblatex}
\usepackage[backend=biber]{biblatex} 
\usepackage{balance}
\usepackage{xcolor}
\usepackage[table]{xcolor}
\usepackage{colortbl}
\usepackage{hyperref}
\hypersetup{
  colorlinks,
  linkcolor=purple,
  citecolor=blue,
  urlcolor=green
}

\newtheorem{lem}{\textbf{Lemma}}

\newtheorem{definition}{\textbf{Definition}}
\newtheorem{proposition}{\textbf{Proposition}}
\newtheorem{theorem}{\textbf{Theorem}}

\usepackage{algpseudocode}
\usepackage{subcaption}

\ifCLASSINFOpdf
\else
\fi
\begin{document}
%
\title{KKT Residual Sliding Manifolds: A Projection-Free  Structured Finite-Time Solver for Nonlinear Constrained Optimisation.}
%
%
%

\author{Abdelhafid Zenati$^{1,2}$, Kamal Youcef-Toumi$^2$
\thanks{ $^1$ Engineering Department, School of Science and Technology (SST), City University of London, Northampton Square,  
		London, EC1V 0HB, United Kingdom.
        (e-mail: abdelhafid.zenati@city.ac.uk).}
\thanks{ $^2$ Department of Mechanical Engineering, Massachusetts Insti
tute of Technology, Cambridge, Massachusetts 02139, United
States. (e-mail: \{youcef, abdelzen\}@mit.edu)}
\thanks{Manuscript received April 19, 2005; revised August 26, 2015.}}

%
%

\markboth{Journal of IEEE Transactions on Automatic Control 2026}%
{Shell \MakeLowercase{\textit{et al.}}: Bare Demo of IEEEtran.cls for IEEE Journals}
%



\maketitle


\begin{abstract}
Nonlinear Constrained Optimisation (NCO) requires the simultaneous satisfaction of the Karush–Kuhn–Tucker (KKT) conditions. Existing methods often enforce these conditions indirectly using projections, penalties, barriers, or sequential approximations. These mechanisms may introduce computational complexity, ill-conditioning, and convergence difficulties. Moreover, for nonconvex problems, a stationary KKT point is not necessarily a genuine local optimum. This paper introduces a novel formulation that recasts NCO as the stabilisation of a unified KKT residual sliding manifold and constructs the SM–KKT solver, a projection-free finite-time optimisation framework. First, stationarity, primal feasibility, and a regularised Fischer–Burmeister complementarity residual are integrated into a unified residual dynamics, with the complementarity component simultaneously encoding inequality feasibility, dual feasibility, and complementary slackness. The zero set of the resulting residual sliding manifold is then shown to be equivalent to the KKT conditions. Second, the local structure of the residual manifold is characterised through the inertia of its Jacobian. A fundamental result establishes that the required saddle-point inertia is equivalent to the simultaneous satisfaction of the Linear Independence Constraint Qualification (LICQ) and the Second-Order Sufficient Condition (SOSC). This inertia property is then exploited in the structured correction algorithm, whose local stability is guaranteed when these conditions hold. Consequently, the framework distinguishes genuine, nondegenerate, locally isolated optimal KKT equilibria from stationary KKT points that fail the required second-order optimality conditions. Based on this formulation, SM–KKTconstructs structured primal–dual corrections and employs a discrete sliding reaching law to regulate the complete KKT residual, guaranteeing finite-time reaching of the prescribed KKT residual boundary layer. Numerical experiments on four representative nonlinear benchmark problems demonstrate recovery of the reference constrained optima with KKT residuals below $(10^{-6})$ in all cases, while reaching the prescribed residual boundary layer earlier and with greater KKT accuracy than the projected primal–dual, augmented-Lagrangian, quadratic-penalty, and log-barrier baselines.
\end{abstract}


%
\IEEEpeerreviewmaketitle

\section{Introduction}

\IEEEPARstart{N}{onlinear} constrained optimisation problems arise in numerous engineering and scientific applications, including optimal control \cite{cntr}, machine learning \cite{ML}, resource allocation \cite{realloca}, and robotics and autonomous systems \cite{robotics}. A standard characterisation of candidate solutions is provided by the Karush--Kuhn--Tucker (KKT) conditions, which combine stationarity, primal feasibility, dual feasibility, and complementary slackness into a unified optimality system \cite{kkt, jahn2007introduction}. Although KKT theory is central to nonlinear programming, designing algorithms that reliably compute KKT points remains challenging, particularly when nonlinear equality constraints, active inequalities, complementarity conditions, and non-convex objective landscapes occur simultaneously \cite{kktcon,syed2026non}.

This paper adopts a control-inspired viewpoint in which the KKT conditions are regarded as defining a residual sliding manifold in the primal--dual space. From this perspective, constrained optimisation becomes a manifold-reaching problem: rather than enforcing stationarity, feasibility, and complementarity through separate algorithmic mechanisms, the optimisation dynamics directly regulate the complete KKT residual. This interpretation naturally connects NCO with sliding-mode control, where trajectories are driven towards a prescribed sliding surface and maintained near it despite nonlinearities and perturbations.

A key difficulty in such a residual-manifold formulation is the treatment of complementarity. Complementarity is nonsmooth by nature and is commonly handled through projection, active-set identification, barrier functions, or penalty terms. In contrast, this work employs a regularised Fischer--Burmeister (FB) residual to embed inequality feasibility, dual feasibility, and complementary slackness directly into the KKT residual vector. The FB residual is therefore not used merely as a complementarity reformulation; it becomes a structural component of the sliding manifold. Moreover, when the original inequality constraints are convex, the corresponding regularised FB complementarity residual preserves convexity. This property is important because it allows complementarity to be incorporated without destroying the geometric structure of the original convex optimisation problem.

Building on this residual formulation, the proposed method develops a structured correction mechanism based on the block structure of the KKT residual Jacobian. The correction exploits the coupling between stationarity, feasibility, and complementarity, and the FB block contributes a residual-dependent curvature term in directions associated with the inequality constraints. A central theoretical result shows that, at a regular KKT point, the residual Jacobian has inertia \((n,m+q,0)\) if and only if the Linear Independence Constraint Qualification (LICQ) and the Second-Order Sufficient Condition (SOSC) hold. This links the geometry of the proposed residual sliding manifold directly to classical second-order optimality theory and ensures that the targeted zero-residual manifold corresponds to a genuine locally optimal KKT point rather than an arbitrary stationary residual equilibrium.

Also, the proposed algorithm imposes a discrete sliding reaching law directly on the KKT residual. Consequently, stationarity, primal feasibility, dual feasibility, and complementarity are driven simultaneously towards zero through a single residual dynamics. The resulting solver is projection-free, does not rely on penalty or barrier continuation, and provides an explicit mechanism for finite-time reaching of the prescribed residual boundary layer and accurate regulation of the complete KKT residual. In the convex case, where the objective and inequality constraints are convex and the equality constraints are affine, driving the residual manifold to zero implies convergence to a globally optimal solution.

The principal contributions of this paper are summarised as follows:
\begin{itemize}
\item A projection-free KKT sliding residual surface formulation based on a regularised FB complementarity mapping, which embeds stationarity, feasibility, dual feasibility, and complementary slackness into a single smooth residual system.

\item A structured primal--dual correction mechanism that exploits the block structure of the residual Jacobian and avoids multiplier projection, penalty-based feasibility restoration, and barrier continuation.

\item An inertia characterisation proving that the residual Jacobian has inertia \((n,m+q,0)\) if and only if LICQ and SOSC hold at a regular KKT point, thereby linking residual-manifold geometry with second-order optimality.

\item A discrete sliding-mode residual solver that treats constrained optimisation as a manifold-reaching problem, providing finite-time convergence to the boundary layer and accurate regulation of the complete KKT residual.

\item A convexity-preserving analysis of the regularised FB complementarity residual, supporting global optimality guarantees for convex optimisation problems under the stated convergence assumptions.

\item Numerical validation on nonlinear constrained benchmark problems, including non-convex objectives, Rosenbrock-type valleys, annular feasible regions, and nonlinear equality and inequality constraints. Comparisons with projected primal--dual, augmented Lagrangian, quadratic penalty, and log-barrier methods demonstrate high KKT accuracy, accurate recovery of constrained optima, and efficient optimisation trajectories.
\end{itemize}

The remainder of the paper is organised as follows. Section~\ref{sec:related_work} reviews the relevant literature on constrained optimisation methods. Section~\ref{sec:problem_setting} formulates the optimisation problem and develops the unified KKT residual representation. The theoretical properties of the proposed residual manifold are established in Section~\ref{sec:algorithm_analysis}. Section~\ref{sec:proposed_algorithm} presents the projection-free sliding optimisation algorithm together with the structured correction strategy. Numerical results and comparative studies are reported in Section~\ref{sec:numerical_validation}. Finally, concluding remarks and directions for future research are provided in Section~\ref{sec:conclusion}.

\section{Related Work}
\label{sec:related_work}

Several classes of algorithms have been developed for NCO, each enforcing the Karush--Kuhn--Tucker (KKT) conditions through a different numerical mechanism \cite{conco}. Although these methods have achieved significant theoretical and practical success, most of them regulate stationarity, feasibility, dual feasibility, and complementarity indirectly through projection, penalty terms, barrier functions, or sequential approximation schemes. This section reviews the principal families of constrained optimisation methods and identifies the gap addressed by the proposed SM--KKT framework.

\subsection{Projection-Based Primal--Dual Methods}

Projected primal--dual methods update the primal and dual variables through descent--ascent dynamics while enforcing non-negativity of the inequality multipliers through projection onto the non-negative orthant \cite{primal,primal2}. These methods are attractive because of their conceptual simplicity, low per-iteration cost, and suitability for large-scale or distributed optimisation settings. Projection also provides a direct mechanism for maintaining dual feasibility. However, projection-based schemes regulate the KKT conditions only partially. While the projection step enforces multiplier non-negativity, stationarity, primal feasibility, and complementarity are not controlled as a single coupled residual system \cite{zhu2018projected}. As a result, convergence near active constraints may be slow, and oscillatory behaviour can occur when the active set changes repeatedly \cite{stegink2018convergence}. In contrast, the proposed SM--KKT method does not rely on projection to enforce complementarity. Instead, dual feasibility, inequality feasibility, and complementary slackness are embedded directly into the FB residual and regulated together with stationarity and equality feasibility through a unified KKT residual dynamics, thereby avoiding projection-induced discontinuities.

\subsection{Penalty and Augmented Lagrangian Methods}

Penalty methods reformulate the constrained optimisation problem by adding constraint-violation terms to the objective function \cite{ref1}. Their main advantage is implementation simplicity, since the original constrained problem is replaced by a sequence of unconstrained or weakly constrained subproblems. However, accurate recovery of the original constrained solution generally requires large penalty parameters. This may lead to ill-conditioning, slow convergence, and sensitivity to parameter selection \cite{ref2}. Moreover, a reduction in the penalised objective does not necessarily imply satisfaction of the original KKT conditions, since the algorithm may converge to points that remain infeasible for the original problem \cite{chen2016penalty}.

Augmented Lagrangian methods improve upon pure penalty methods by combining multiplier updates with quadratic penalty terms \cite{ref6}. This often yields stronger practical performance, since optimality and feasibility are handled through both multiplier estimation and penalty correction. Nevertheless, the convergence behaviour remains dependent on the interaction between penalty parameters, multiplier updates, and subproblem accuracy. Poor parameter choices may lead to slow convergence, oscillatory multiplier behaviour, or excessive penalty growth \cite{birgin2025global}.

The proposed SM--KKT framework differs fundamentally from penalty and augmented Lagrangian methods in that it does not modify the objective function through penalty terms. Instead, the complete KKT residual is treated as the primary quantity to be regulated. Consequently, feasibility and complementarity are enforced directly through the residual sliding manifold, eliminating penalty continuation, reducing ill-conditioning associated with large penalty parameters, and avoiding repeated penalty and multiplier tuning. Furthermore, because the original optimisation problem is preserved, the proposed framework directly regulates the true KKT conditions rather than the optimality conditions of a penalised surrogate problem.

\subsection{Barrier and Interior-Point Methods}

Barrier and interior-point methods enforce inequality feasibility by constructing smooth approximations that keep the iterates within the interior of the feasible region \cite{ref3,ref4}. These methods possess strong theoretical foundations and form the basis of many successful optimisation solvers. Their effectiveness is closely related to the central-path interpretation, where the solution of the original problem is approached by solving a sequence of barrier subproblems. Despite their success, barrier-based approaches require a continuation strategy for the barrier parameter and may need strictly feasible initial points \cite{ref5}. As the barrier parameter decreases, the associated linear systems can become increasingly ill-conditioned, especially near active constraints. Furthermore, inequality handling is achieved by remaining inside the feasible region, rather than by directly enforcing complementarity as part of the KKT residual \cite{wright1998ill}.

By contrast, the proposed SM--KKT framework incorporates complementarity directly into the residual sliding system through a regularised Fischer--Burmeister mapping. Consequently, it eliminates barrier continuation and central-path tracking, avoids the need for strictly feasible iterates, and reduces the ill-conditioning associated with vanishing barrier parameters. Instead, the complete KKT residual is regulated directly, enabling accurate enforcement of stationarity, feasibility, and complementarity within a unified projection-free framework.

\subsection{Sequential Quadratic Programming}

Sequential Quadratic Programming (SQP) methods are among the most powerful approaches for smooth NCO \cite{ref8}. They approximate the nonlinear problem by a sequence of quadratic programming subproblems constructed from local models of the objective and constraints. Under suitable regularity and second-order conditions, SQP methods can achieve fast local convergence and high solution accuracy. However, this performance is obtained at the cost of solving a constrained quadratic subproblem at each iteration \cite{solodov2004sequential}. The method also depends on reliable Hessian or quasi-Newton approximations, suitable merit or filter mechanisms, and effective line-search or trust-region strategies \cite{shen2016stabilized, murray1997some}. In non-convex settings, additional safeguards are often needed to handle indefiniteness, inconsistent linearisations, or convergence to non-minimising stationary points \cite{andrei2022sequential}. 

The proposed SM--KKT framework does not solve a quadratic programming subproblem at every iteration. Instead, it computes structured primal--dual corrections directly from the KKT residual Jacobian and embeds them within a sliding residual dynamics. Consequently, the computational overhead associated with repeated QP solves, Hessian approximations, and line-search or trust-region strategies is avoided. Moreover, the residual Jacobian is exploited not merely as a local linearisation, but as a tool for inertia analysis, establishing a direct connection between the residual manifold, LICQ, SOSC, and the local optimality of the limiting solution.

\subsection{Fischer--Burmeister and Complementarity Reformulations}

Complementarity reformulations based on nonlinear complementarity functions, including the FB function, provide an alternative to active-set and projection-based treatments of inequality constraints \cite{ref9}. Their main advantage is that inequality feasibility, multiplier non-negativity, and complementary slackness can be represented through a single residual equation. This makes complementarity functions particularly suitable for residual-based formulations and semismooth Newton-type methods \cite{jiang1998global, de2010exact}. In the present work, however, the FB residual is used in a different way. It is not introduced only to solve a complementarity equation. Instead, it becomes a structural component of the proposed KKT sliding manifold. This allows stationarity, equality feasibility, inequality feasibility, dual feasibility, and complementary slackness to be regulated simultaneously through one residual dynamics. Moreover, the regularised FB residual used in this paper preserves convexity of the complementarity residual whenever the original inequality constraints are convex. This structure-preserving property is important because it allows complementarity to be embedded into the residual manifold without destroying the geometry of the original convex optimisation problem.
\begin{figure}[!t]
\centering

\begin{tikzpicture}[
    >=Latex,
    font=\scriptsize,
    flow/.style={
        -{Latex[length=1.4mm]},
        semithick
    },
    box/.style={
        draw,
        rounded corners=2pt,
        align=center,
        inner xsep=2pt,
        inner ysep=3pt
    }
]


\node[
    box,
    draw=red!55!black,
    fill=red!5,
    text width=0.80\columnwidth
] (problem) {
    \textbf{Constrained Optimisation Problem}
    \\[1mm]
    $\displaystyle
    \min_x f(x)
    \qquad
    \mathrm{s.t.}\qquad
    g(x)\leq0,\quad h(x)=0
    $
};


\node[
    box,
    draw=orange!70!black,
    fill=orange!8,
    below=14mm of problem,
    text width=0.235\columnwidth,
    minimum height=17mm,
    xshift=-0.28\columnwidth
] (sx) {
    \textbf{Stationarity}
    \\[1mm]
    $\mathcal M_x:S^x(z)=0$
    \\[1mm]
    Eq.~\eqref{eq:discrete_Sx}
};

\node[
    box,
    draw=purple!40!black,
    fill=yellow!8,
    below=14mm of problem,
    text width=0.235\columnwidth,
    minimum height=17mm
] (sc) {
    \textbf{Complementarity}
    \\[1mm]
    $\mathcal M_c:S^c(z)=0$
    \\[1mm]
    Eq.~\eqref{eq:discrete_Sc}
};

\node[
    box,
    draw=magenta!55!black,
    fill=magenta!5,
    below=14mm of problem,
    text width=0.235\columnwidth,
    minimum height=17mm,
    xshift=0.28\columnwidth
] (sh) {
    \textbf{Feasibility}
    \\[1mm]
    $\mathcal M_h:S^h(z)=0$
    \\[1mm]
    Eq.~\eqref{eq:discrete_Sh}
};


\coordinate (branch)
    at ($(problem.south)+(0,-6mm)$);

\draw[semithick]
    (problem.south) -- (branch);

\draw[flow]
    (branch) -| (sx.north);

\draw[flow]
    (branch) -- (sc.north);

\draw[flow]
    (branch) -| (sh.north);

\node[
    font=\tiny\bfseries,
    text=blue!70!black,
    fill=white,
    align=center,
    inner xsep=0mm,
    inner ysep=0.7mm,
    anchor=west
] at ([xshift=-8mm,yshift=1mm]branch)
{KKT Sliding Residual\\ Formulation};


\node[
    draw=red!55!black,
    dashed,
    rounded corners=5pt,
    line width=1pt,
    inner xsep=3mm,
    inner ysep=3mm,
    fit=(problem)(sx)(sc)(sh),
    label={
        [font=\bfseries\footnotesize,
         text=red!55!black,
         fill=white]
        above:Sliding Problem Formulation
    }
] (stageone) {};


\node[
    box,
    draw=cyan!60!black,
    fill=cyan!6,
    below=11mm of stageone.south,
    text width=0.80\columnwidth
] (residual) {
    \textbf{Unified KKT Residual}
    \\[1mm]
    $\displaystyle
    S_k=
    \begin{bmatrix}
    S_k^x & S_k^c & S_k^h
    \end{bmatrix}^{\!\top}
    \qquad
    \text{Eq.~\eqref{eq:discrete_residual}}
    $
};


\node[
    box,
    draw=orange!70!black,
    fill=yellow!10,
    below=5mm of residual,
    text width=0.80\columnwidth
] (reach) {
    \textbf{Sliding Reaching Law}
    \\[1mm]
    $\displaystyle
    r_k=
    -\left[
    QS_k+
    \mathcal K
    \operatorname{sat}
    \left(\frac{S_k}{\phi}\right)
    \right]
    \qquad
    \text{Eq.~\eqref{eq:desired_residual}}
    $
};

\draw[flow]
(residual) -- (reach);


\node[
    box,
    draw=green!55!black,
    fill=green!7,
    below=5mm of reach,
    text width=0.80\columnwidth
] (correction) {
    \textbf{Structured Primal--Dual Correction}
    \\[1mm]
    $\displaystyle
    J_k=
    \left.
    \frac{\partial S}{\partial z}
    \right|_{z=z_k},
    \quad
    J_k\Delta z_k=r_k
    $
    \\[1mm]
    $\displaystyle
    \Delta z_k=
    \begin{bmatrix}
    \Delta x_k&
    \Delta\lambda_k&
    \Delta\mu_k
    \end{bmatrix}^\top
    \quad
    \text{Eqs.~\eqref{eq:linear_residual_system}--%
    \eqref{eq:block_jacobian_discrete}}
    $
};

\draw[flow]
(reach) -- (correction);


\node[
    box,
    draw=purple!60!black,
    fill=purple!6,
    below=5mm of correction,
    text width=0.80\columnwidth
] (update) {
    \textbf{Discrete Primal--Dual Update}
    \\[1mm]
    $\displaystyle
    z_{k+1}=z_k+\eta\Delta z_k
    \qquad
    \text{Eq.~\eqref{eq:discrete_algorithm}}
    $
};

\draw[flow]
(correction) -- (update);


\node[
    draw=green!55!black,
    dashed,
    rounded corners=4pt,
    line width=0.75pt,
    inner xsep=3mm,
    inner ysep=3mm,
    fit=(residual)(reach)(correction)(update),
    label={
 [font=\bfseries\footnotesize,
     text=green!50!black,
     fill=white,
     inner xsep=3mm,
     inner ysep=1mm]
    above=0pt:Sliding Solution Algorithm
    }
] (stagetwo) {};

\begin{scope}[on background layer]
\draw[flow]
(stageone.south)
--
++(0,-5mm)
--
(residual.north);
\end{scope}


\node[
    box,
    draw=blue!60!black,
    fill=blue!6,
    below=11mm of stagetwo.south,
    text width=0.34\columnwidth,
    minimum height=10mm
] (test) {
    $\displaystyle
    \|S_{k+1}\|
    \leq
    \varepsilon_{\mathrm{tol}}
    $?
};


\node[
    box,
    draw=orange!70!black,
    fill=orange!8,
    left=3mm of test,
    text width=0.17\columnwidth,
    minimum height=10mm
] (repeat) {
    $k\leftarrow k+1$
    \\
    \textbf{Repeat}
};


\node[
    box,
    draw=green!55!black,
    fill=green!8,
    right=3mm of test,
    text width=0.20\columnwidth,
    minimum height=10mm
] (solution) {
    \textbf{Converged}
    \\
    KKT solution
};

\draw[flow]
(test.west)
--
node[above,font=\tiny]{No}
(repeat.east);

\draw[flow]
(test.east)
--
node[above,font=\tiny]{Yes}
(solution.west);


\node[
    draw=blue!55!black,
    dashed,
    rounded corners=4pt,
    line width=0.65pt,
    inner xsep=3mm,
    inner ysep=5mm,
    fit=(repeat)(test)(solution),
    label={
        [font=\bfseries\footnotesize,
         text=blue!55!black,
         fill=white]
        above:Convergence Criterion
    }
] (stagethree) {};

\begin{scope}[on background layer]
\draw[flow]
(update.south)
--
(test.north);
\end{scope}
\end{tikzpicture}

\caption{Conceptual architecture of the proposed SM--KKT framework.
The sliding problem formulation expresses the constrained optimisation
problem through the stationarity, regularised Fischer--Burmeister
complementarity, and equality-feasibility manifolds. The sliding solution
algorithm assembles these conditions into the unified KKT residual $S_k$,
applies the sliding reaching law, computes the structured primal--dual
correction, and updates the primal--dual variables. Convergence is declared
when $\|S_{k+1}\|\leq\varepsilon_{\mathrm{tol}}$; otherwise, the iteration
continues until the prescribed KKT tolerance is reached.}
\label{fig:smkkt_framework}

\end{figure}
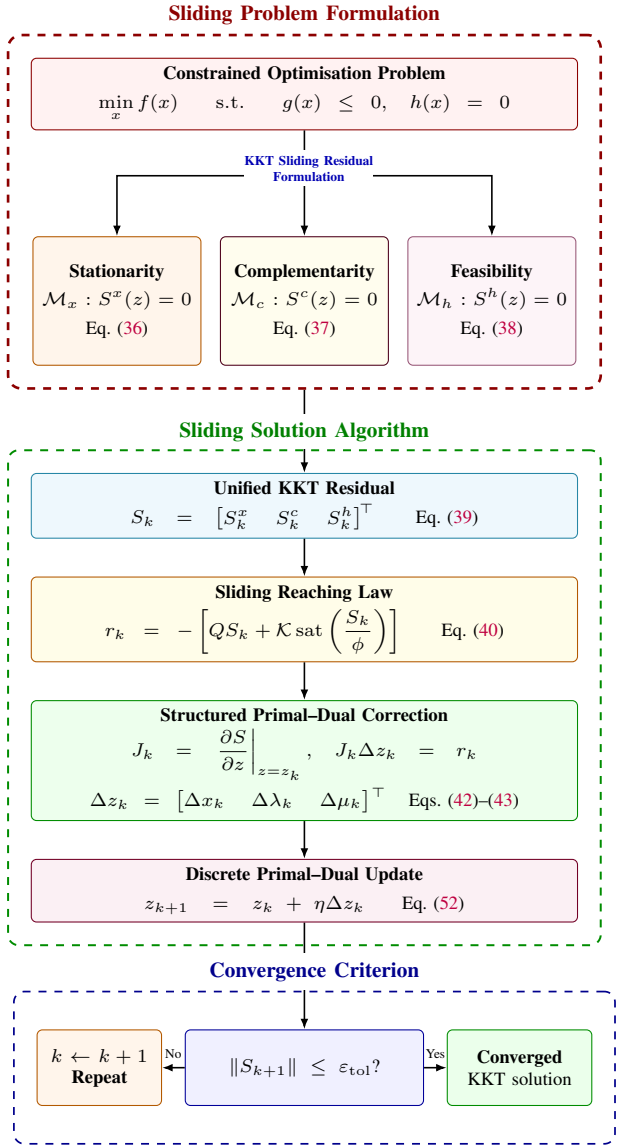
\section{Problem Setting}
\label{sec:problem_setting}

This work considers a class of NCO problems involving both inequality and equality constraints. Let \(x\in\mathbb{R}^{n}\) denote the primal decision variable. The optimisation problem is written as
\begin{equation}
\label{eq:problem}
\begin{aligned}
\min_{x\in\mathbb{R}^{n}} \quad & f(x)\\
\mathrm{s.t.}\quad 
& g(x)\leq 0,\\
& h(x)=0,
\end{aligned}
\end{equation}
where \(f:\mathbb{R}^{n}\rightarrow\mathbb{R}\) is the objective function, \(g:\mathbb{R}^{n}\rightarrow\mathbb{R}^{m}\) denotes the vector of inequality constraints, and \(h:\mathbb{R}^{n}\rightarrow\mathbb{R}^{q}\) denotes the vector of equality constraints. The inequality \(g(x)\leq 0\) is interpreted componentwise, namely
\begin{equation}
g_i(x)\leq 0,\qquad i=1,\ldots,m.
\end{equation}
Throughout this paper, it is assumed that \(f\), \(g\), and \(h\) are continuously differentiable on an open set containing the feasible region. The feasible set associated with \eqref{eq:problem} is
\begin{equation}
\label{eq:feasible_set}
\mathcal{F}
=
\left\{
x\in\mathbb{R}^{n}
:~~
g(x)\leq 0,\;~
h(x)=0
\right\}.
\end{equation}
The inequality constraints restrict the solution to lie on or inside admissible regions, whereas the equality constraints define manifolds on which the solution must lie exactly. In this sense, the equality constraints are not one-sided restrictions; deviations on either side of \(h(x)=0\) are infeasible.

\subsection{Karush–Kuhn–Tucker (KKT) Conditions}
The Lagrangian associated with \eqref{eq:problem} is defined by
\begin{equation}
\label{eq:lagrangian}
\mathcal{L}(x,\lambda,\mu)
=
f(x)+\lambda^{\top}g(x)+\mu^{\top}h(x),
\end{equation}
where \(\lambda\in\mathbb{R}^{m}_{+}\) is the vector of Lagrange multipliers associated with the inequality constraints and \(\mu\in\mathbb{R}^{q}\) is the vector of Lagrange multipliers associated with the equality constraints. The inequality multipliers must satisfy
\begin{equation}
\lambda\geq 0,
\end{equation}
whereas the equality multipliers are free in sign:
\begin{equation}
\mu\in\mathbb{R}^{q}.
\end{equation}

A point \(x^\star\in\mathcal{F}\), together with multipliers \(\lambda^\star\) and \(\mu^\star\), satisfies the KKT conditions if (see \cite{kktcon,voloshinov2018generalization})
\begin{subequations}
\label{eq:kkt}
\begin{align}
\nabla_x\mathcal{L}(x^\star,\lambda^\star,\mu^\star)&=0,
\label{eq:kkt_stationarity}\\
g(x^\star)&\leq 0,
\label{eq:kkt_ineq_feas}\\
h(x^\star)&=0,
\label{eq:kkt_eq_feas}\\
\lambda^\star&\geq 0,
\label{eq:kkt_dual_feas}\\
\lambda^\star\odot g(x^\star)&=0.
\label{eq:kkt_complementarity}
\end{align}
\end{subequations}
Here, \(\odot\) denotes componentwise multiplication. Condition \eqref{eq:kkt_stationarity} is the stationarity condition:

\begin{equation}
\nabla f(x^\star)
+
J_g(x^\star)^\top \lambda^\star
+
J_h(x^\star)^\top\mu^\star
=
0.
\end{equation}
where $J_g\in\mathbb{R}^{m\times n}$ and $J_h\in\mathbb{R}^{q\times n}$ are the Jacobians of constraints $g$ and $h$ respectively, defined as
\begin{equation} \label{jacg}
J_g(x^\star)
=
\begin{bmatrix}
\nabla g_1(x^\star)^\top\\
\nabla g_2(x^\star)^\top\\
\vdots\\
\nabla g_m(x^\star)^\top
\end{bmatrix},~~J_h(x^\star)
=
\begin{bmatrix}
\nabla h_1(x^\star)^\top\\
\nabla h_2(x^\star)^\top\\
\vdots\\
\nabla h_q(x^\star)^\top
\end{bmatrix}.
\end{equation}

The conditions \eqref{eq:kkt_ineq_feas} and \eqref{eq:kkt_eq_feas} are the primal feasibility conditions, where \eqref{eq:kkt_dual_feas} is the dual feasibility condition, and  \eqref{eq:kkt_complementarity} is the complementary slackness condition.
\subsection{ Second-order Condition }
In addition to the first-order KKT conditions, second-order optimality conditions are used to determine whether a KKT point is a constrained local minimiser \cite{kien2018second, meng2015first}. Let $H^\star
=
\nabla_{xx}^{2}\mathcal{L}(x^\star,\lambda^\star,\mu^\star)
$, then
\begin{equation}
\label{hes}
H^\star=
\nabla^{2}f(x^\star)
+
\sum_{i=1}^{m}\lambda_i^\star\nabla^{2}g_i(x^\star)
+
\sum_{j=1}^{q}\mu_j^\star\nabla^{2}h_j(x^\star).
\end{equation}
Let \(\mathcal A(x^\star)=\{i:g_i(x^\star)=0\}\) denote the active inequality set and define
\[
C=
\begin{bmatrix}
J_{g_{\mathcal A}}(x^\star)\\
J_h(x^\star)
\end{bmatrix}.
\]
The admissible first-order feasible directions are those in the null space of \(C\). Let \(Z\) be a basis for this null space, so that
\begin{equation}
   CZ=0. 
\end{equation}
The second-order necessary condition requires
\begin{equation}
    Z^\top H^\star Z\succeq0,
\end{equation}
whereas the second-order sufficient condition requires
\begin{equation}
    Z^\top H^\star Z\succ0.
\end{equation}
Thus, the active constraint Jacobian determines the feasible tangent directions, while the reduced Hessian \(Z^\top H^\star Z\) determines the curvature along those directions. This reduced-Hessian condition will be used below to relate classical second-order optimality theory to the inertia of the proposed KKT residual Jacobian.

For convex problems, the second-order sufficient condition is not
required to establish optimality: any point satisfying the KKT
conditions is a global minimiser \cite{meng2015first}. However, SOSC, or an equivalent strict curvature condition, remains necessary when one seeks local isolation \cite{fernandez2012local}, nonsingularity of the residual Jacobian, and
asymptotic stability of a particular KKT point. In the absence of
strict curvature, the solution set may be non-isolated and the
residual Jacobian may be singular, even though every KKT point is
globally optimal.

Before developing the proposed framework, several preliminary concepts are introduced that will be used throughout the theoretical analysis. In particular, the inertia of structured saddle-point matrices plays a central role in establishing the relationship between the residual Jacobian, second-order optimality conditions, and the convergence properties of the proposed algorithm. The following definition and lemma summarise the required results.

\begin{definition}[\color{blue}Matrix Inertia \cite{ref10}] Let $A\in\mathbb{R}^{n\times n}$ be a square matrix whose eigenvalues are counted according to their algebraic multiplicities. The inertia of \(A\), denoted by \[ \operatorname{In}(A), \] is defined as the ordered triple \[ \operatorname{In}(A) = (n_{+},n_{-},n_{0}), \] where \begin{itemize} \item \(n_{+}\) is the number of eigenvalues of \(A\) with positive real part, \item \(n_{-}\) is the number of eigenvalues of \(A\) with negative real part, \item \(n_{0}\) is the number of eigenvalues of \(A\) lying on the imaginary axis. \end{itemize} Thus, \[ n_{+}+n_{-}+n_{0}=n. \] If \(A\) is symmetric, the inertia reduces to the numbers of positive, negative, and zero eigenvalues of \(A\). \end{definition}
\begin{lem}[ \color{blue}Inertia of the Saddle-Point Matrix \cite{ref11}] \label{lem:kkt_inertia} Let \[ K = \begin{bmatrix} G & A^\top\\ A & 0 \end{bmatrix},~~G\in\mathbb{R}^{n\times n}, \quad A\in\mathbb{R}^{m\times n}, \] and assume that \(A\) has full row rank, $\operatorname{rank}(A)=m.$ Let \[ Z\in\mathbb{R}^{n\times(n-m)} \] be a basis for the null space of \(A\), i.e., $AZ=0.$ Then \begin{equation} \label{eq:kkt_inertia} \operatorname{In}(K) = \operatorname{In}(Z^\top GZ) + (m,m,0). \end{equation} Consequently, if the reduced Hessian $G$ satisfies \begin{equation} \label{eq:kkt_inertia_positive} Z^\top GZ \succ 0 \Leftrightarrow \operatorname{In}(K) = (n,m,0). \end{equation} In particular, the matrix \(K\) is nonsingular. \end{lem}

While the preceding results establish the spectral foundation of the proposed framework, complementarity is incorporated through the regularised Fischer--Burmeister (FB) function. The following lemma introduces the FB function used throughout this paper.
\begin{lem}[\color{blue}Fischer--Burmeister Function \cite{ref9}]
\label{def:FB} The Fischer--Burmeister function is the mapping defined by
\begin{equation}
\label{eq:FB_definition}
\phi_{\mathrm{FB}}:\mathbb{R}^2\rightarrow\mathbb{R}, ~~\phi_{\mathrm{FB}}(a,b)
=
a+b-\sqrt{a^{2}+b^{2}}.
\end{equation}

The Fischer--Burmeister function is a nonlinear complementarity function and satisfies
\[
\phi_{\mathrm{FB}}(a,b)=0
\]
if and only if
\[
a\ge0,
\qquad
b\ge0,
\qquad
ab=0.
\]
\end{lem}

In this work, FB function is not merely a complementarity reformulation, as will be shown later. It is essential because it enables the complementarity conditions to be incorporated directly into the residual sliding manifold. Rather than treating complementarity as a separate algebraic condition, the FB function converts it into a residual equation that becomes an integral part of the sliding manifold. Consequently, stationarity, feasibility, and complementarity are driven simultaneously towards zero through a unified residual dynamics. A second advantage is that the FB residual produces a favourable Jacobian structure. In the proposed framework, the Jacobian of the residual system is directly related to the inertia and second-order properties of the optimisation problem. This connection allows the convergence analysis in noncovex problem to be linked to second-order optimality conditions and provides a mechanism for distinguishing convergence to a local minimiser from convergence to just  stationary or saddle point.

Furthermore, the FB formulation preserves important structural properties of the original optimisation problem. In particular, when the constraint function $g(x)$ is convex, the corresponding regularised FB complementarity residual remains convex. Unlike many complementarity reformulations that may alter the geometry of the feasible set, the proposed residual preserves the convex structure of the inequality constraints while embedding complementarity directly into the sliding manifold. As a result and under the assumptions of convex objective and constraint functions together with affine equality constraints, the condition $S(z)=0$ is equivalent to satisfaction of the KKT conditions. Since the KKT conditions are necessary and sufficient for optimality in convex optimisation, driving the sliding manifold to zero implies convergence to the globally optimal solution rather than merely convergence to a stationary residual equilibrium.

\section{Theoretical Foundations of the KKT Sliding Manifold Framework}
\label{sec:algorithm_analysis}
The core idea of this paper is to construct a projection-free dynamical optimisation scheme based on a sliding-mode residual approach, whose convergence is assessed directly through the residuals associated with the Karush--Kuhn--Tucker (KKT) conditions in \eqref{eq:kkt}. Rather than treating the KKT conditions as separate algebraic constraints, the proposed formulation gathers stationarity, equality feasibility, inequality feasibility, dual feasibility, and complementary slackness into a unified residual system. This residual system is interpreted as a sliding manifold, and the algorithm is designed to drive the trajectory towards this manifold through a signed primal--dual reaching law.

\subsection{Sliding Manifold Building}
An elegant component of the proposed formulation is the use of a regularised Fischer--Burmeister complementarity surface , see Lemma \ref{def:FB}. This surface encodes the inequality-related KKT conditions in a single residual block, namely
\[
\lambda\geq0,\qquad g(x)\leq0,\qquad \lambda\odot g(x)=0.
\]
Thus, inequality feasibility, dual feasibility, and complementary slackness are enforced simultaneously without requiring an explicit projection of the multipliers onto the non-negative orthant and allows the algorithm to operate directly in the primal--dual variables $z=[x~
\lambda~
\mu]^T$. More precisely, the proposed residual surface is constructed as
\begin{equation}
    S(z)
=
\begin{bmatrix}
S^x(z)&
S^c(z)&
S^h(z)
\end{bmatrix}^T,
\end{equation}

where \(S^x\) represents stationarity, \(S^h\) represents equality feasibility, and \(S^c\) is a regularised nonlinear complementarity residual based on the FB surface. We will prove later that the zero set of \(S\) is therefore equivalent to the KKT set of the original constrained optimisation problem. Consequently, convergence of the sliding residual to zero implies simultaneous satisfaction of stationarity, equality feasibility, inequality feasibility, dual feasibility, and complementary slackness.

The stationarity residual is defined as
\begin{equation}
\label{eq:Sx}
S^x(z)
=
\nabla_x\mathcal{L}(x,\lambda,\mu).
\end{equation}

To encode the inequality-related KKT conditions, we introduce a nonlinear complementarity residual based on a regularised FB surface. For each inequality constraint, define
\begin{equation}
\label{eq:Sc}
S^c(z)
= -\phi_{\mathrm{FB}}(\lambda,-g(x))=
-\lambda
+
g(x)
+
\sqrt{\lambda^2+g(x)^2+\rho_c},
\end{equation}
where \(\rho_c>0\) is a small regularisation parameter and all operations are performed componentwise. In the limiting case \(\rho_c\rightarrow0\), the condition
\begin{equation}
    S^c(z)=0 \Longleftrightarrow \lambda\ge0,
~
g(x)\le0,
~
\lambda\odot g(x)=0.
\end{equation}

Thus, inequality feasibility, dual feasibility, and complementary slackness are encoded simultaneously through a single residual block. The equality-feasibility residual is defined as
\begin{equation}
\label{eq:Sh}
S^h(z)
=
h(x).
\end{equation}

The complete KKT/sliding residual is therefore given by
\begin{equation}
\label{eq:full_residual}
S(z)
=
\begin{bmatrix}
S^x(z)\\
S^c(z)\\
S^h(z)
\end{bmatrix}
=
\begin{bmatrix}
\nabla_x\mathcal{L}(x,\lambda,\mu)\\
-\phi_{\mathrm{FB}}(\lambda,-g(x))\\
h(x)
\end{bmatrix}.
\end{equation}

The proposed optimisation dynamics are designed to drive the residual surface \eqref{eq:full_residual} to zero. This is significant because the components of \(S\) correspond directly to the KKT conditions: \(S^x\) represents stationarity, \(S^c\) represents the inequality-related KKT conditions through the FB complementarity surface, and \(S^h\) represents equality feasibility. The role of \eqref{eq:full_residual} can be understood as follows. If
\begin{equation}
S(z^\star)=0,
\end{equation}
then, by construction,
\begin{equation}
\nabla_x\mathcal{L}(x^\star,\lambda^\star,\mu^\star)=0,
\end{equation}
\begin{equation}
-\phi_{\mathrm{FB}}(\lambda,-g(x))=0,
\end{equation}
and
\begin{equation}
h(x^\star)=0.
\end{equation}
In the limiting complementarity case \(\rho_c\rightarrow0\), the condition \(S^c(z^\star)=0\) implies
\begin{equation}
\lambda^\star\ge0,
\qquad
g(x^\star)\le0,
\qquad
\lambda^\star\odot g(x^\star)=0.
\end{equation}
Therefore, stationarity, primal feasibility, dual feasibility, complementary slackness, and equality feasibility all hold simultaneously. Hence, the full KKT conditions are satisfied.

This formulation provides the basis for the sliding-mode-inspired algorithm developed in this work. The main idea is to design signed primal--dual update laws that force the trajectory towards the residual manifold \(S=0\). Equality constraints are incorporated directly through \(h(x)\), while the inequality-related KKT conditions are encoded through the FB residual \(S^c\). The resulting method is projection-free because multiplier non-negativity and complementarity are enforced through the complementarity residual rather than through explicit projection.

\subsection{Advantages of the Residual Sliding Manifold Formulation}
\label{sec:residual_manifold}

The selection of the residual manifold
\begin{equation}\label{resi}
   \mathcal M = \{z=
\begin{bmatrix}
x&
\lambda&
\mu
\end{bmatrix}^T:~~~S(z)=0\} 
\end{equation}

plays a central role in the proposed approach. Unlike penalty,
barrier, or merit-function formulations, which replace the
original constrained optimisation problem by an auxiliary
objective, the proposed residual is constructed directly from the
first-order optimality conditions. Consequently, no additional
approximations or relaxation terms are introduced, making the
resulting formulation one of the least conservative descriptions
of the constrained optimisation problem.

The residual vector combines the stationarity, complementarity,
and feasibility conditions into a single manifold \eqref{eq:full_residual}. As will be shown later in
Proposition~\ref{prop:kkt_equiv_s_zero}, the manifold $S(z)=0$
is equivalent to the KKT conditions. Therefore, driving the
residual to zero is not merely a surrogate objective but directly
enforces the optimality conditions of the original constrained
optimisation problem. Consequently, if an update law can be
constructed such that 
\begin{equation}
    S_k\rightarrow0,
\end{equation}
then the iterates converge to a KKT point of the original
problem. The usefulness of this formulation extends beyond first-order
optimality. The Jacobian of the residual mapping possesses a
special structured form inherited from the KKT conditions
themselves. As shown in the subsequent analysis, the Jacobian
evaluated at a regular solution satisfies the inertia property
\begin{equation} \label{inertia}
   \operatorname{In}(J^\star) =(n,m+p,0). 
\end{equation}
This inertia structure implies that the corresponding KKT point
is non-degenerate and locally isolated. Moreover, it is
equivalent to the second-order sufficient optimality condition
under the standard constraint qualification assumptions.
Consequently, $S(z^\star)=0$ together with the inertia condition identifies a genuine local minimum rather than a saddle point or a degenerate stationary solution. In this sense, the residual manifold not only encodes the first-order optimality conditions but also reveals the local second-order geometry of the optimisation landscape.

Another important advantage of the proposed residual formulation
is computational. The particular choice of the stationarity,
Fischer--Burmeister, and feasibility residuals yields the
structured Jacobian
\begin{equation}\label{j}
    J
=\dfrac{\partial S}{\partial z}=
\begin{bmatrix}
H & J_g^\top & J_h^\top
\\
D^gJ_g & D^\lambda & 0
\\
J_h & 0 & 0
\end{bmatrix}.
\end{equation}
where 
\begin{equation}\label{Dg}
    D^g
=
\operatorname{diag}
\left(
1+
\frac{g_i(x)}
{\sqrt{\lambda_{i}^{2}+g_i(x)^2+\rho_c}}
\right),
\end{equation} 
\begin{equation}\label{Dlam}
   D^\lambda
=
\operatorname{diag}
\left(
-1+
\frac{\lambda_{i}}
{\sqrt{\lambda_{i}^{2}+g_i(x)^2+\rho_c}}
\right). 
\end{equation}
The complementarity block contributes the diagonal matrix
\(D_\lambda\), which can be eliminated analytically. This permits
the original system of dimension \((n+m+p)\) to be reduced to a
saddle-point system of dimension \((n+p)\). As a consequence, the
correction direction can be computed without explicitly forming
the inverse of the full residual Jacobian. The resulting solver
preserves the exact Newton-type correction associated with the linearised residual equation while substantially reducing the computational burden.

An additional advantage of the proposed residual sliding formulation is that the Fischer--Burmeister block introduces a stabilising curvature correction in the reduced Hessian. Indeed, at iteration $k$, we have
\begin{equation}
   W_k=-(D_k^\lambda)^{-1}D_k^g\succ0, 
\end{equation}

the reduced curvature matrix can be written as
\begin{equation}\label{reduced}
\bar H_k
=
H_k
+
J_g(x_k)^\top W_kJ_g(x_k). 
\end{equation}
Therefore, the inequality residual contributes the positive-semidefinite term
\[
J_g(x_k)^\top W_kJ_g(x_k)\succeq0.
\]
This term acts as a curvature regularisation induced by the residual manifold itself. Consequently, the sliding formulation not only drives the trajectory toward the KKT manifold but also improves the conditioning of the structured correction system by reinforcing curvature in directions visible through the inequality Jacobian.

Moreover, The Fischer--Burmeister residual introduces a curvature term that is closely related to the Hessian structure appearing in logarithmic barrier methods. Near the residual manifold, where \(S^c_{i,k}\approx0\), one has
\[
W_{i,k}
=
\frac{S^c_{i,k}+\lambda_{i,k}}
{S^c_{i,k}-g_i(x_k)}
\approx
-\frac{\lambda_{i,k}}{g_i(x_k)}.
\]
Hence,
\[
J_g(x_k)^\top W_kJ_g(x_k)
\approx
\sum_{i=1}^{m}
\left(
-\frac{\lambda_{i,k}}{g_i(x_k)}
\right)
\nabla g_i(x_k)\nabla g_i(x_k)^\top .
\]

This term has the same rank-one curvature structure as the logarithmic barrier Hessian. Indeed, for the barrier function
\[
B_i(x)=-\log(-g_i(x)),
\qquad g_i(x)<0,
\]
Thus, the Fischer--Burmeister correction induces a barrier-like curvature regularisation in the reduced Hessian. This explains why \eqref{reduced} may remain better conditioned even when the multiplier variables have not yet reached their limiting KKT values. In this sense, the residual sliding manifold combines complementarity enforcement with a barrier-type curvature effect, without explicitly introducing a logarithmic barrier objective.

Furthermore, the proposed FB formulation preserves the convex structure of the original optimisation problem. As established in the following lemma, convex inequality constraints yield a convex regularised FB complementarity residual. Consequently, for convex optimisation problems with affine equality constraints, \(S(z)=0\) is equivalent to the KKT conditions and therefore guarantees convergence to the globally optimal solution.
\begin{lem}[\color{blue}Convexity of the FB Residual]
\label{lem:convexity_sc}
Consider
\[
S_i^c(x,\lambda_i)
=
-\lambda_i
+
g_i(x)
+
\sqrt{\lambda_i^2+g_i(x)^2+\rho_c},~
\rho_c>0.
\]
If \(g_i\) is twice continuously differentiable and convex, then
\(S_i^c(x,\lambda_i)\) is convex in \((x,\lambda_i)\).
\end{lem}
\begin{IEEEproof}
    The proof is provided in Appendix~\ref{appendix1}.
\end{IEEEproof}
\subsection{Equivalence Between the KKT Residual System and the KKT Conditions}
The following proposition  establishes the relationship between the proposed residual system and the KKT conditions  of the original constrained optimisation problem. The objective is to show that the zeros of the residual mapping are in one-to-one correspondence with the KKT points. Consequently, solving the nonlinear system \eqref{resi} is equivalent to satisfying the first-order necessary optimality conditions. The analysis first establishes the equivalence between the residual equation \(S(z)=0\) and the KKT conditions in \eqref{eq:kkt}. 
\begin{proposition}[\color{blue}Sliding Manifolds and KKT Conditions]\label{prop:kkt_equiv_s_zero}Let \(f\), \(g\), and \(h\) be continuously differentiable. Then
\[
S(z^\star)=0
\]
if and only if $z^\star=
(x^\star,\lambda^\star,\mu^\star)^T$
satisfies the Karush--Kuhn--Tucker conditions of
\eqref{eq:kkt}.

\end{proposition}

\begin{IEEEproof}
See the proof in Appendix \ref{appendix2}.
\end{IEEEproof}

Proposition~\ref{prop:kkt_equiv_s_zero} shows that the zero-residual manifold in \eqref{resi} coincides exactly with the KKT set of the original constrained optimisation problem. Consequently, the proposed algorithm can be interpreted as a residual-driven dynamical system whose objective is to steer the trajectory towards the manifold \eqref{resi}. Therefore,
\[
S(z_k)\rightarrow0
\]
is equivalent to convergence towards the KKT set, providing the theoretical justification for using the residual system as the sliding manifold.

\subsection{Residual Jacobian Inertia as a Local Optimality Certificate}
\label{sec:inertia_optimality}
The folowing theorem studies the Jacobian of the residual mapping at a regular KKT point. Under LICQ and SOSC, the Jacobian possesses the inertia property \eqref{inertia} which implies local nonsingularity and strong regularity of the KKT solution. 
\begin{theorem}[\color{blue} Jacobian Inertia and Local Optimality]
\label{thm:inertia_characterisation}
Let $z^*$ be a KKT point of problem \eqref{eq:problem}. Partition the inequality constraints into the strongly active and inactive sets
\[
\mathcal A
=
\{i:\;g_i(x^\star)=0,\ \lambda_i^\star>0\},
~
\mathcal I
=
\{i:\;g_i(x^\star)<0,\ \lambda_i^\star=0\}.
\]
Let
\[
r_A=|\mathcal A|,
\qquad
r_I=|\mathcal I|,
\qquad
m=r_A+r_I.
\]
Define
\[
H^\star
=
\nabla_{xx}^{2}
L(x^\star,\lambda^\star,\mu^\star),
~~
C
=
\begin{bmatrix}
J_A(x^\star)&
J_h(x^\star)
\end{bmatrix}^T,
\]
where \(J_A(x^\star)\) is the Jacobian of the strongly active inequality constraints and \(J_h(x^\star)\) is the Jacobian of the equality constraints. Let \(Z\) be a basis matrix for the null space of \(C\),
\[
CZ=0.
\]

Then the following statements are equivalent:

\begin{enumerate}
\item
The residual Jacobian evaluated at the KKT point satisfies
\[
\operatorname{In}(J^\star)
=
(n,m+q,0).
\]

\item
The KKT point satisfies LICQ,
\[
\operatorname{rank}(C)=r_A+q,
\]
and the second-order sufficient condition,
\[
Z^\top H^\star Z\succ0.
\]
\end{enumerate}

Consequently, the inertia condition $\operatorname{In}(J^\star)=(n,m+q,0)$ is equivalent to regular local optimality: the KKT point is nondegenerate, locally isolated, and satisfies the second-order sufficient optimality condition.
\end{theorem}

\begin{IEEEproof}  The proof is provided in Appendix~\ref{appendix3}.
\end{IEEEproof}
Proposition~\ref{prop:kkt_equiv_s_zero} and Theorem~\ref{thm:inertia_characterisation} jointly transform the constrained optimisation problem into a residual regulation problem. The former establishes that \(S(z)=0\) is equivalent to satisfaction of the KKT conditions, while the latter demonstrates that the inertia signature
\[
\operatorname{In}(J^\star)=(n,m+q,0)
\]
characterises regular local minimisers through LICQ and SOSC. Therefore, the construction of a stable update law that asymptotically drives \(S(z)\) to zero is sufficient to recover the desired solution. Any asymptotically stable equilibrium satisfying the above inertia condition necessarily corresponds to a nondegenerate and locally isolated local minimiser, thereby excluding degenerate stationary points and saddle-type KKT solutions.

\section{Projection-Free Sliding Solver for KKT Residual Systems (SM--KKT)}
\label{sec:proposed_algorithm}
This section presents the proposed projection-free sliding solver for the constrained optimisation problem introduced in Section~\ref{sec:problem_setting}. The central idea is to view constrained optimisation as a residual-manifold regulation problem. Rather than enforcing stationarity, feasibility, and complementarity through separate algorithmic mechanisms, the proposed approach constructs a unified KKT residual surface and drives the primal--dual trajectory toward its zero set. Since the residual vanishes if and only if the KKT conditions are satisfied, convergence of the residual dynamics implies convergence to the optimality manifold.  The proposed framework is inspired by sliding-mode control, where the objective is to drive the system trajectory toward a prescribed manifold and maintain it there despite modelling uncertainties. Here, the KKT residual manifold plays the role of the sliding surface. A residual reaching law is therefore designed directly in the primal--dual space to force the trajectory toward the zero-residual set while continuously reducing stationarity, feasibility, and complementarity errors.

The proposed method possesses several attractive properties. First, the constrained optimisation problem is reformulated as a unified residual system, enabling convergence to be analysed through residual decay. Second, the Fischer--Burmeister residual eliminates the need for explicit projection operators and yields a projection-free treatment of inequality constraints. Third, the update law combines smooth residual correction with saturated sliding terms, providing robustness while avoiding excessive oscillations near the solution. Finally, the resulting algorithm is computationally inexpensive, requiring only first-order information and avoiding Hessian evaluations, matrix factorisations, barrier schedules, and quadratic programming subproblems. Consequently, the computational cost per iteration remains low and scales favourably for large-scale constrained optimisation problems.

\subsection{Discrete-Time Implementation} \label{implemntation}

For numerical implementation, the optimisation dynamics are formulated directly in discrete time. Let
\[
z_k=
\begin{bmatrix}
x_k&
\lambda_k&
\mu_k
\end{bmatrix}^T,
\]
where \(x_k\in\mathbb{R}^{n}\) denotes the primal variable, \(\lambda_k\in\mathbb{R}^{m}\) the inequality multiplier, and \(\mu_k\in\mathbb{R}^{p}\) the equality multiplier.

The stationarity residual is computed as
\begin{equation}
\label{eq:discrete_Sx}
S_{k}^x
=
\nabla_x\mathcal{L}(x_k,\lambda_k,\mu_k).
\end{equation}

The inequality-related KKT conditions are represented through the regularised Fischer--Burmeister residual
\begin{equation}
\label{eq:discrete_Sc}
S_{k}^c
=
-\lambda_k
+
g(x_k)
+
\sqrt{\lambda_k^2+g(x_k)^2+\rho_c},
\end{equation}
where \(\rho_c>0\) is a small regularisation parameter and all operations are performed componentwise.

The equality-feasibility residual is
\begin{equation}
\label{eq:discrete_Sh}
S_{k}^h
=
h(x_k).
\end{equation}

The complete residual vector is therefore
\begin{equation}
\label{eq:discrete_residual}
S_k
=
\begin{bmatrix}
S_{k}^x&
S_{k}^c&
S_{k}^h
\end{bmatrix}^T.
\end{equation}

Let
\[
J_k
=
\frac{\partial S(z_k)}{\partial z}
\]
denote the Jacobian of the residual mapping evaluated at the current iterate. The desired residual variation is prescribed through the sliding reaching law
\begin{equation}
\label{eq:desired_residual}
r_k
=
-
\left[
Q S_k
+
\mathcal K
\operatorname{sat}
\!\left(
\frac{S_k}{\phi}
\right)
\right],
\end{equation}
where $r_k=
\begin{bmatrix}
r_k^x&
r_k^c&
r_k^h
\end{bmatrix}^T$ and 
\[
Q
=
\operatorname{diag}
(q_xI_n,q_cI_m,q_hI_p),~\mathcal K
=
\operatorname{diag}
(\varepsilon_xI_n,
\varepsilon_cI_m,
\varepsilon_hI_p),
\]
with positive gains
\[
q_x,q_c,q_h>0,
\qquad
\varepsilon_x,\varepsilon_c,\varepsilon_h>0,
\]
and boundary-layer parameters
\[
\phi_x,~\phi_c,~\phi_h>0.
\]

\begin{algorithm}[t]
\caption{Projection-Free Sliding KKT Residual Solver}
\label{alg:main_sliding_kkt}
\begin{algorithmic}[1]
\Require \(z_0=(x_0,\lambda_0,\mu_0)\), \(Q\succ0\), \(\mathcal K\succ0\), \(\phi>0\), \(\eta>0\), \(\rho_c>0\), \(\varepsilon_{\rm tol}>0\).
\Ensure Approximate KKT point \(z_k\).
\For{\(k=0,1,2,\ldots\)}
\State Evaluate \(S_k\) using \eqref{eq:discrete_residual}.
\If{\(\|S_k\|<\varepsilon_{\rm tol}\)}
    \State \Return \(z_k\).
\EndIf
\State Define
\[
r_k=
-\left[
QS_k+
\mathcal K\operatorname{sat}\left(\frac{S_k}{\phi}\right)
\right].
\]
\State Compute
\[
\Delta z_k=\textsc{StructuredCorrection}(z_k,r_k).
\]
\State Update
\[
z_{k+1}=z_k+\eta\Delta z_k.
\]
\EndFor
\end{algorithmic}
\end{algorithm}
Using the first-order taylor approximation
\begin{equation}
\label{eq:first_order_residual}
S(z_k+\Delta z_k)
=
S(z_k)
+
J_k\Delta z_k
+
\mathcal O(\|\Delta z_k\|^2),
\end{equation}
the correction direction is chosen so that the residual follows the prescribed sliding dynamics. This leads to the linearised residual equation
\begin{equation}
\label{eq:linear_residual_system}
J_k\Delta z_k
=
r_k, ~~~\Delta z_k
=
\begin{bmatrix}
\Delta x_k&
\Delta\lambda_k&
\Delta\mu_k
\end{bmatrix}^T.
\end{equation}

Direct inversion of \(J_k\) is avoided. Instead, the special
block structure of the residual Jacobian is exploited. This
structure is a direct consequence of the particular residual
construction adopted in this work. The stationarity residual
depends on the Lagrangian gradient, the complementarity
conditions are represented through the regularised
Fischer--Burmeister function, and the equality constraints are
introduced explicitly through the feasibility residual. As a
result, the Jacobian naturally assumes the structured KKT form
\begin{equation}
\label{eq:block_jacobian_discrete}
J_k
=
\begin{bmatrix}
H_k & J_g(x_k)^\top & J_h(x_k)^\top\\
D_k^gJ_g(x_k) & D_k^\lambda & 0\\
J_h(x_k) & 0 & 0
\end{bmatrix},
\end{equation}
where $J_g$ and   $J_h$  are given in \eqref{jacg}, respectively, and 
\begin{equation}
    H_k
=
\nabla_{xx}^{2}\mathcal L(x_k,\lambda_k,\mu_k),
\end{equation}

and
\begin{equation}\label{Dg}
    D_k^g
=
\operatorname{diag}
\left(
1+
\frac{g_i(x_k)}
{\sqrt{\lambda_{i,k}^{2}+g_i(x_k)^2+\rho_c}}
\right),
\end{equation} 
\begin{equation}\label{Dlam}
   D_k^\lambda
=
\operatorname{diag}
\left(
-1+
\frac{\lambda_{i,k}}
{\sqrt{\lambda_{i,k}^{2}+g_i(x_k)^2+\rho_c}}
\right). 
\end{equation}

which consists of a Hessian block, constraint Jacobians, and a
diagonal complementarity block. This structure is particularly
advantageous because the complementarity variables can be
eliminated analytically through block elimination, reducing the
original system of dimension \((n+m+p)\) to a saddle-point
system of dimension \((n+p)\). Consequently, the correction
direction can be computed without forming the inverse of the
full residual Jacobian. The computational cost is therefore
significantly reduced while preserving the exact solution of the
linearised residual equation. Moreover, the resulting reduced
system inherits the classical saddle-point structure of KKT
systems, allowing well-established numerical linear algebra
techniques to be employed efficiently and robustly. Furthermore, the same structure reveals the intrinsic primal--dual geometry of the optimisation problem, which is later used to establish the saddle nature of the residual
Jacobian and the local stability properties of the proposed
algorithm.

Therefore, the complementarity block yields
\begin{equation}
\label{eq:lambda_block_discrete}
D_k^gJ_g(x_k)\Delta x_k
+
D_k^\lambda\Delta\lambda_k
=
r_k^c.
\end{equation}

Assuming \(D_k^\lambda\) is nonsingular,
\begin{equation}
\label{eq:delta_lambda_discrete}
\Delta\lambda_k
=
[D_k^\lambda]^{-1}
\left(
r_k^c
-
D_k^gJ_g(x_k)\Delta x_k
\right).
\end{equation}

Substituting \eqref{eq:delta_lambda_discrete} into the stationarity
block produces the reduced saddle-point system
\begin{equation}
\label{eq:reduced_saddle_discrete}
\begin{bmatrix}
\bar H_k & J_h(x_k)^\top\\
J_h(x_k) & 0
\end{bmatrix}
\begin{bmatrix}
\Delta x_k\\
\Delta\mu_k
\end{bmatrix}
=
\begin{bmatrix}
\bar r_k^x\\
r_k^h
\end{bmatrix},
\end{equation}
where
\begin{equation}
\label{eq:Hbar_discrete}
\bar H_k
=
H_k
-
J_g(x_k)^\top
[D_k^\lambda]^{-1}
D_k^g
J_g(x_k),
\end{equation}
and
\begin{equation}
\label{eq:rbar_discrete}
\bar r_k^x
=
r_k^x
-
J_g(x_k)^\top
[D_k^\lambda]^{-1}
r_k^c.
\end{equation}

After solving \eqref{eq:reduced_saddle_discrete} for
\((\Delta x_k,\Delta\mu_k)\), the inequality multiplier
correction is recovered from \eqref{eq:delta_lambda_discrete}. 

\begin{algorithm}[t]
\caption{\textsc{StructuredCorrection}}
\label{alg:structured_correction}
\begin{algorithmic}[1]
\Require Current point \(z_k=(x_k,\lambda_k,\mu_k)\), desired residual variation \(r_k\).
\Ensure Correction \(\Delta z_k=(\Delta x_k,\Delta\lambda_k,\Delta\mu_k)\).

\State Partition
\[
r_k=
\begin{bmatrix}
r_k^x&
r_k^c&
r_k^h
\end{bmatrix}^T.
\]

\State Evaluate \(J_g(x_k)\), \(J_h(x_k)\), and \(H_k=\nabla^2_{xx}\mathcal L(x_k,\lambda_k,\mu_k)\).

\State Form \(D_k^g\) and \(D_k^\lambda\) using \eqref{Dg} and \eqref{Dlam}.

\State Compute
\[
\bar H_k
=
H_k
-
J_g(x_k)^\top
[D_k^\lambda]^{-1}
D_k^g
J_g(x_k),
\]
\[
\bar r_k^x
=
r_k^x
-
J_g(x_k)^\top
[D_k^\lambda]^{-1}
r_k^c.
\]

\State Solve
\[
\begin{bmatrix}
\bar H_k & J_h(x_k)^\top\\
J_h(x_k) & 0
\end{bmatrix}
\begin{bmatrix}
\Delta x_k\\
\Delta\mu_k
\end{bmatrix}
=
\begin{bmatrix}
\bar r_k^x\\
r_k^h
\end{bmatrix}.
\]

\State Recover
\[
\Delta\lambda_k
=
[D_k^\lambda]^{-1}
\left(
r_k^c
-
D_k^gJ_g(x_k)\Delta x_k
\right).
\]

\State Set
\[
\Delta z_k=
\begin{bmatrix}
\Delta x_k&
\Delta\lambda_k&
\Delta\mu_k
\end{bmatrix}^T.
\]
\State \Return \(\Delta z_k\).
\end{algorithmic}
\end{algorithm}

Finally, the discrete-time optimisation algorithm is therefore
\begin{equation}
\label{eq:discrete_algorithm}
z_{k+1}
=
z_k
+
\eta\,\Delta z_k,
\end{equation}
where \(\eta>0\) is a damping factor. Equivalently,
\begin{align}
x_{k+1}
&=
x_k+\eta\Delta x_k,
\\
\lambda_{k+1}
&=
\lambda_k+\eta\Delta\lambda_k,
\\
\mu_{k+1}
&=
\mu_k+\eta\Delta\mu_k.
\end{align}

The iterations terminate when
\begin{equation}
\label{eq:termination}
\|S_k\|
<
\varepsilon_{\mathrm{tol}},
\end{equation}
for a prescribed tolerance $\varepsilon_{\mathrm{tol}}>0.$
 \subsection{Convergence Analysis of the Proposed Algorithm}
\label{subsec:convergence_analysis}

The proposed optimisation framework transforms the constrained
optimisation problem into a residual regulation problem. By
Proposition~\ref{prop:kkt_equiv_s_zero}, the condition
\[
S(z)=0
\]
is equivalent to satisfaction of the KKT conditions. Furthermore,
Theorem~\ref{thm:inertia_characterisation} establishes that the
inertia condition \eqref{inertia} is equivalent to regular local optimality through LICQ and the second-order sufficient condition. Consequently, the convergence
analysis reduces to studying whether the proposed update law
asymptotically drives the residual trajectory toward the manifold
\(S(z)=0\) while preserving the local saddle geometry induced by
the residual Jacobian.
\begin{theorem}[\color{blue}Boundary-Layer Convergence]
\label{thm:residual_stability} Let \(z^\star=(x^\star,\lambda^\star,\mu^\star)\) be a regular local minimiser satisfying the the Proposition \ref{prop:kkt_equiv_s_zero} and Theorem~\ref{thm:inertia_characterisation}. Assume that \(J(z)\) remains nonsingular along the generated trajectory given by the updated law  and the structured correction in Algorithm \ref{alg:main_sliding_kkt} and Algorithm \ref{alg:structured_correction}, respectivelly.
If, for each residual component remains inside the ball given by:
\begin{equation}
   \delta_i=\eta\varepsilon_i -\bar R_{k}^{i}>0, 
\end{equation}
where \(\bar R_{k}^{i}\) bounds the corresponding Taylor remainder, then every residual component reaches the boundary layer
\[
\mathcal B_\phi
=
\{S:\ |S^{i}|\le\phi_i,\ i=1,\ldots,n+m+p\}
\]
in finite time $N_i$, where 
\begin{equation}
    N_i
\le
\left\lceil
\frac{|S_{0}^{i}|-\phi_i}{\delta_i}
\right\rceil.
\end{equation}

Moreover, inside \(\mathcal B_\phi\), the residual dynamics are locally asymptotically stable if
\[
\rho
\left(
I-\eta(Q+\mathcal K\Phi^{-1})
\right)
<1,
~~
\Phi=\operatorname{diag}(\phi).
\]
For diagonal \(Q\), \(\mathcal K\), and \(\Phi\), this is equivalent to
\begin{equation}
   0<
\eta
\left(
q_i+\frac{\varepsilon_i}{\phi_i}
\right)
<2,
~~
i=1,\ldots,n+m+p. 
\end{equation}
Consequently, the residual converges to an invariant neighbourhood of the origin whose size is determined by the boundary-layer widths and the Taylor approximation error. Since \(J^\star\) is nonsingular, this residual neighbourhood induces a corresponding primal--dual neighbourhood of the nondegenerate locally isolated KKT point \(z^\star\).
\end{theorem}

\begin{IEEEproof} Since \(S\) is continuously differentiable, Taylor's theorem gives
\[
S(z_{k+1})
=
S(z_k)
+
J_k(z_{k+1}-z_k)
+
R_k,~R_k
=
\mathcal O(\|\eta\Delta z_k\|^2).
\]

Using
\[
z_{k+1}=z_k+\eta\Delta z_k,
\]
we obtain
\[
S_{k+1}
=
S_k+\eta J_k\Delta z_k+R_k.
\]
Because the structured correction is computed from
\[
J_k\Delta z_k=r_k,
\]
it follows that
\[
S_{k+1}
=
S_k+\eta r_k+R_k.
\]
Substituting the reaching law gives
\[
S_{k+1}
=
S_k
-
\eta
\left[
QS_k+
\mathcal K
\operatorname{sat}
\left(
\frac{S_k}{\phi}
\right)
\right]
+
R_k.
\]

For the \(i\)-th residual component,
\[
S_{k+1}^{i}
=
S_{k}^{i}
-
\eta q_iS_{k}^{i}
-
\eta\varepsilon_i
\operatorname{sat}
\left(
\frac{S_{k}^{i}}{\phi_i}
\right)
+
R_{k}^{i}.
\]

Outside the boundary layer,
\[
|S_{k}^{i}|>\phi_i \implies \operatorname{sat}
\left(
\frac{S_{k}^{i}}{\phi_i}
\right)
=
\operatorname{sgn}(S_{k}^{i}).
\]

Therefore,
\[
S_{k+1}^{i}
=
(1-\eta q_i)S_{k}^{i}
-
\eta\varepsilon_i\operatorname{sgn}(S_{k}^{i})
+
R_{k}^{i}.
\]
Multiplying by \(\operatorname{sgn}(S_{k}^{i})\) yields
\[
\operatorname{sgn}(S_{k}^{i})S_{k+1}^{i}
=
(1-\eta q_i)|S_{k}^{i}|
-
\eta\varepsilon_i
+
\operatorname{sgn}(S_{k}^{i})R_{k}^{i}.
\]
Using
\[
|\operatorname{sgn}(S_{k}^{i})R_{k}^{i}|
\le
|R_{k}^{i}|
\le
\bar R_{k}^{i},
\]
we obtain
\[
|S_{k+1}^{i}|
\le
(1-\eta q_i)|S_{k}^{i}|
-
\eta\varepsilon_i
+
\bar R_{k}^{i}.
\]
Since \(q_i>0\), this implies
\[
|S_{k+1}^{i}|
\le
|S_{k}^{i}|
-
\left(
\eta\varepsilon_i-\bar R_{k}^{i}
\right).
\]
If
\[
\eta\varepsilon_i>\bar R_{k}^{i},
\]
then \(|S_{k}^{i}|\) decreases by at least
\[
\delta_i
=
\eta\varepsilon_i-\bar R_{k}^{i}>0
\]
at every step outside the boundary layer. Hence the number of steps required to reach
\[
|S^{i}|\le\phi_i
\]
is bounded by
\[
N_i
\le
\left\lceil
\frac{|S_{0}^{i}|-\phi_i}{\delta_i}
\right\rceil.
\]
Thus each residual component reaches the boundary layer in finite time. Once inside the boundary layer $|S_{k}^{i}|\le \phi_i$, the saturation function satisfies
\[
\operatorname{sat}
\left(
\frac{S_{k}^{i}}{\phi_i}
\right)
=
\frac{S_{k}^{i}}{\phi_i}.
\]
In vector form,
\[
\operatorname{sat}
\left(
\frac{S_k}{\phi}
\right)
=
\Phi^{-1}S_k,~~
\Phi=\operatorname{diag}(\phi).
\]
Therefore, the residual dynamics become
\[
S_{k+1}
=
\left[
I-\eta(Q+\mathcal K\Phi^{-1})
\right]S_k
+
R_k.
\]
Define
\[
M_\eta
=
I-\eta(Q+\mathcal K\Phi^{-1}).
\]
If
\[
\rho(M_\eta)<1,
\]
then \(M_\eta\) is Schur stable. Hence there exists a norm \(\|\cdot\|_\star\) and a constant \(\gamma\in(0,1)\) such that
\[
\|M_\eta S\|_\star
\le
\gamma\|S\|_\star \implies \|S_{k+1}\|_\star
\le
\gamma\|S_k\|_\star
+
\|R_k\|_\star.
\]

If the Taylor remainder is bounded in the layer by
\[
\|R_k\|_\star\le \bar R_\star,
\]
then repeated application gives
\[
\|S_k\|_\star
\le
\gamma^{k-k_0}\|S_{k_0}\|_\star
+
\frac{1-\gamma^{k-k_0}}{1-\gamma}\bar R_\star.
\]
Consequently,
\[
\limsup_{k\to\infty}
\|S_k\|_\star
\le
\frac{\bar R_\star}{1-\gamma}.
\]
Therefore, the residual remains ultimately bounded in a neighbourhood of the origin whose radius is determined by the Taylor remainder and the contraction margin \(1-\gamma\).

For diagonal \(Q\), \(\mathcal K\), and \(\Phi\), the matrix \(M_\eta\) is diagonal, and its eigenvalues are
\[
\lambda_i(M_\eta)
=
1-\eta
\left(
q_i+\frac{\varepsilon_i}{\phi_i}
\right).
\]
Thus the Schur condition is equivalent to
\[
\left|
1-\eta
\left(
q_i+\frac{\varepsilon_i}{\phi_i}
\right)
\right|
<1,
\]
which gives the explicit step-size condition
\[
0<
\eta
\left(
q_i+\frac{\varepsilon_i}{\phi_i}
\right)
<2.
\]

Finally, by Theorem~\ref{thm:inertia_characterisation}, the residual Jacobian \(J^\star\) is nonsingular at the regular local minimiser. Hence, by the inverse function theorem, \(S\) is locally invertible near \(z^\star\). Therefore, ultimate boundedness of \(S_k\) implies ultimate boundedness of \(z_k\) around \(z^\star\). More precisely, there exists \(c_z>0\) such that
\[
\|z_k-z^\star\|
\le
c_z\|S_k\|
\]
for all \(z_k\) sufficiently close to \(z^\star\). Thus the residual boundary-layer accuracy directly induces a corresponding primal--dual accuracy around the nondegenerate locally isolated KKT point.
\end{IEEEproof}

The convergence properties of the proposed framework are established through the residual manifold and its associated Jacobian structure. The algorithm is designed to drive the residual vector \(S(z)\), containing stationarity, feasibility, and complementarity conditions, towards zero through a prescribed sliding-mode reaching law. However, unlike conventional primal--dual approaches, convergence is not characterised solely through first-order optimality conditions.

A key theoretical result of the paper is that the residual manifold is constructed such that

\[
S(z)=0
\Longleftrightarrow
\text{KKT conditions}.
\]

Furthermore, under LICQ, the Jacobian of the residual system possesses a specific inertia structure if and only if the second-order sufficient optimality condition holds. Consequently, the residual Jacobian is nonsingular precisely at locally optimal KKT points satisfying LICQ and SOSC. This result is significant because many primal--dual, complementarity, and gradient-based methods may converge to any KKT point, including saddle points or degenerate stationary solutions. In contrast, the proposed framework embeds second-order information directly within the residual dynamics through the Jacobian structure, thereby linking local convergence to genuine locally optimal solutions.

An important consequence is the separation between optimality characterisation and algorithm design. Once the residual manifold and its Jacobian properties have been established, the optimisation problem is transformed into a residual regulation problem. The role of the numerical algorithm is then simply to drive the residual vector \(S(z)\) towards zero. Any update law capable of stabilising the residual dynamics may, in principle, be employed. The proposed sliding-mode reaching law is adopted because it provides a direct and robust mechanism for regulating the residual system and enforcing convergence towards the KKT manifold. Therefore, the optimality guarantees originate primarily from the residual manifold and its Jacobian structure rather than from a particular optimisation update rule.

For convex optimisation problems with convex objective and inequality constraints and affine equality constraints, the KKT conditions are both necessary and sufficient for global optimality. Therefore,

\[
S(z)=0
\Longleftrightarrow
\text{KKT conditions}
\Longleftrightarrow
\text{global optimum},
\]

which implies that driving the sliding manifold to zero is equivalent to converging to the optimal solution.

From a computational perspective, the dominant cost per iteration arises from solving the structured linear system associated with the residual Jacobian. This complexity is comparable to Newton-type primal--dual methods. However, the structured block formulation exploits the sparsity and partitioned structure of the KKT system while avoiding explicit multiplier projections and repeated active-set updates. The resulting framework therefore combines a theoretically grounded residual formulation with an efficient computational structure that is amenable to larger-scale implementations.

\section{Numerical Validation}
\label{sec:numerical_validation}
The numerical study evaluates the central claim of the proposed projection-free SM–KKT framework: that direct regulation of a unified KKT residual can simultaneously address stationarity, primal and dual feasibility, and complementarity, thereby recovering a valid constrained optimum rather than merely reducing the objective function. Across the four benchmark problems, SM–KKT demonstrates the strongest overall performance, achieving the smallest final KKT residuals, accurately recovering the reference solutions, and reaching the prescribed residual boundary layer earlier than the projected primal–dual, augmented-Lagrangian, quadratic-penalty, and log-barrier methods. The benchmarks encompass distinct optimisation challenges, including nonconvexity, narrow Rosenbrock valleys, annular feasible regions with active-set transitions, and nonlinear equality manifolds with oscillatory objectives. All methods are evaluated using common initialisation, stopping conditions, and residual assessment criteria. Together, these experiments provide evidence of the accuracy and computational efficiency of SM–KKT across a diverse range of nonlinear constrained optimisation problems.
\subsection{Benchmark Problems}
The proposed projection-free SM--KKT solver is evaluated on four nonlinear constrained optimisation problems selected to exhibit different objective landscapes, feasible-set geometries, and constraint interactions. The benchmark suite includes nonconvex objective functions, nonlinear equality and inequality constraints, annular feasible regions, and highly nonlinear trigonometric constraints. These examples are intentionally chosen to assess the ability of the proposed residual-driven framework to simultaneously enforce stationarity, feasibility, and complementarity while converging towards the optimal solution.

The following benchmark problems are considered:

\begin{enumerate}
    \item Complex nonlinear optimisation with mixed equality and inequality constraints.
    \item Rosenbrock-type constrained optimisation problem.
    \item Banana-shaped objective with annular feasible region.
    \item Trigonometric nonlinear constrained optimisation problem.
\end{enumerate}

For each benchmark, the numerical experiments evaluate the convergence of the proposed KKT residual manifold, the evolution of stationarity, feasibility, and complementarity measures, the reduction of the objective error \( |f(x_k)-f^\star| \), and the recovery of the optimal solution. In addition, the first iteration at which the residual enters the prescribed sliding boundary layer is reported to quantify the convergence speed of the residual dynamics. The benchmark set therefore provides a comprehensive assessment of convergence behaviour, constraint satisfaction, trajectory evolution, and computational efficiency across a diverse range of constrained optimisation scenarios.

\begin{table}[t]
\centering
\rowcolors{2}{gray!15}{white}
\caption{Summary of the four nonlinear constrained optimisation benchmark problems, including their objective structure and associated equality and inequality constraints.}
\begin{tabular}{lcccc}
\toprule
Problem & Objective & Equality & Inequality & Type\\
\midrule
P1 & Nonlinear & Yes & Yes & Nonconvex\\
P2 & Rosenbrock & Yes & Yes & Nonconvex\\
P3 & Banana & No & Yes & Annular feasible set\\
P4 & Trigonometric & Yes & Yes & Highly nonlinear\\
\bottomrule
\end{tabular}
\end{table}
\subsection{Compared Algorithms}

The proposed method is compared against four representative constrained optimisation algorithms commonly used in nonlinear programming to solve optimisation problem (1) in the paper.

\begin{enumerate}
\item Proposed SM--KKT Solver.
\item Projected Primal--Dual Method.
\item Augmented Lagrangian Method.
\item Quadratic Penalty Method.
\item Log-Barrier Penalty Method.
\end{enumerate}

The projected primal--dual method represents a classical projection-based strategy for enforcing dual feasibility. The augmented Lagrangian method combines multiplier updates with quadratic penalties and is widely regarded as a robust benchmark for constrained optimisation. The quadratic penalty method represents a purely penalty-based approach, while the log-barrier method provides an interior-point baseline based on barrier curvature. These methods cover a broad spectrum of constrained optimisation techniques and therefore provide a meaningful basis for comparison.

The selected algorithms also provide a direct comparison with the main features of the proposed framework. In particular, they represent projection-based, penalty-based, barrier-based, and primal--dual optimisation paradigms, whereas the proposed SM--KKT solver operates directly on the KKT residual manifold through a projection-free residual regulation mechanism. This comparison therefore highlights the effectiveness of the proposed residual-driven formulation across several widely used classes of constrained optimisation methods.

For all simulations, the same stopping tolerance, residual definitions, and performance metrics are used to ensure a consistent comparison across the selected algorithms. In addition to objective convergence, the algorithms are evaluated using stationarity, feasibility, complementarity, final KKT residual, computational time, and the iteration at which the residual first enters the prescribed sliding boundary layer. Unless otherwise stated, the algorithmic parameters are fixed as reported in Table~\ref{tab:simulation_parameters}. Problem-specific step sizes are only adjusted when required to maintain numerical stability for highly nonlinear benchmarks.

\begin{table}[thpb]
\centering
\rowcolors{2}{gray!15}{white}
\caption{Common numerical parameters and algorithmic settings used for all
benchmark problems and comparison methods.}
\label{tab:simulation_parameters}
\resizebox{\columnwidth}{!}{
\begin{tabular}{lclc}
\hline
Parameter & Value & Parameter & Value \\
\hline
Maximum iterations
& \(2500\)
& Stopping tolerance
& \(10^{-6}\) \\
FB regularisation \(\rho_c\)
& \(10^{-18}\)
& Residual step size \(\eta\)
& \(0.8\) \\
Sliding gains \(q_x,q_c,q_h\)
& \(2,2,2\)
& Saturation gains \(\varepsilon_x,\varepsilon_c,\varepsilon_h\)
& \(10^{-2}\) \\
Boundary layers \(\phi_x,\phi_c,\phi_h\)
& \(10^{-3}\)
& Maximum correction step
& \(0.12\) \\

Penalty parameter \(\rho\)
& \(1.0\)
& Log-barrier parameter \(\tau\)
& \(10^{-2}\) \\
\hline
\end{tabular}}
\end{table}

\subsection{Performance Metrics}

The performance of each algorithm is evaluated using the following metrics:

\begin{itemize}
\item Objective function value \(f(x)\).
\item Stationarity residual $\|\nabla_x\mathcal L(x,\lambda,\mu)\|$.
\item Equality constraint violation $\|h(x)\|.$
\item Complementarity residual $\|\lambda\odot g(x)\|.$
\item Overall KKT residual $\|S(x,\lambda,\mu)\|.$
\item Path length.
\item Distance to the reference solution.
\item Number of iterations.
\end{itemize}

These metrics quantify both optimality and feasibility while also characterising the efficiency of the optimisation trajectory.

\subsection{\color{blue}Problem I: Complex Nonlinear Mixed Constraints}
\label{subsec:problem_complex_nonlinear}

The first benchmark is a two-dimensional nonlinear constrained optimisation problem involving a nonconvex objective, three nonlinear inequality constraints, and two nonlinear equality constraints. This example is used to assess the ability of the proposed solver to regulate stationarity, feasibility, and complementarity in a strongly coupled KKT residual system. The problem is defined as
\[
\begin{aligned}
f(x)
=&\;0.8(x_1-0.9)^2
+1.4(x_2-0.4)^2
+0.15x_1^4
+0.10x_2^4  \\
&+0.20x_1x_2
+0.05\sin(2x_1+x_2).
\end{aligned}
\]

The inequality constraints ($g(x)\leq0$) are
\begin{align*}
g_1(x)&=x_1^2+0.6x_2^2-1.20,
\\
g_2(x)&=-x_1-0.75,
\\
g_3(x)&=0.35x_1+x_2-1.10.
\end{align*}

The equality constraints $(h(x)=0)$ are
\begin{align*}
   h_1(x)&=x_1+x_2-0.85,
\\
h_2(x)&=0.20x_1^2+x_2-0.65. 
\end{align*}

The initial condition used for all algorithms is
\[
x_0=
\begin{bmatrix}
-1.35\\
1.45
\end{bmatrix},
\qquad
\lambda_0=
\begin{bmatrix}
0.2\\
0.2\\
0.2
\end{bmatrix},
\qquad
\mu_0=
\begin{bmatrix}
0\\
0
\end{bmatrix}.
\]
The regularisation parameter in the Fischer--Burmeister residual is
\[
\rho_c=10^{-8}.
\]

This benchmark is particularly demanding because the feasible set is defined by the intersection of curved inequality and equality manifolds, while the sinusoidal term in the objective introduces additional nonconvexity. Therefore, convergence requires simultaneous regulation of the primal variables, inequality multipliers, and equality multipliers.

\subsubsection{Results and Discussion}

The numerical results for the complex nonlinear mixed-constraint problem are
summarised in Table~\ref{tab:problem_complex_results}. The three-dimensional
objective surface and optimisation trajectories are shown in
Fig.~\ref{fig:problem1_3d}, while the corresponding two-dimensional
constraint geometry and optimisation trajectories are shown in
Fig.~\ref{fig:problem1_trajectory}. The KKT residual histories and convergence
metrics are presented in Figs.~\ref{fig:problem1_residuals}
and~\ref{fig:problem1_metrics}, respectively.

As illustrated in Fig.~\ref{fig:problem1_3d}, the competing algorithms evolve
over a nonconvex objective landscape subject to coupled equality and inequality
constraints. The proposed SM--KKT solver follows a comparatively direct
trajectory towards the reference constrained optimum.
\begin{figure}[thpb]
\centering
\includegraphics[width=\linewidth]{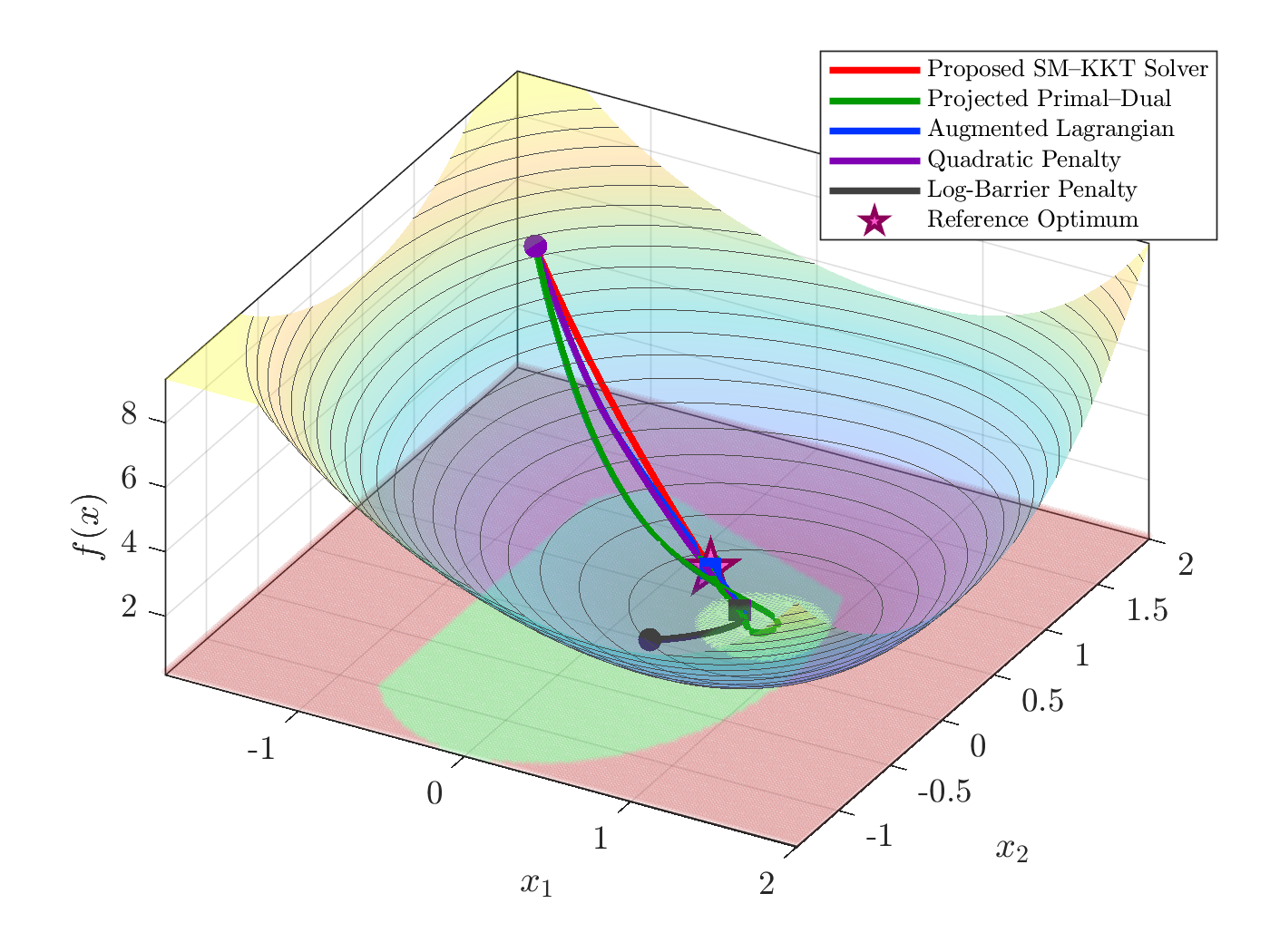}
\caption{Problem I: three-dimensional nonconvex objective surface and
optimisation trajectories. The proposed SM--KKT solver follows a direct
and efficient path towards the reference constrained optimum while
navigating the coupled equality and inequality constraints.}
\label{fig:problem1_3d}
\end{figure}
The corresponding two-dimensional representation in
Fig.~\ref{fig:problem1_trajectory} makes the constraint geometry more explicit.
The objective contours, feasible and infeasible regions, and equality and
inequality boundaries show that convergence requires simultaneous navigation
of the objective landscape and the coupled constraint geometry. In this view,
the SM--KKT trajectory progresses directly towards the reference constrained
optimum while approaching the admissible constraint manifold.

\begin{figure}[thpb]
\centering
\includegraphics[width=\linewidth]{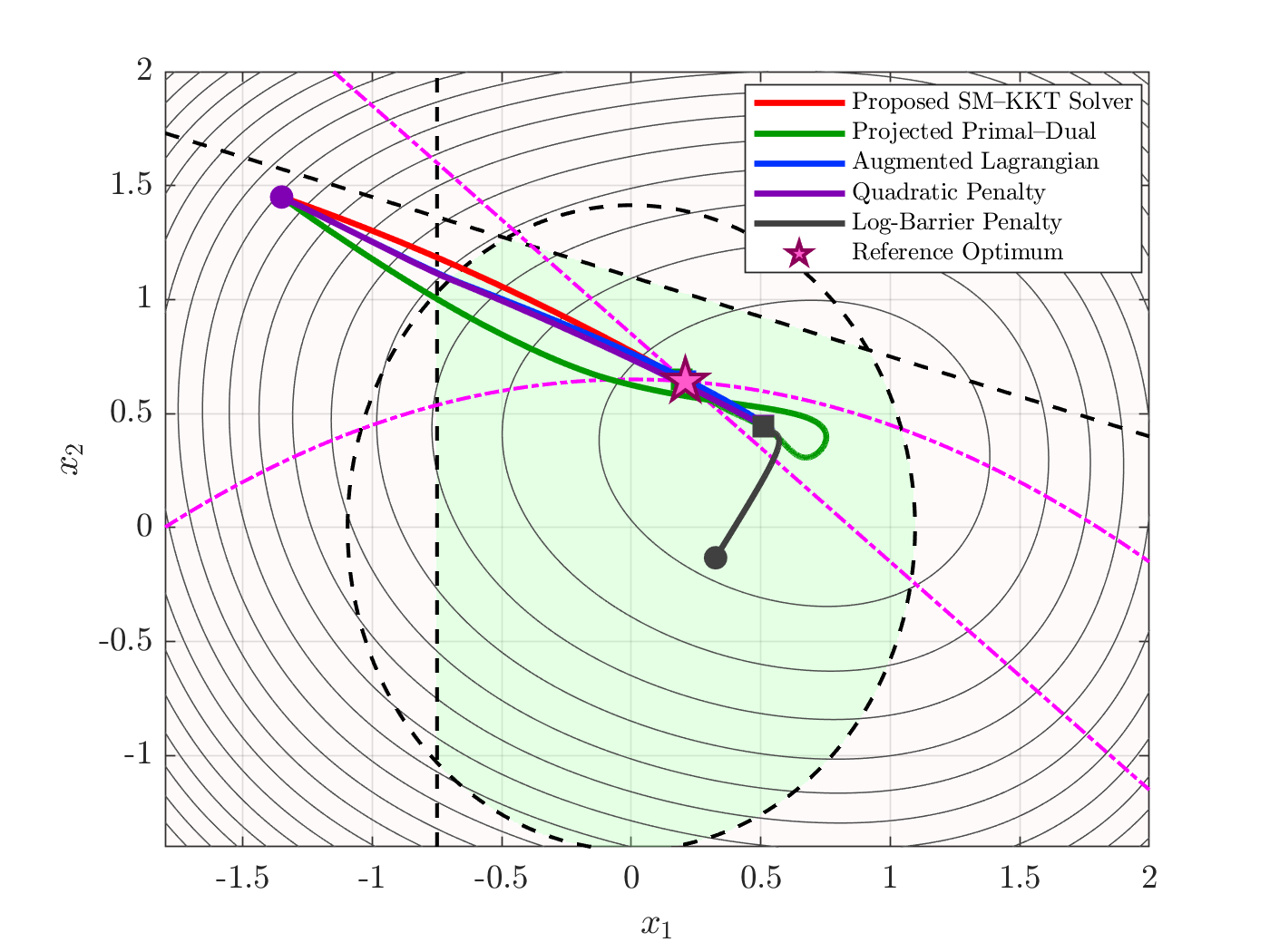}
\caption{Problem I: objective contours (\textcolor{darkgray}{dark grey}), feasible region (\textcolor{green!60!black}{green}),
infeasible region (\textcolor{pink!85!black}{light pink}), inequality constraint boundaries (black dashed),
equality constraints (\textcolor{magenta}{magenta dashed}), and optimisation trajectories.
The proposed SM--KKT solver follows a direct and efficient path towards
the reference constrained optimum while navigating the coupled nonlinear
constraint geometry.}
\label{fig:problem1_trajectory}
\end{figure}

Quantitatively, Table~\ref{tab:problem_complex_results} shows that SM--KKT is
the only method that reaches the prescribed KKT tolerance within the maximum
iteration budget. It converges in \(444\) iterations and achieves a final KKT residual of
\[
\|S\| = 9.022\times10^{-7}.
\]
Furthermore, the residual enters the prescribed sliding boundary layer after
\(390\) iterations. At termination, the stationarity, equality-feasibility,
and complementarity residuals are
\[
\|S^x\| = 9.02\times10^{-7},~
\|S^h\| = 1.72\times10^{-12}, ~
\|S^c\| = 2.39\times10^{-9},
\]
respectively. These values show that stationarity, equality feasibility, and
complementarity are regulated simultaneously rather than sequentially. Since
the terminal overall residual is essentially determined by
\(\|S^x\|\), while \(\|S^h\|\) and \(\|S^c\|\) are substantially smaller,
the final accuracy is primarily limited by stationarity rather than by
residual feasibility or complementarity errors. This behaviour is also evident in Fig.~\ref{fig:problem1_residuals}. The
SM--KKT residual components decrease jointly towards the prescribed tolerance, whereas the competing methods exhibit slower convergence, stagnation, or persistent residual errors.

\begin{figure*}[t]
\centering
\includegraphics[width=\linewidth]{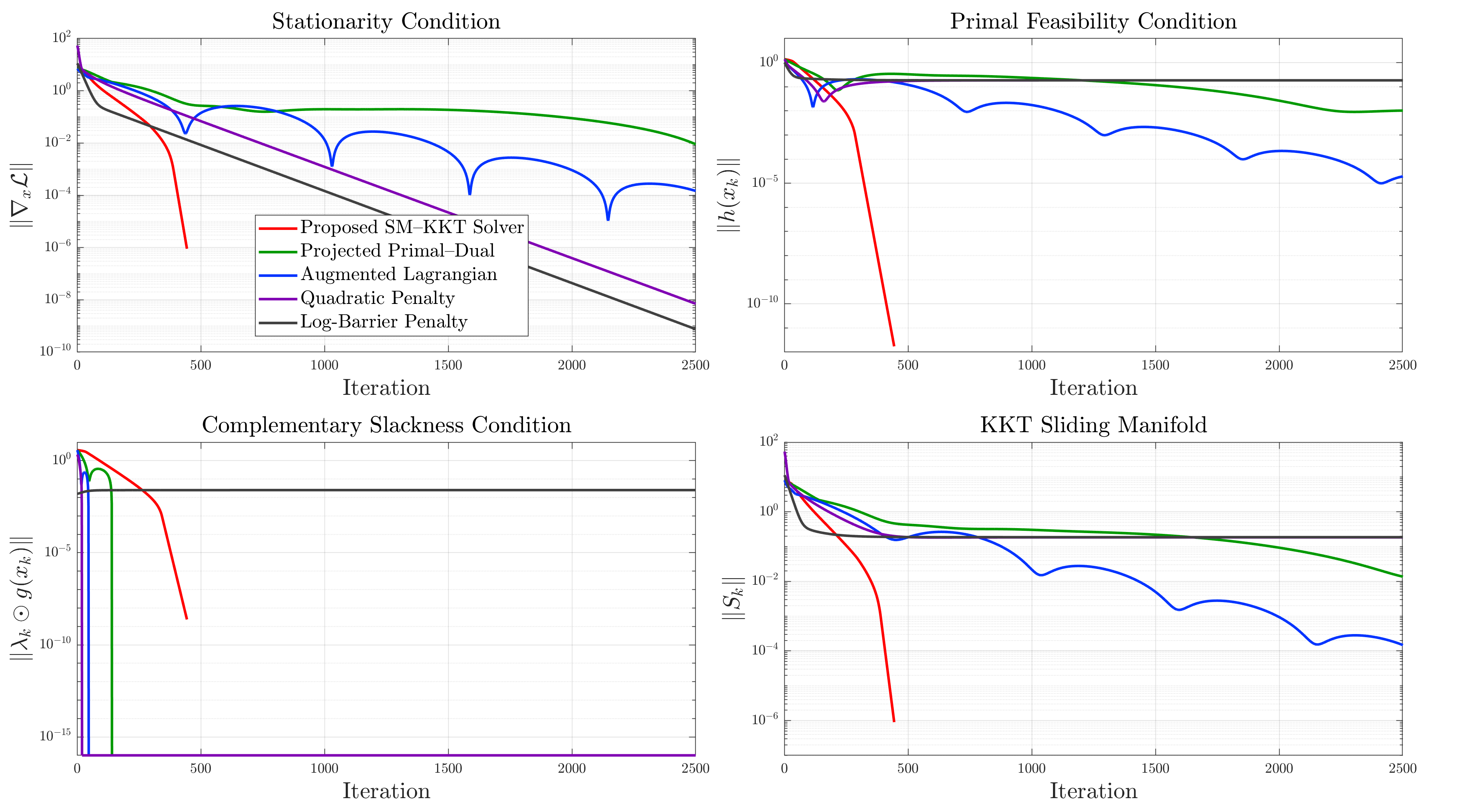}
\caption{Problem I: comparative evolution of the stationarity,
equality-feasibility, complementarity, and overall KKT residuals.
The proposed SM--KKT solver drives all residual components simultaneously
below the prescribed tolerance, whereas the competing methods exhibit
slower decay, stagnation, or persistent constraint violations.}
\label{fig:problem1_residuals}
\end{figure*}

Among the competing methods, the Augmented Lagrangian method provides the
closest result. As reported in Table~\ref{tab:problem_complex_results}, it
achieves a final KKT residual of $\|S\| = 1.469\times10^{-4},$ and enters the prescribed boundary layer after \(1993\) iterations. Although
its final objective error is only \(4.0\times10^{-5}\), its terminal KKT
residual remains more than two orders of magnitude larger than that obtained
by SM--KKT. This distinction is important because a small objective error
alone does not establish comparable stationarity, feasibility, and
complementarity accuracy.

The projected primal--dual method terminates with $\|S\| = 1.362\times10^{-2},$ and does not reach the prescribed convergence criterion within the maximum
iteration budget. As shown in Fig.~\ref{fig:problem1_residuals}, substantial
residual errors remain at termination. Table~\ref{tab:problem_complex_results}
also reports an objective error of \(2.182\times10^{-2}\), considerably larger
than those obtained by SM--KKT and the Augmented Lagrangian method. Its
cumulative primal trajectory length is \(3.133\), which is also the largest
among the compared methods. The quadratic-penalty and log-barrier methods exhibit still larger terminal
KKT residuals of $1.843\times10^{-1}
$ and $1.867\times10^{-1},$ respectively. Their corresponding objective errors are \(3.164\times10^{-1}\) and \(3.171\times10^{-1}\). Together with the residual
histories in Fig.~\ref{fig:problem1_residuals}, these results show that
objective reduction along an optimisation trajectory should not by itself be
interpreted as convergence to the constrained optimum; the complete KKT
residual must also be driven to the required accuracy.

Figure~\ref{fig:problem1_metrics} provides a complementary comparison through
the multiplier errors, objective error, and cumulative primal trajectory
length. The SM--KKT solver combines accurate multiplier recovery and objective
convergence with a cumulative primal path length of \(1.759\). By comparison,
the projected primal--dual and Augmented Lagrangian methods produce path
lengths of \(3.133\) and \(2.643\), respectively. The quadratic-penalty method
produces a path length of \(2.112\). These results are consistent with the
comparatively direct SM--KKT trajectory observed in
Figs.~\ref{fig:problem1_3d} and~\ref{fig:problem1_trajectory}.
\begin{figure*}[t]
\centering
\includegraphics[width=\linewidth]{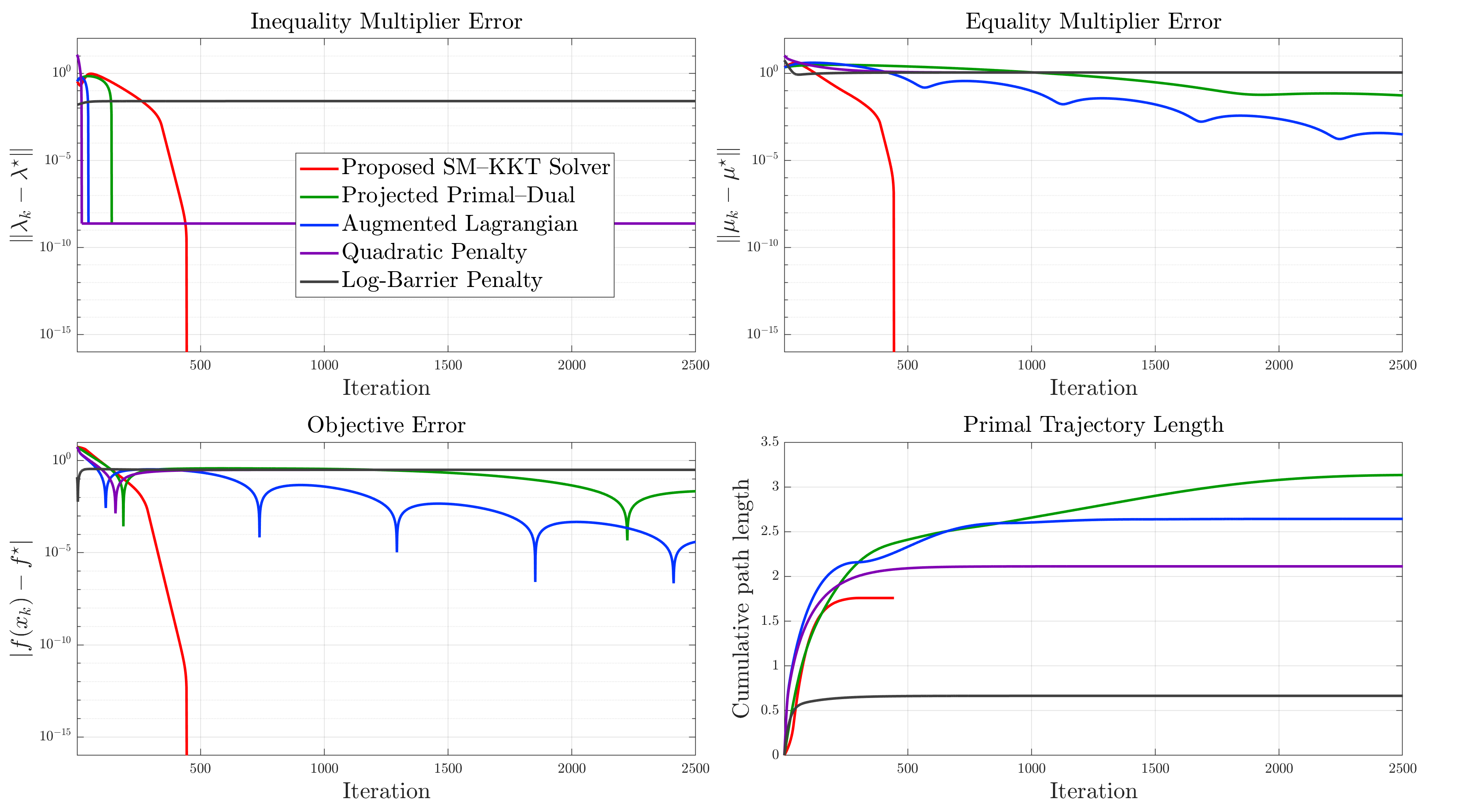}
\caption{Problem I: evolution of multiplier errors, objective error,
and cumulative primal trajectory length. The proposed SM--KKT solver
achieves accurate multiplier recovery and objective convergence while
following a direct and efficient path towards the constrained optimum,
whereas the competing methods exhibit slower or less accurate convergence.}
\label{fig:problem1_metrics}
\end{figure*} It is important, however, not to interpret path length independently of
terminal accuracy. The log-barrier method has the shortest reported path,
\(0.665\), compared with \(1.759\) for SM--KKT. However, it terminates with
an objective error of \(3.171\times10^{-1}\) and a KKT residual of
\(1.867\times10^{-1}\). Hence, its shorter trajectory does not represent a
more efficient route to the constrained optimum. A meaningful comparison of
trajectory efficiency must therefore be considered jointly with the attained
KKT and objective accuracies.

Taken together, Figs.~\ref{fig:problem1_3d},
\ref{fig:problem1_trajectory}, \ref{fig:problem1_residuals},
and~\ref{fig:problem1_metrics}, together with
Table~\ref{tab:problem_complex_results}, show that the principal advantage of
SM--KKT in this benchmark is not simply objective reduction or a short primal
trajectory. Rather, the method combines a direct trajectory towards the
reference solution with simultaneous regulation of stationarity, equality
feasibility, and complementarity. It is the only tested method that reaches
the prescribed KKT tolerance within the iteration budget while recovering the
reference constrained optimum to high accuracy. These results support the use
of direct KKT-residual regulation for nonlinear mixed-constraint optimisation without explicit projection onto the feasible set.

\begin{table}[thpb]
\centering
\rowcolors{2}{gray!15}{white}
\caption{Problem I numerical comparison. The proposed SM--KKT solver is the
only method that reaches the prescribed KKT tolerance, while also recovering
the reference optimum with high accuracy and maintaining a competitive primal
trajectory length.}
\label{tab:problem_complex_results}
\begin{tabular}{lcccc}
\hline
Method & $\Phi$ Enter & $|f(x)-f^\star|$ & $\|S_k\|$ & Path \\
\hline
SM--KKT       & 390  & $\leq 10^{-18}$      & $9.022\times10^{-7}$ & 1.759 \\
Projected PD  & --   & $2.182\times10^{-2}$ & $1.362\times10^{-2}$ & 3.133 \\
Aug. Lag.     & 1993 & $4.0\times10^{-5}$   & $1.469\times10^{-4}$ & 2.643 \\
Quad. Penalty & --   & $3.164\times10^{-1}$ & $1.843\times10^{-1}$ & 2.112 \\
Log-Barrier   & --   & $3.171\times10^{-1}$ & $1.867\times10^{-1}$ & 0.665 \\
\hline
\end{tabular}
\end{table}
\subsection{\color{blue}Problem II: Rosenbrock-Type Constrained Optimisation}
\label{subsec:problem_rosenbrock}

The second benchmark considers a constrained Rosenbrock-type optimisation problem. This example is particularly challenging because the objective function contains a narrow curved valley that frequently causes slow convergence and oscillatory behaviour in gradient-based methods. The addition of multiple inequality constraints and an equality constraint further restricts the feasible region, creating a highly coupled optimisation landscape in which stationarity, feasibility, and complementarity must be satisfied simultaneously. The optimisation problem has the following objective function 
\[
f(x)
=
5(x_2-x_1^2)^2
+
(1-x_1)^2
+
0.1(x_1+x_2)^2 .
\]

The inequality constraints ($g(x)\leq0$) are
\begin{align*}
  g_1(x)&=x_1^2+x_2^2-1.80,
\\
g_2(x)&=-x_1-0.85,
\\
g_3(x)&=x_2-1.55,  
\end{align*}

while the equality $(h(x)=0)$ constraint is

\[
h(x)=x_1+\frac{1}{2}x_2-0.20.
\]

The feasible region is therefore formed by the intersection of a circular domain, two affine inequality constraints, and an affine equality manifold. The resulting constrained optimum lies within the Rosenbrock valley, making the benchmark particularly suitable for evaluating optimisation algorithms under strong geometric anisotropy. All methods are initialized from

\[
x_0=
\begin{bmatrix}
-1.25\\
1.65
\end{bmatrix},
\qquad
\lambda_0=
\begin{bmatrix}
0.1\\
0.1\\
0.1
\end{bmatrix},
\qquad
\mu_0=0.
\]

The objective landscape together with the optimisation trajectories is illustrated in Fig.~\ref{fig:problem2_3d}, while Fig.~\ref{fig:problem2_trajectory} shows the feasible region and projected optimisation paths.

\subsubsection{Results and Discussion}

The numerical results for the Rosenbrock benchmark are summarised in
Table~\ref{tab:problem_rosenbrock_results}. The three-dimensional Rosenbrock
objective surface and optimisation trajectories are shown in
Fig.~\ref{fig:problem2_3d}, while the corresponding two-dimensional objective
contours and optimisation trajectories are shown in
Fig.~\ref{fig:problem2_trajectory}. The KKT residual histories and convergence
metrics are presented in Figs.~\ref{fig:problem2_residuals}
and~\ref{fig:problem2_metrics}, respectively.

As illustrated in Fig.~\ref{fig:problem2_3d}, the optimisation algorithms
evolve over the characteristic narrow and strongly curved Rosenbrock valley
while simultaneously satisfying the imposed equality and inequality
constraints. Despite this challenging geometry, the proposed SM--KKT solver
progresses towards the reference constrained optimum.

\begin{figure}[thpb]
\centering
\includegraphics[width=\linewidth]{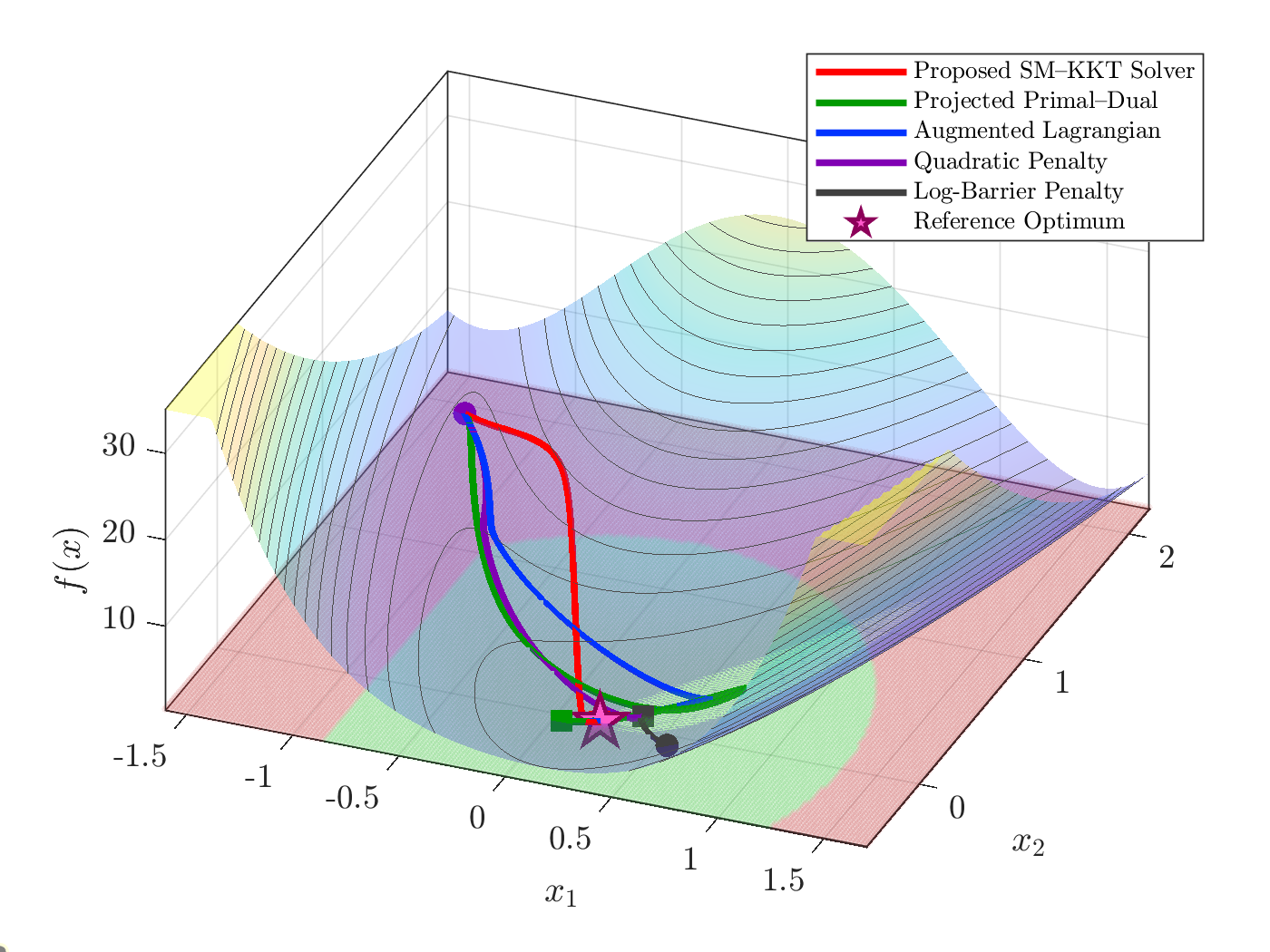}
\caption{Problem II: three-dimensional Rosenbrock-type objective surface
and optimisation trajectories. The proposed SM--KKT solver successfully
follows the narrow curved valley towards the reference constrained optimum,
despite the strong nonconvexity and coupled constraint geometry.}
\label{fig:problem2_3d}
\end{figure}

The corresponding two-dimensional representation in
Fig.~\ref{fig:problem2_trajectory} makes the optimisation geometry more
explicit. The objective contours and constraint boundaries highlight the
difficulty of progressing through the narrow curved valley while satisfying
the constraints. The SM--KKT trajectory successfully follows this geometry
towards the reference constrained optimum.

\begin{figure}[thpb]
\centering
\includegraphics[width=\linewidth]{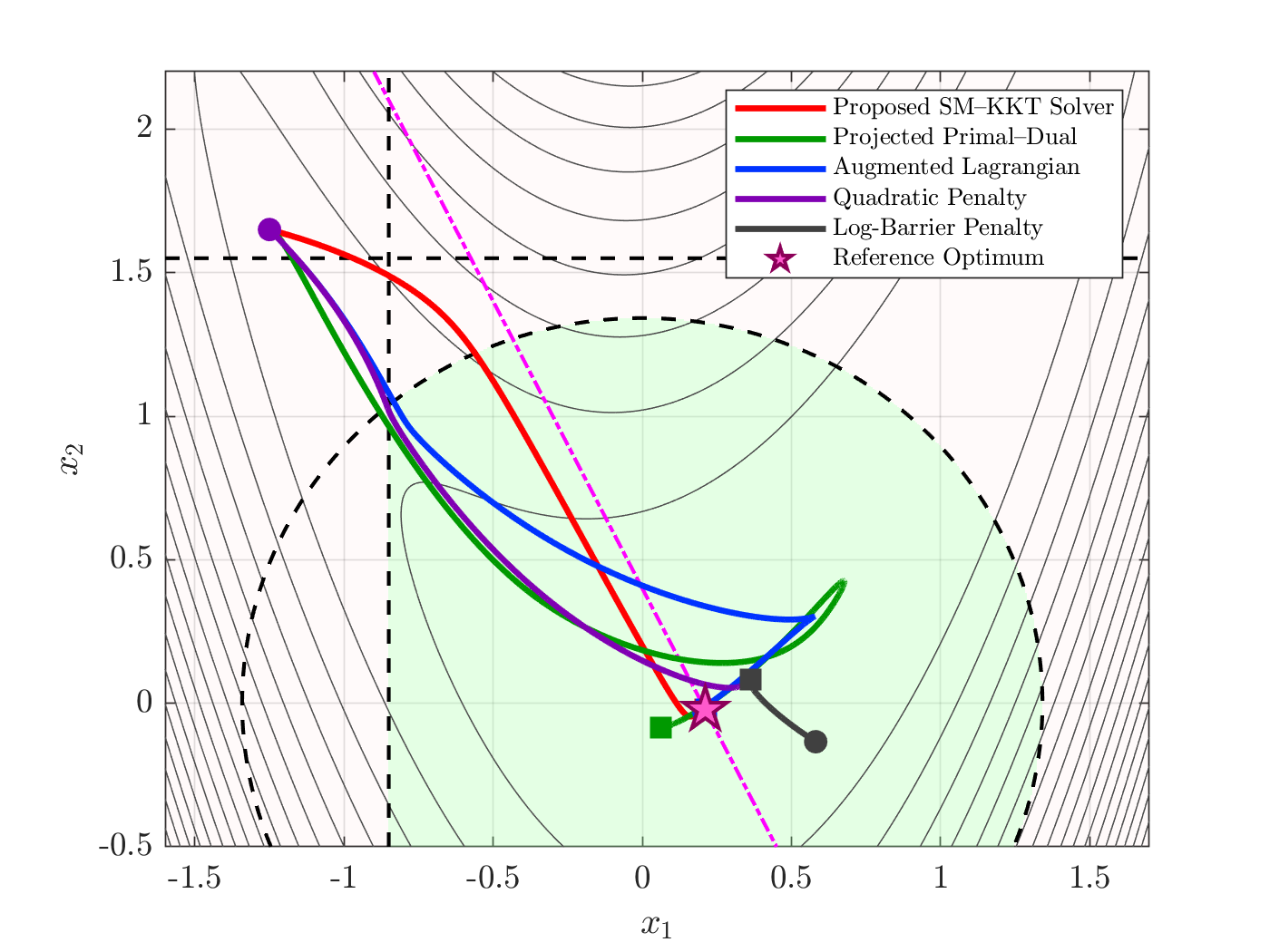}
\caption{Problem II: objective contours (\textcolor{darkgray}{dark grey}), feasible region (\textcolor{green!60!black}{green}),
infeasible region (\textcolor{pink!85!black}{light pink}), inequality constraint boundaries (black dashed),
equality constraint (\textcolor{magenta}{magenta dashed}), and optimisation trajectories.
The proposed SM--KKT solver successfully follows the narrow curved valley
towards the reference constrained optimum, despite the strong nonconvexity
and coupled constraint geometry.}
\label{fig:problem2_trajectory}
\end{figure}

Quantitatively, Table~\ref{tab:problem_rosenbrock_results} shows that SM--KKT
is the only method that reaches the prescribed KKT tolerance within the
maximum iteration budget of \(2500\) iterations. The proposed method converges
in \(1342\) iterations and achieves a final KKT residual of
\[
\|S\| = 9.801\times10^{-7}.
\]

Furthermore, the residual enters the prescribed sliding boundary layer after
\(1202\) iterations. At termination, the stationarity, equality-feasibility,
and complementarity residuals are
\[
\|S^x\| = 9.80\times10^{-7}, ~
\|S^h\| = 7.40\times10^{-14},~
\|S^c\| = 2.36\times10^{-12},
\]
respectively. These values show that stationarity, equality feasibility, and
complementarity are regulated simultaneously. Since the terminal overall
residual is essentially determined by \(\|S^x\|\), while
\(\|S^h\|\) and \(\|S^c\|\) are several orders of magnitude smaller, the
final numerical accuracy is primarily limited by stationarity rather than by
residual feasibility or complementarity errors. The objective error of at
most \(10^{-18}\), reported in
Table~\ref{tab:problem_rosenbrock_results}, further confirms accurate recovery
of the reference constrained optimum.

This behaviour is also evident in Fig.~\ref{fig:problem2_residuals}. The
SM--KKT solver exhibits a transient increase in the stationarity residual
while negotiating the strongly curved Rosenbrock valley, followed by a
sustained reduction of all KKT residual components towards the prescribed
tolerance. In contrast, the competing methods stagnate at substantially
larger residual levels and fail to achieve comparable simultaneous
satisfaction of the KKT conditions.

\begin{figure*}[thpb]
\centering
\includegraphics[width=\linewidth]{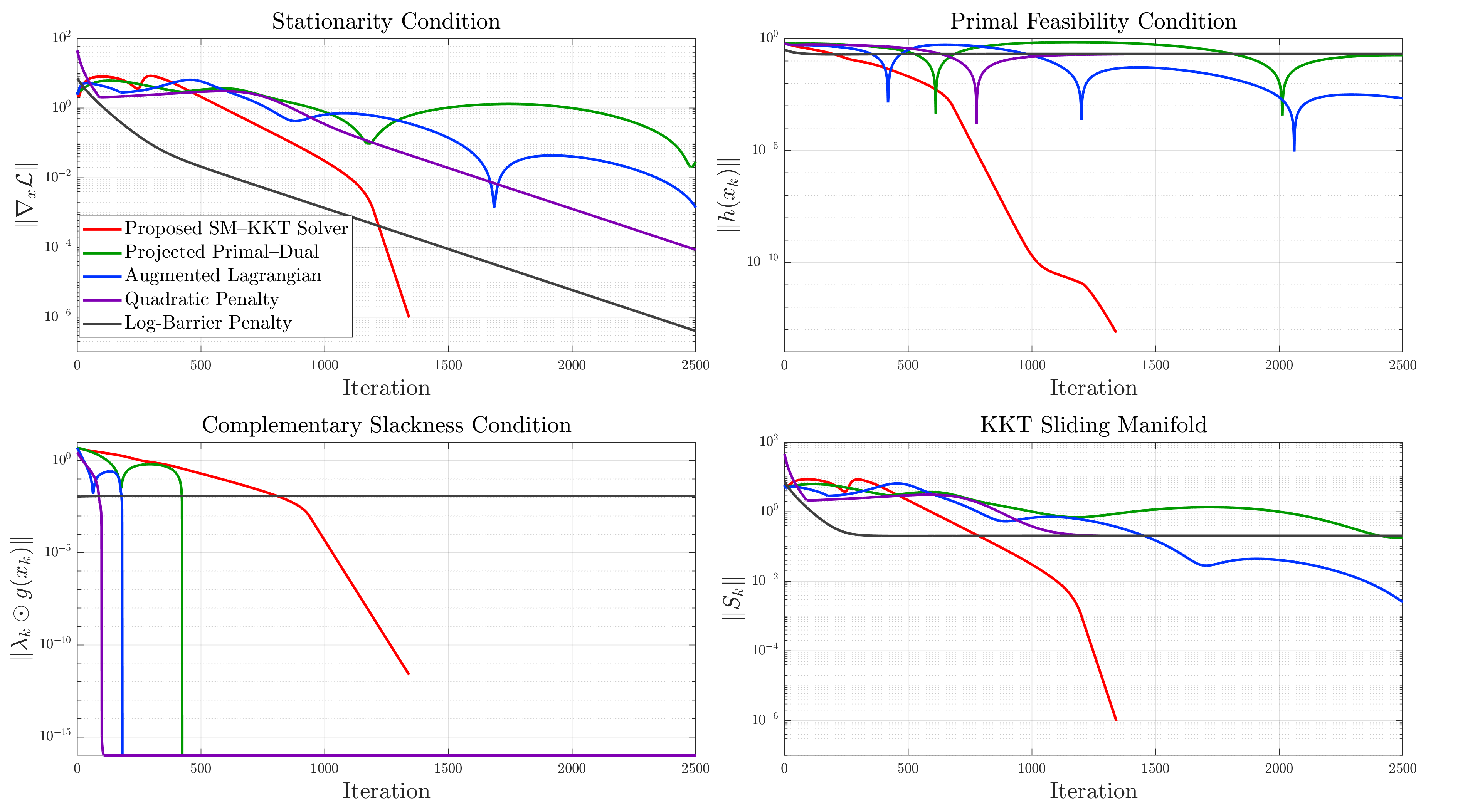}
\caption{Problem II: comparative evolution of the stationarity,
equality-feasibility, complementarity, and overall KKT residuals.
The proposed SM--KKT solver drives all residual components to the
prescribed tolerance despite the narrow curved Rosenbrock valley,
whereas the competing methods stagnate at substantially higher
residual levels.}
\label{fig:problem2_residuals}
\end{figure*}

Among the competing methods, the Augmented Lagrangian method provides the
closest result. As reported in Table~\ref{tab:problem_rosenbrock_results}, it
terminates with a final KKT residual of
\[
\|S\| = 2.531\times10^{-3}.
\]
Although its final objective error is relatively small,
\(2.670\times10^{-3}\), the method does not enter the prescribed sliding
boundary layer and requires the full iteration budget. Its terminal KKT
residual is more than three orders of magnitude larger than that obtained by
SM--KKT. This distinction is important because a small objective error alone
does not establish comparable stationarity, feasibility, and complementarity
accuracy.

The projected primal--dual method exhibits a substantially larger discrepancy,
terminating with
\[
\|S\| = 1.830\times10^{-1}.
\]
As shown in Fig.~\ref{fig:problem2_residuals}, substantial residual errors
remain at termination. Its trajectory in
Fig.~\ref{fig:problem2_trajectory} also undergoes a pronounced excursion
before approaching the constrained optimum. This behaviour is reflected in
the cumulative primal trajectory length of \(3.597\), reported in
Table~\ref{tab:problem_rosenbrock_results}, which is the largest among the
compared methods.

The quadratic-penalty and log-barrier methods exhibit similarly large
terminal KKT residuals of approximately
\[
\|S\| = 2.05\times10^{-1}.
\]
Their corresponding objective errors are \(2.111\times10^{-1}\) and
\(2.108\times10^{-1}\), respectively. Together with the residual histories
in Fig.~\ref{fig:problem2_residuals}, these results show that neither method
approaches the KKT manifold to the prescribed accuracy. Consequently,
objective reduction along an optimisation trajectory should not by itself be
interpreted as convergence to the constrained optimum; the complete KKT
residual must also be driven to the required accuracy.

Figure~\ref{fig:problem2_metrics} provides a complementary comparison through
the multiplier errors, objective error, and cumulative primal trajectory
length. The SM--KKT solver combines accurate multiplier recovery and objective
convergence with a cumulative primal path length of \(2.353\). By comparison,
the projected primal--dual and Augmented Lagrangian methods produce path
lengths of \(3.597\) and \(2.985\), respectively. The quadratic-penalty method
produces a comparable but slightly longer path of \(2.376\). These results are
consistent with the efficient SM--KKT trajectory observed in
Figs.~\ref{fig:problem2_3d} and~\ref{fig:problem2_trajectory}.

\begin{figure*}[thpb]
\centering
\includegraphics[width=\linewidth]{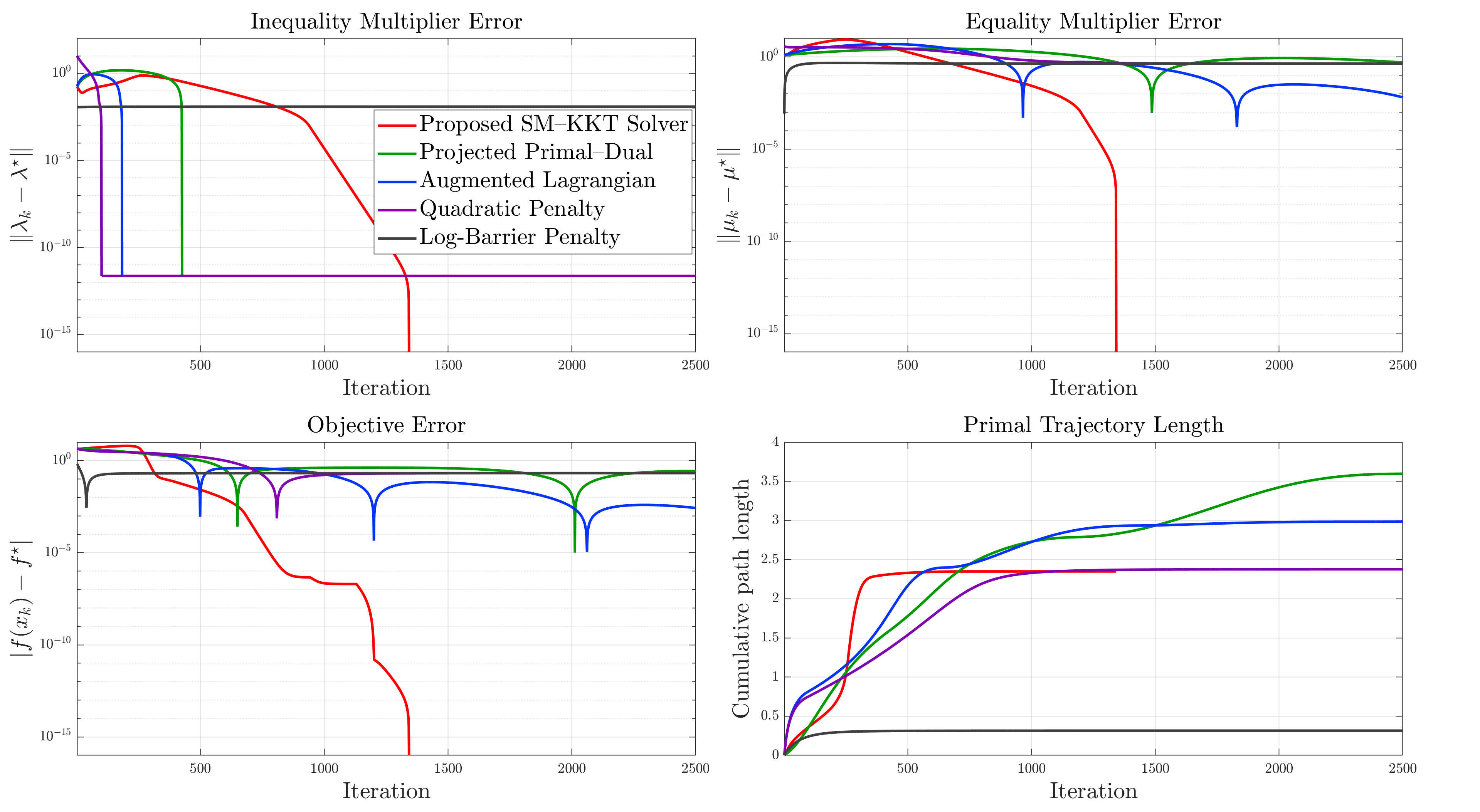}
\caption{Problem II: evolution of multiplier errors, objective error,
and cumulative primal trajectory length. The proposed SM--KKT solver
achieves accurate multiplier recovery and objective convergence while
maintaining an efficient trajectory through the narrow curved valley,
whereas the competing methods exhibit larger errors or longer
optimisation paths.}
\label{fig:problem2_metrics}
\end{figure*}

It is important, however, not to interpret path length independently of
terminal accuracy. The log-barrier method has the shortest reported path,
\(0.316\), compared with \(2.353\) for SM--KKT. However, it terminates with
an objective error of \(2.108\times10^{-1}\) and a KKT residual of
\(2.050\times10^{-1}\). Hence, its shorter trajectory does not represent a
more efficient route to the constrained optimum. A meaningful comparison of
trajectory efficiency must therefore be considered jointly with the attained
KKT and objective accuracies.

Taken together, Figs.~\ref{fig:problem2_3d},
\ref{fig:problem2_trajectory}, \ref{fig:problem2_residuals},
and~\ref{fig:problem2_metrics}, together with
Table~\ref{tab:problem_rosenbrock_results}, show that SM--KKT provides the
best overall performance for this benchmark. The method is the only tested
algorithm that reaches the prescribed KKT tolerance within the iteration
budget while simultaneously recovering the reference constrained solution
and regulating stationarity, equality feasibility, and complementarity to
high accuracy. These results demonstrate that direct KKT-residual regulation
remains effective in the narrow and strongly curved landscape characteristic
of the Rosenbrock benchmark.

\begin{table}[thpb]
\centering
\rowcolors{2}{gray!15}{white}
\caption{Problem II numerical comparison. The proposed SM--KKT solver is the
only method that reaches the prescribed KKT tolerance, while accurately
recovering the reference constrained optimum and maintaining an efficient
primal trajectory.}
\label{tab:problem_rosenbrock_results}
\begin{tabular}{lcccc}
\hline
Method & $\Phi$ Enter & $|f(x)-f^\star|$ & $\|S_k\|$ & Path \\
\hline
SM--KKT       & 1202 & $\leq 10^{-18}$      & $9.801\times10^{-7}$ & 2.353 \\
Projected PD  & --   & $2.713\times10^{-1}$ & $1.830\times10^{-1}$ & 3.597 \\
Aug. Lag.     & --   & $2.670\times10^{-3}$ & $2.531\times10^{-3}$ & 2.985 \\
Quad. Penalty & --   & $2.111\times10^{-1}$ & $2.050\times10^{-1}$ & 2.376 \\
Log-Barrier   & --   & $2.108\times10^{-1}$ & $2.050\times10^{-1}$ & 0.316 \\
\hline
\end{tabular}
\end{table}
\subsection{\color{blue}Problem III: Banana Objective with Annular Feasible Set}
\label{subsec:problem_banana}

The third benchmark considers a nonconvex banana-shaped objective function constrained by an annular feasible region together with an affine equality constraint. Unlike the previous examples, the feasible set is defined by the intersection of two nonlinear inequality constraints representing inner and outer circular boundaries. Consequently, the optimisation trajectory must evolve within a narrow annular region while simultaneously satisfying an equality manifold. This configuration creates a challenging geometric optimisation problem in which both active-set transitions and curvature effects play an important role.

The optimisation problem  objective function is
\[
f(x)
=
(1-x_1)^2
+
8(x_2-x_1^2)^2
+
0.05\sin(3x_1).
\]

The inequality constraints ($g(x)\leq0$) are
\begin{align*}
    g_1(x)
&=
x_1^2+x_2^2-2.2,
\\
g_2(x)
&=
0.45-(x_1^2+x_2^2),
\end{align*}
which define the annular feasible region
\[
0.45
\le
x_1^2+x_2^2
\le
2.2.
\]

The equality $(h(x)=0)$ constraint is

\[
h(x)
=
x_1+x_2-0.35.
\]

The feasible set is therefore the intersection of an annulus and an affine manifold. The resulting constrained optimum lies on a highly curved objective landscape, making this benchmark particularly suitable for assessing the ability of the proposed solver to navigate complex feasible geometries while maintaining KKT feasibility. All methods are initialized from

\[
x_0=
\begin{bmatrix}
-1.4\\
0.25
\end{bmatrix},
\qquad
\lambda_0=
\begin{bmatrix}
0.1\\
0.1
\end{bmatrix},
\qquad
\mu_0=0.
\]

The objective surface and optimisation trajectories are shown in Fig.~\ref{fig:problem3_3d}, while Fig.~\ref{fig:problem3_trajectory} illustrates the annular feasible region, equality manifold, and optimisation paths followed by the competing algorithms.

\subsubsection{Results and Discussion}

The numerical results for the banana objective with annular feasible set are
summarised in Table~\ref{tab:problem_banana_results}. The three-dimensional
objective surface and optimisation trajectories are shown in
Fig.~\ref{fig:problem3_3d}, while the corresponding two-dimensional constraint
geometry and optimisation trajectories are shown in
Fig.~\ref{fig:problem3_trajectory}. The KKT residual histories and convergence
metrics are presented in Figs.~\ref{fig:problem3_residuals}
and~\ref{fig:problem3_metrics}, respectively.

As illustrated in Fig.~\ref{fig:problem3_3d}, the optimisation algorithms
evolve over a highly nonconvex banana-shaped objective landscape while
simultaneously satisfying the constraints defining the feasible set. The
proposed SM--KKT solver progresses towards the reference constrained optimum
along a stable and well-behaved trajectory.
\begin{figure}[thpb]
\centering
\includegraphics[width=\linewidth]{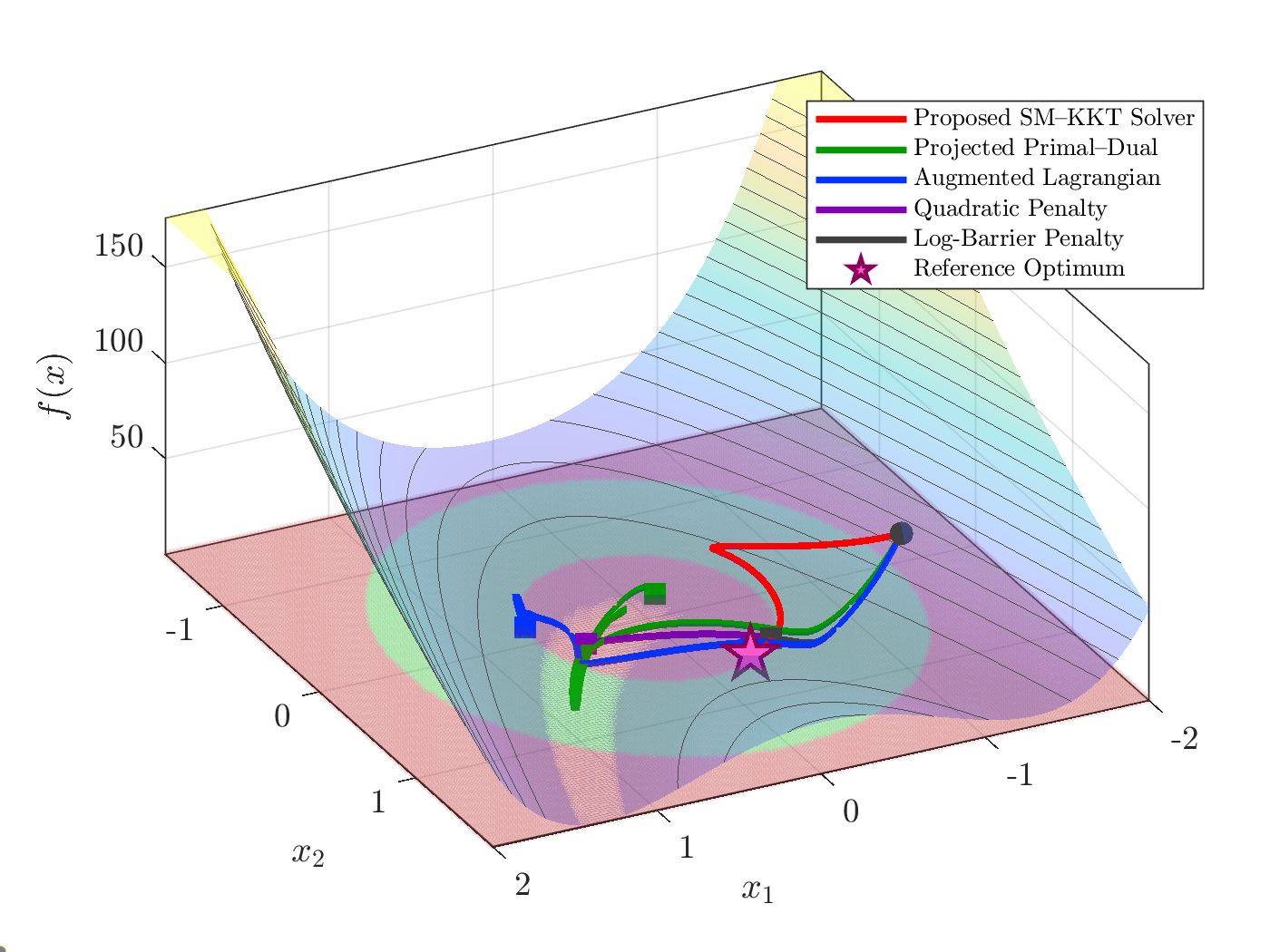}
\caption{Problem III: three-dimensional banana-shaped objective surface
and optimisation trajectories. The proposed SM--KKT solver follows a
well-behaved path towards the reference constrained optimum despite the
nonconvex objective landscape and geometrically restrictive annular
constraint geometry.}
\label{fig:problem3_3d}
\end{figure}
The corresponding two-dimensional representation in
Fig.~\ref{fig:problem3_trajectory} makes the constraint geometry more explicit.
The annular feasible region, together with the equality constraint, creates a
geometrically restrictive domain in which the optimisation trajectory must
navigate the nonconvex objective landscape while maintaining convergence
towards the constrained solution. In this setting, the SM--KKT trajectory
follows a well-behaved path towards the reference constrained optimum.
\begin{figure}[thpb]
\centering
\includegraphics[width=\linewidth]{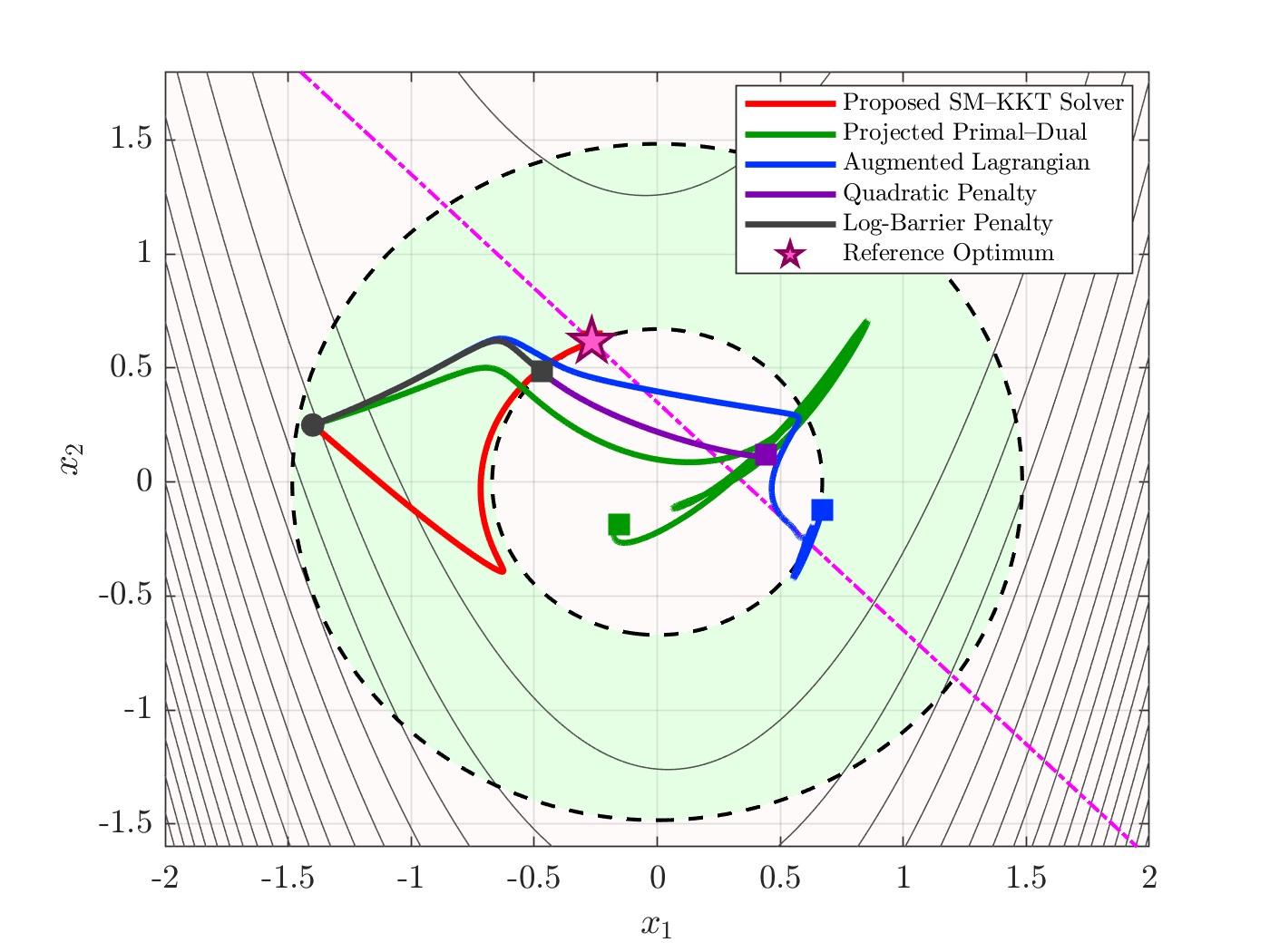}
\caption{Problem III: objective contours (\textcolor{darkgray}{dark grey}), annular feasible region (\textcolor{green!60!black}{green}), infeasible region (\textcolor{pink!85!black}{light pink}), inequality constraint boundaries (black dashed), equality constraint (\textcolor{magenta}{magenta dashed}), and optimisation trajectories. The proposed SM–KKT solver follows a well-behaved path through the geometrically restrictive annular feasible set towards the reference constrained optimum, despite the nonconvex objective landscape.}
\label{fig:problem3_trajectory}
\end{figure}
Quantitatively, Table~\ref{tab:problem_banana_results} shows that SM--KKT is
the only method that reaches the prescribed KKT tolerance. The proposed method
converges in \(960\) iterations and achieves a final KKT residual of
\[
\|S\| = 9.529\times10^{-7}.
\]
Furthermore, the residual enters the prescribed sliding boundary layer after
\(867\) iterations. At termination, the stationarity, equality-feasibility,
and complementarity residuals are
\[
\|S^x\| = 9.52\times10^{-7},~
\|S^h\| = 2.40\times10^{-14},~
\|S^c\| = 4.53\times10^{-14},
\]
respectively. These values show that stationarity, equality feasibility, and complementarity are regulated simultaneously. Since the terminal overall
residual is essentially determined by the stationarity component, while the equality-feasibility and complementarity residuals are substantially smaller,
the final numerical accuracy is primarily limited by stationarity. The objective error of at most \(10^{-18}\), reported in Table~\ref{tab:problem_banana_results}, further confirms accurate recovery of the reference constrained optimum. This behaviour is also evident in Fig.~\ref{fig:problem3_residuals}. The
SM--KKT solver drives all components of the KKT residual towards the
prescribed tolerance while maintaining bounded behaviour throughout the
optimisation process. In contrast, the competing methods exhibit persistent
stationarity, feasibility, or complementarity errors and fail to achieve
comparable satisfaction of the complete KKT conditions.

\begin{figure*}[thpb]
\centering
\includegraphics[width=\linewidth]{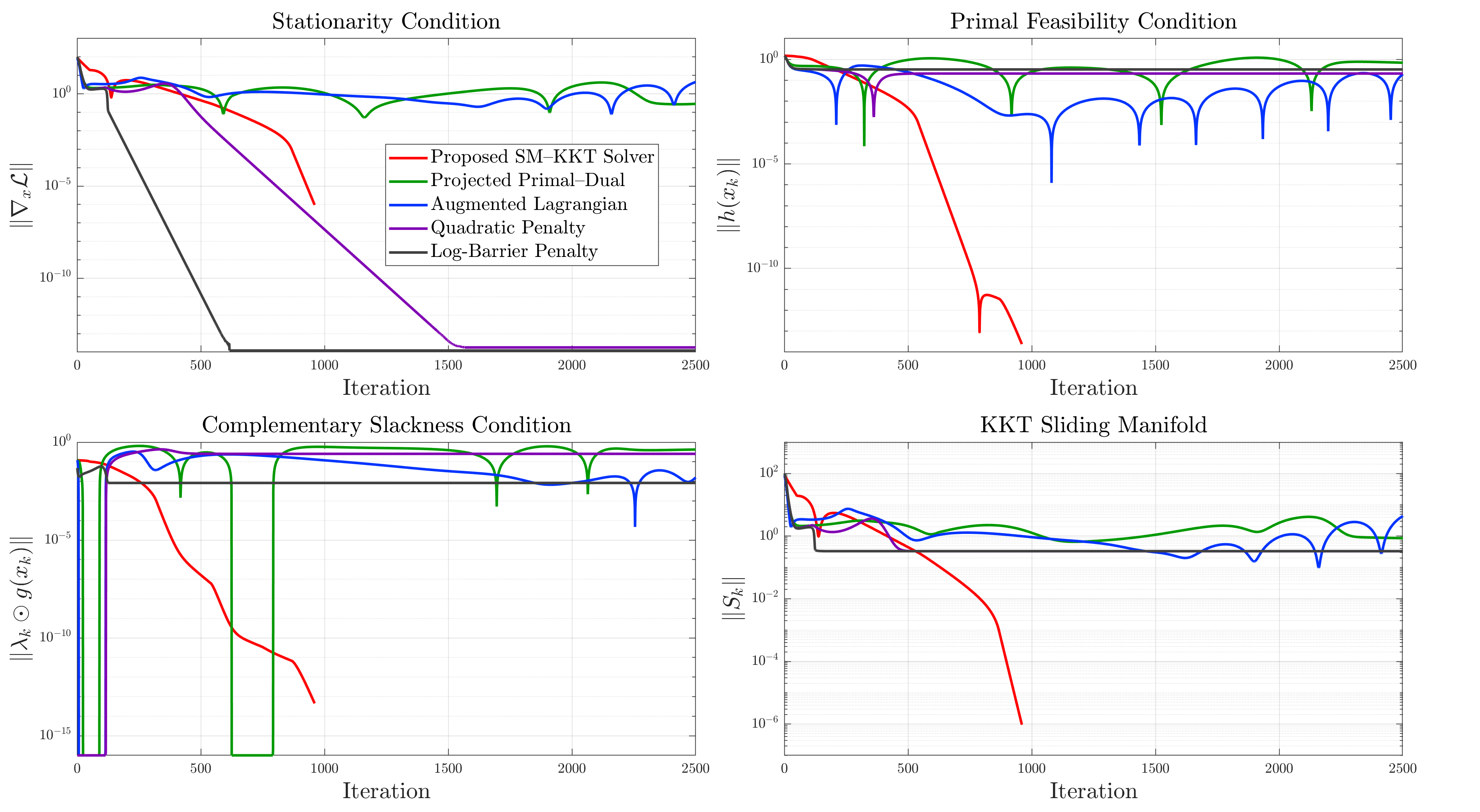}
\caption{Problem III: comparative evolution of the stationarity,
equality-feasibility, complementarity, and overall KKT residuals.
The proposed SM--KKT solver drives all residual components simultaneously
below the prescribed tolerance despite the restrictive annular constraint
geometry, whereas the competing methods exhibit persistent stationarity,
feasibility, or complementarity errors.}
\label{fig:problem3_residuals}
\end{figure*}

The projected primal--dual method performs poorly on this benchmark. It
terminates with a final KKT residual of $\|S\| = 8.612\times10^{-1},$
and an objective error of \(2.274\). As illustrated in
Fig.~\ref{fig:problem3_trajectory}, its trajectory undergoes substantial
excursions through the annular region before termination. This behaviour is
also reflected in its cumulative primal path length of \(6.473\), which is
the largest among the compared methods. The Augmented Lagrangian method exhibits an even larger terminal KKT residual.
Its final stationarity residual reaches $\|S^x\| = 4.392,$ resulting in an overall KKT residual of
$\|S\| = 4.443.$ The method therefore does not approach the reference KKT point to the
prescribed accuracy. Its objective error of \(1.176\) further indicates a
substantial discrepancy from the reference constrained optimum. The quadratic-penalty and log-barrier methods provide intermediate results,
but neither reaches the prescribed KKT tolerance. Their final KKT residuals
are \(3.311\times10^{-1}\) and \(3.318\times10^{-1}\), respectively, with
corresponding objective errors of \(3.538\) and \(1.267\). These results
reinforce the importance of assessing convergence through the complete KKT
residual rather than through objective reduction or trajectory behaviour
alone.

Figure~\ref{fig:problem3_metrics} provides a complementary comparison through
the multiplier errors, objective error, and cumulative primal trajectory
length. The SM--KKT solver combines accurate multiplier recovery and objective
convergence with a cumulative primal path length of \(2.177\). By comparison,
the projected primal--dual and Augmented Lagrangian methods produce path
lengths of \(6.473\) and \(3.471\), respectively. These results are consistent
with the stable and comparatively direct SM--KKT trajectory observed in
Figs.~\ref{fig:problem3_3d} and~\ref{fig:problem3_trajectory}.
\begin{figure*}[thpb]
\centering
\includegraphics[width=\linewidth]{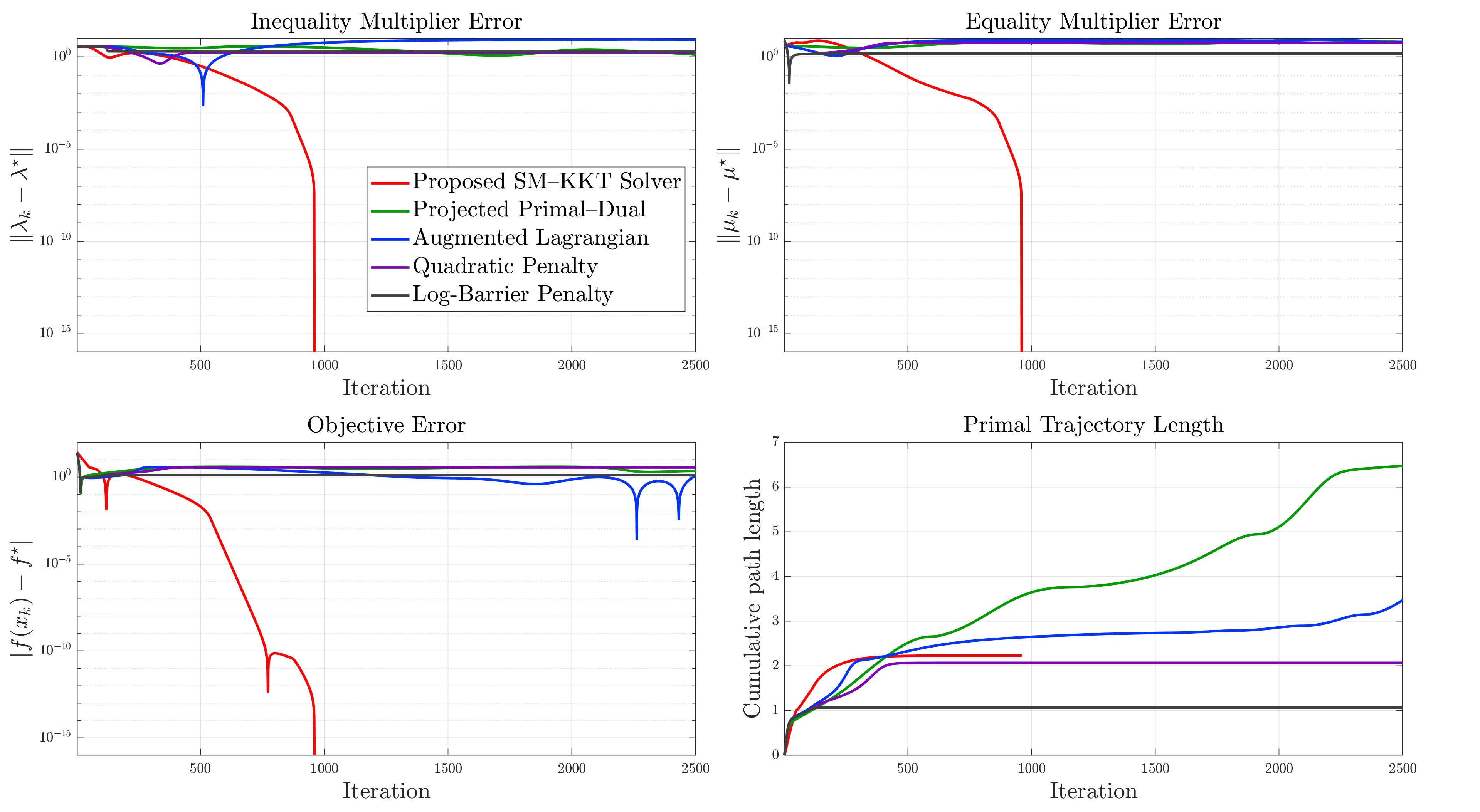}
\caption{Problem III: evolution of multiplier errors, objective error,
and cumulative primal trajectory length. The proposed SM--KKT solver
achieves accurate multiplier recovery and objective convergence while
following an efficient path through the restrictive annular geometry,
whereas the competing methods exhibit persistent multiplier errors or
less accurate convergence.}
\label{fig:problem3_metrics}
\end{figure*} As in the previous benchmarks, path length must be interpreted jointly with
terminal accuracy. The quadratic-penalty and log-barrier methods have shorter
reported paths of \(2.068\) and \(1.070\), respectively, compared with
\(2.177\) for SM--KKT. However, they terminate with KKT residuals of
\(3.311\times10^{-1}\) and \(3.318\times10^{-1}\), respectively, together
with substantially larger objective errors. Their shorter trajectories
therefore do not represent more efficient routes to the constrained optimum.
A meaningful comparison of trajectory efficiency must account for the
accuracy of the terminal KKT point.

Taken together, Figs.~\ref{fig:problem3_3d},
\ref{fig:problem3_trajectory}, \ref{fig:problem3_residuals},
and~\ref{fig:problem3_metrics}, together with
Table~\ref{tab:problem_banana_results}, demonstrate the overall advantage of
SM--KKT on this benchmark. Despite the strongly nonconvex objective and the
restrictive annular constraint geometry, SM--KKT is the only tested method
that reaches the prescribed KKT tolerance while accurately recovering the
reference constrained optimum. In particular, the simultaneous reduction of
the stationarity, equality-feasibility, and complementarity residuals,
together with accurate multiplier recovery and a well-behaved primal
trajectory, indicates convergence to the complete KKT solution rather than
improvement in the objective value alone. These results highlight the
effectiveness and robustness of direct KKT-residual regulation when both
objective nonconvexity and complex feasible-set geometry influence the
optimisation dynamics.
\begin{table}[thpb]
\rowcolors{2}{gray!15}{white}
\centering
\caption{Problem III numerical comparison. The proposed SM–KKT solver achieves the best overall performance, reaching the prescribed KKT tolerance while accurately recovering the reference constrained optimum and maintaining an efficient primal trajectory.}
\label{tab:problem_banana_results}
\begin{tabular}{lcccc}
\hline
Method & Phi Enter & \(|f(x)-f^\star|\) & \(\|S_k\|\) & Path \\
\hline
SM--KKT & 867 & \( \leq 10^{-18} \) & \(9.529\times10^{-7}\) & \(2.177\)\\
Projected PD & -- & \(2.274\times10^{0}\) & \(8.612\times10^{-1}\) & \(6.473\)\\
Aug. Lag. & -- & \(1.176\times10^{0}\) & \(4.443\times10^{0}\) & \(3.471\)\\
Quad. Penalty & -- & \(3.538\times10^{0}\) & \(3.311\times10^{-1}\) & \(2.068\)\\
Log-Barrier & -- & \(1.267\times10^{0}\) & \(3.318\times10^{-1}\) & \(1.070\)\\
\hline
\end{tabular}
\end{table}
\subsection{\color{blue}Problem IV: Trigonometric Nonlinear Constrained Optimisation}
\label{subsec:problem_trigonometric}

The fourth benchmark considers a highly nonlinear constrained optimisation problem with a trigonometric objective term, nonlinear inequality constraints, and a nonlinear equality constraint. This example is included to assess the behaviour of the proposed SM--KKT solver in the presence of oscillatory objective geometry and nonlinear feasibility manifolds. Unlike the previous examples, the equality constraint itself is nonlinear, which creates a curved feasible manifold and increases the difficulty of simultaneously satisfying stationarity, feasibility, and complementarity. The optimisation problem cost function  is: 
\[
f(x)
=
0.6(x_1-0.5)^2
+
0.8(x_2-0.8)^2
+
0.18\sin(4x_1)\cos(3x_2).
\]

The inequality constraints ($g(x)\leq0$) are
\begin{align*}
 g_1(x)&=x_1^2+x_2^2-1.3,
\\
g_2(x)&=-x_1-0.9,
\\
g_3(x)&=x_1+0.4x_2-0.8.
\end{align*}
The nonlinear equality $(h(x)=0)$ constraint is
\[
h(x)=\sin(0.7x_1)+x_2-0.25.
\]

All methods are initialized from
\[
x_0=
\begin{bmatrix}
-1.2\\
-0.7
\end{bmatrix},
\qquad
\lambda_0=
\begin{bmatrix}
0.1\\
0.1\\
0.1
\end{bmatrix},
\qquad
\mu_0=0.
\]

The feasible set is formed by the intersection of a circular inequality, two affine inequality constraints, and a nonlinear equality manifold. The trigonometric perturbation in the objective generates local oscillations in the descent direction, while the nonlinear equality constraint imposes a curved manifold that must be tracked accurately. This benchmark therefore provides a stringent test of robustness for the proposed residual-manifold formulation.

\subsubsection{Results and Discussion}
The numerical results for the trigonometric nonlinear constrained problem are summarised in Table~\ref{tab:problem_trig_results}. The three-dimensional
objective surface and optimisation trajectories are shown in
Fig.~\ref{fig:problem4_3d}, while the corresponding two-dimensional constraint
geometry and optimisation trajectories are shown in
Fig.~\ref{fig:problem4_trajectory}. The KKT residual histories and convergence
metrics are presented in Figs.~\ref{fig:problem4_residuals}
and~\ref{fig:problem4_metrics}, respectively.

As illustrated in Fig.~\ref{fig:problem4_3d}, this benchmark combines an
oscillatory trigonometric objective landscape with nonlinear constraint
geometry. The curved equality manifold further restricts the admissible
solution set, requiring the optimisation process to reduce the objective while
simultaneously satisfying the coupled KKT conditions. Despite this geometry,
the proposed SM--KKT solver follows a stable trajectory towards the reference
constrained optimum.
\begin{figure}[thpb]
\centering
\includegraphics[width=\linewidth]{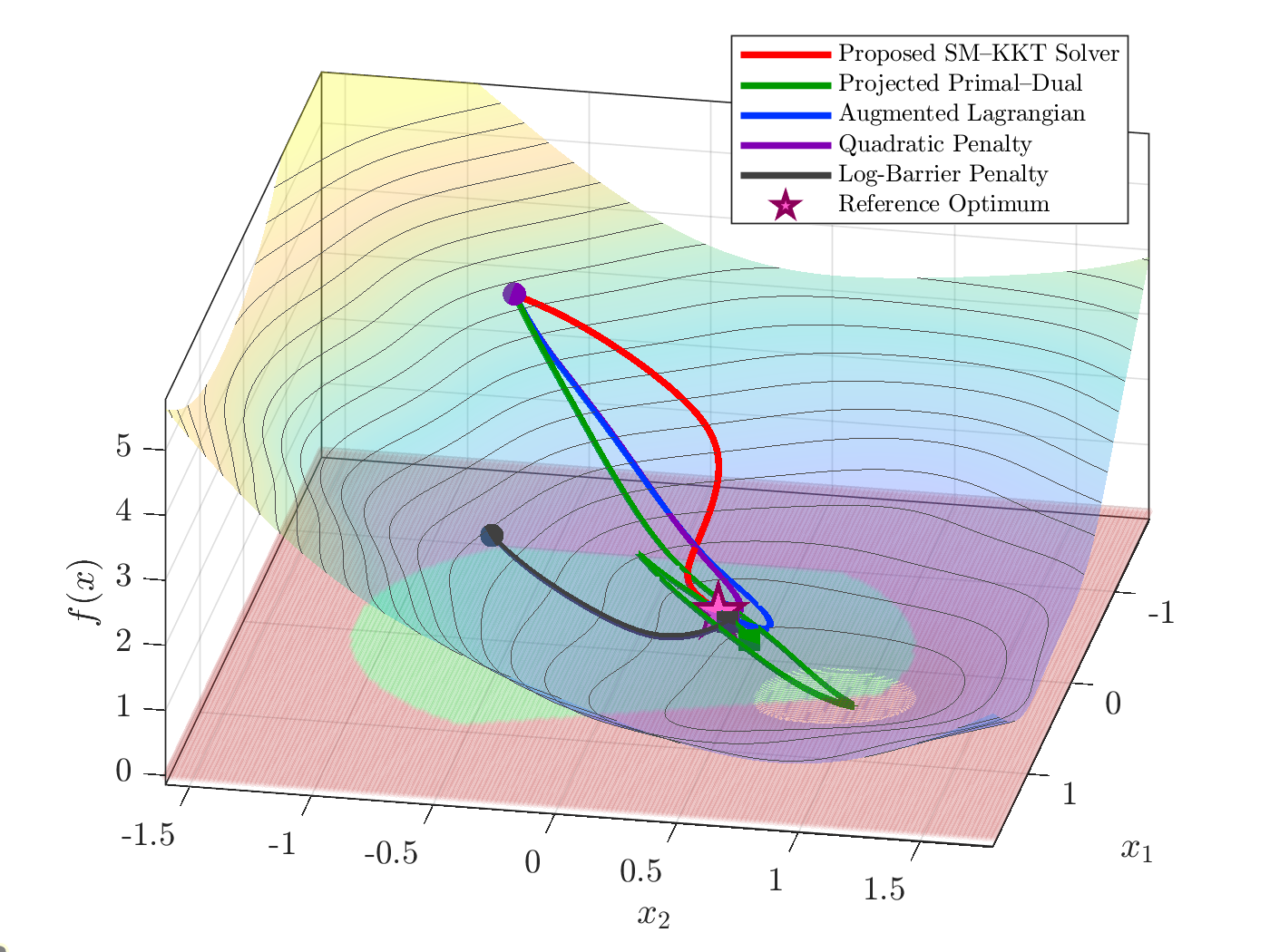}
\caption{Problem IV: three-dimensional trigonometric objective surface
and optimisation trajectories. The proposed SM--KKT solver converges
efficiently towards the reference constrained optimum despite the
oscillatory objective landscape and curved nonlinear equality manifold.}
\label{fig:problem4_3d}
\end{figure}
The corresponding two-dimensional representation in
Fig.~\ref{fig:problem4_trajectory} makes the interaction between the objective
contours and nonlinear constraint geometry more explicit. The SM--KKTtrajectory progresses towards the reference constrained optimum while
remaining well behaved in the presence of the curved equality manifold andinequality boundaries.
\begin{figure}[thpb]
\centering
\includegraphics[width=\linewidth]{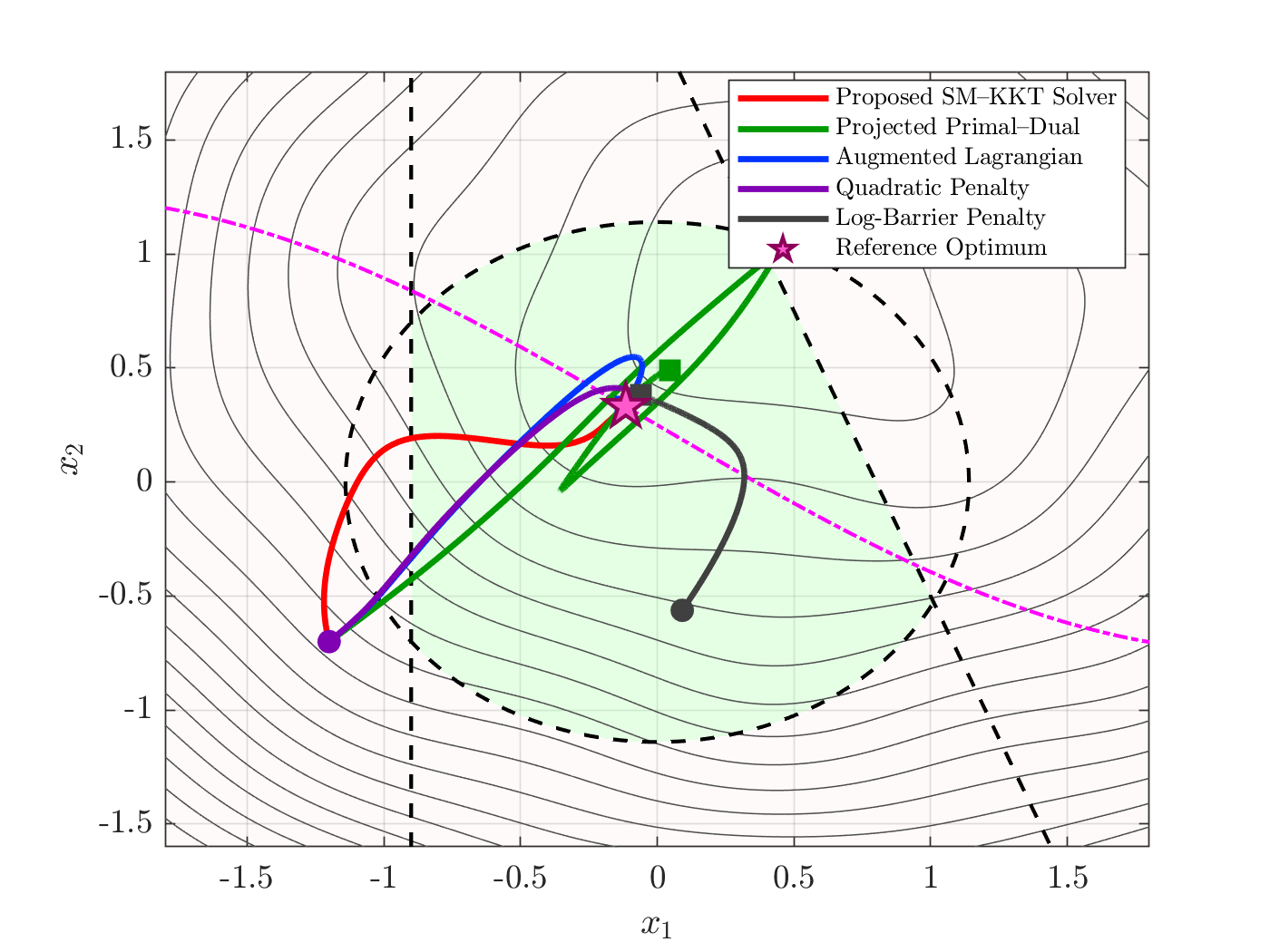}
\caption{Problem IV: objective contours (\textcolor{darkgray}{dark grey}), feasible region (\textcolor{green!60!black}{green}),
infeasible region (\textcolor{pink!85!black}{light pink}), inequality constraint boundaries (black dashed),
nonlinear equality constraint (\textcolor{magenta}{magenta dashed}), and optimisation
trajectories. The proposed SM--KKT solver follows a well-behaved path
towards the reference constrained optimum despite the oscillatory
objective landscape and curved nonlinear equality manifold.}
\label{fig:problem4_trajectory}
\end{figure}
Quantitatively, Table~\ref{tab:problem_trig_results} shows that SM--KKT
achieves the fastest convergence to the prescribed KKT tolerance. The method
reaches a final KKT residual of
\[
\|S\| = 9.871\times10^{-7},
\]
and enters the prescribed sliding boundary layer after only \(501\)
iterations. At termination, the stationarity, equality-feasibility, and complementarity
residuals are
\[
\|S^x\| = 9.84\times10^{-7},~
\|S^h\| = 7.34\times10^{-8}, ~
\|S^c\| = 1.03\times10^{-8},
\]
respectively. Thus, all three components of the KKT residual are reduced tosmall values simultaneously. The stationarity component remains the dominant
contribution to the terminal residual, while equality feasibility and
complementarity are regulated to still smaller levels. The objective error of
at most \(10^{-18}\), reported in Table~\ref{tab:problem_trig_results},
further confirms accurate recovery of the reference constrained optimum. The residual histories in Fig.~\ref{fig:problem4_residuals} provide a direct
view of this convergence behaviour. SM--KKT drives the stationarity,
equality-feasibility, and complementarity residuals towards the prescribed
tolerance while reaching the boundary layer substantially earlier than the
other successful method. The competing approaches either converge more slowly
or terminate with appreciably larger residual errors.

\begin{figure*}[thpb]
\centering
\includegraphics[width=\linewidth]{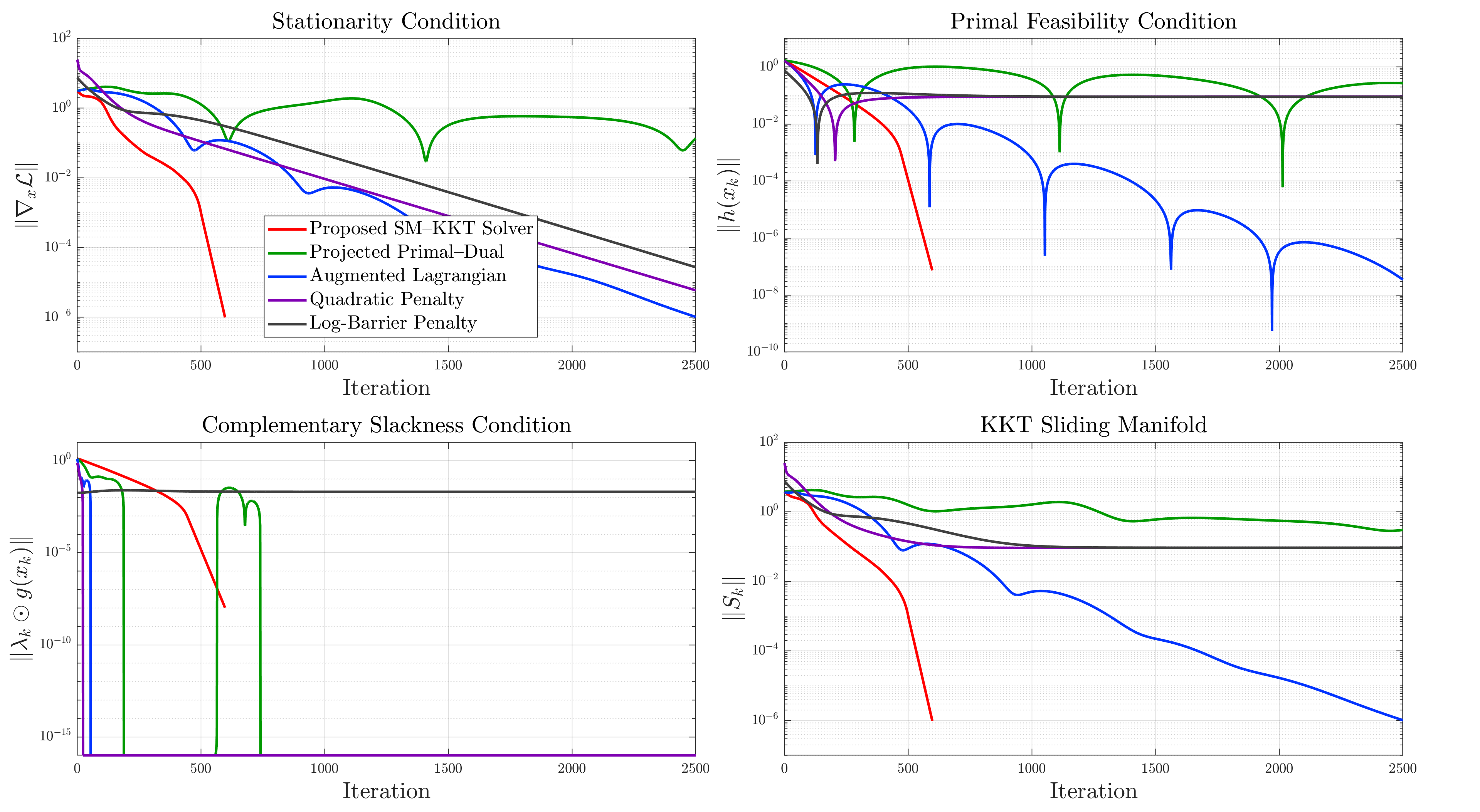}
\caption{Problem IV: comparative evolution of the stationarity,
equality-feasibility, complementarity, and overall KKT residuals.
The proposed SM--KKT solver drives all residual components to the
prescribed tolerance despite the nonlinear constraint geometry, whereas
the competing methods exhibit slower convergence or persistent residual
errors.}
\label{fig:problem4_residuals}
\end{figure*}

Among the competing methods, the Augmented Lagrangian method provides the
closest result. It ultimately reaches a final KKT residual of $\|S\| = 1.015\times10^{-6},$ and, like SM--KKT, attains an objective error of at most \(10^{-18}\).
However, it enters the prescribed boundary layer only after \(1315\) iterations, compared with \(501\) iterations for SM--KKT. The proposed method therefore reaches the required residual neighbourhood considerably earlier, while also producing a slightly shorter cumulative primal trajectory of \(1.893\), compared with \(1.951\) for the Augmented Lagrangian method.
The projected primal--dual method exhibits substantially poorer terminal accuracy, with a final KKT residual of $\|S\| = 3.036\times10^{-1},$ and an objective error of \(1.545\times10^{-1}\). Its cumulative primal path length of \(4.325\) is also the largest among the compared methods, indicating a considerably less efficient trajectory towards the constrained solution. The quadratic-penalty and log-barrier methods provide intermediate terminal
residuals of \(9.181\times10^{-2}\) and \(9.272\times10^{-2}\),
respectively. Their corresponding objective errors are
\(5.026\times10^{-2}\) and \(4.953\times10^{-2}\). Neither method therefore
reaches the prescribed KKT tolerance. These results again illustrate that
trajectory length or objective reduction alone is insufficient for assessing
convergence in a constrained problem; the complete KKT residual must also be
considered.

Figure~\ref{fig:problem4_metrics} provides a complementary comparison through
the multiplier errors, objective error, and cumulative primal trajectory
length. The proposed SM--KKT solver achieves accurate multiplier recovery and
objective convergence while maintaining an efficient primal trajectory. The
Augmented Lagrangian method eventually attains comparable terminal accuracy
but requires substantially more iterations before entering the prescribed
boundary layer, whereas the remaining methods retain larger terminal errors.
\begin{figure*}[thpb]
\centering
\includegraphics[width=\linewidth]{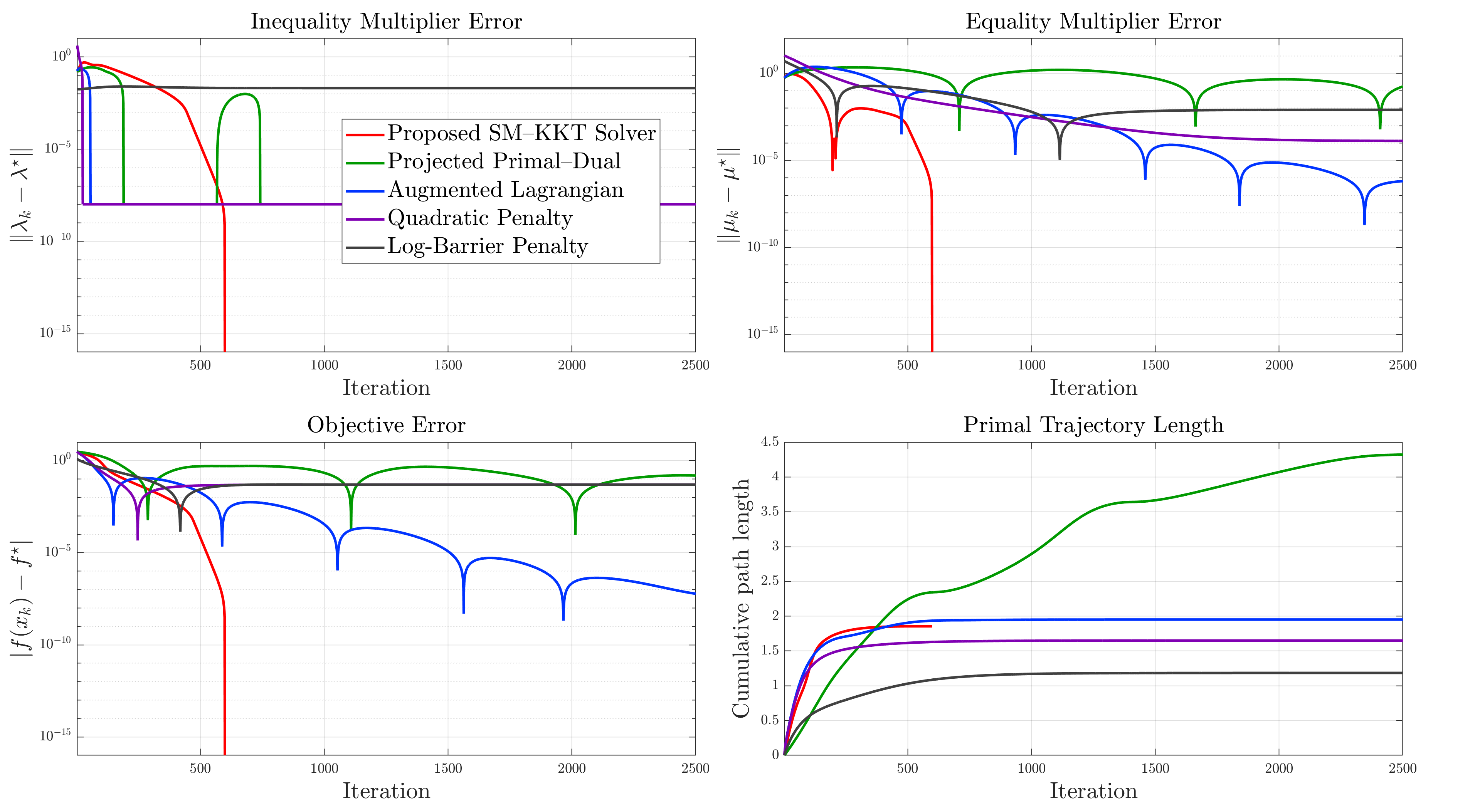}
\caption{Problem IV: evolution of multiplier errors, objective error,
and cumulative primal trajectory length. The proposed SM--KKT solver
achieves accurate multiplier recovery and objective convergence while
following an efficient path towards the constrained optimum, whereas
the competing methods exhibit slower or less accurate convergence.}
\label{fig:problem4_metrics}
\end{figure*}
As in the preceding benchmarks, path length must be interpreted jointly with
terminal accuracy. The quadratic-penalty and log-barrier methods have shorter
reported paths of \(1.650\) and \(1.185\), respectively, compared with
\(1.893\) for SM--KKT. However, their terminal KKT residuals remain on the
order of \(10^{-2}\), whereas SM--KKT reaches the prescribed \(10^{-6}\)
tolerance. Their shorter trajectories therefore do not represent more
efficient convergence to the reference constrained optimum.

Taken together, Figs.~\ref{fig:problem4_3d},
\ref{fig:problem4_trajectory}, \ref{fig:problem4_residuals},
and~\ref{fig:problem4_metrics}, together with
Table~\ref{tab:problem_trig_results}, demonstrate the convergence advantage
of SM--KKT on this benchmark. Despite the oscillatory objective landscape and
curved nonlinear equality manifold, the method accurately recovers the
reference constrained optimum while simultaneously reducing stationarity,
equality-feasibility, and complementarity residuals to the prescribed
accuracy. Although the Augmented Lagrangian method ultimately attains a
comparable terminal KKT residual and objective accuracy, SM--KKT enters the
prescribed boundary layer substantially earlier and follows a slightly shorter
primal trajectory. The results therefore demonstrate that direct KKT-residual
regulation can provide both high terminal accuracy and rapid convergence in
the presence of strongly nonlinear objective and constraint geometry.

\begin{table}[thpb]
\centering
\rowcolors{2}{gray!15}{white}
\caption{Problem IV numerical comparison. The proposed SM–KKT solver achieves the fastest convergence, reaching the prescribed KKT tolerance while accurately satisfying the nonlinear constraints and maintaining an efficient primal trajectory.}
\label{tab:problem_trig_results}
\begin{tabular}{lcccc}
\hline
Method & Phi Enter & \(|f(x)-f^\star|\) & \(\|S_k\|\) & Path \\
\hline
SM--KKT & 501 & \( \leq 10^{-18} \) & \(9.871\times10^{-7}\) & \(1.893\)\\
Projected PD & -- & \(1.545\times10^{-1}\) & \(3.036\times10^{-1}\) & \(4.325\)\\
Aug. Lag. & 1315 & \( \leq 10^{-18} \) & \(1.015\times10^{-6}\) & \(1.951\)\\
Quad. Penalty & -- & \(5.026\times10^{-2}\) & \(9.181\times10^{-2}\) & \(1.650\)\\
Log-Barrier & -- & \(4.953\times10^{-2}\) & \(9.272\times10^{-2}\) & \(1.185\)\\
\hline
\end{tabular}
\end{table}
\subsection{Overall Discussion}

The four benchmark problems collectively demonstrate the
effectiveness of the proposed projection-free SM--KKT framework
across diverse nonlinear constrained optimisation scenarios,
including nonconvex objectives, Rosenbrock-type valleys,
annular feasible regions with active-set transitions, and
nonlinear equality manifolds with oscillatory objectives.
The results provide numerical evidence that direct regulation
of the unified KKT residual enables simultaneous treatment
of stationarity, primal and dual feasibility, and
complementarity. Across the considered benchmarks, SM--KKT
recovers the reference constrained solutions with high
numerical accuracy and achieves the smallest or comparably
small final KKT residuals. Furthermore, it is the only
tested implementation that consistently reaches the
prescribed residual tolerance and enters the residual
boundary layer in all four problems. Among the competing implementations, the augmented-Lagrangian
method generally exhibits the closest performance, attaining
comparable final accuracy in certain benchmarks. However,
under the numerical settings considered, it requires
substantially more iterations to enter the prescribed
residual boundary layer and exhausts the maximum iteration
budget. The projected primal--dual, quadratic-penalty,
and log-barrier implementations exhibit weaker overall
KKT accuracy in the present comparisons, with some
terminating before sufficiently reducing the complete
KKT residual. These observations do not imply inherent
limitations in the convergence capabilities of these
established methods, but highlight the differences in
their performance under the adopted benchmark conditions.

An important observation is that reduction of the
objective function alone does not establish convergence
to a constrained optimum. Accurate assessment also
requires satisfaction of the original constraints,
stationarity, dual feasibility, and complementarity.
The proposed SM--KKT framework addresses these conditions
through direct regulation of a unified residual system.
Complementarity is incorporated using the
Fischer--Burmeister residual, while the primal--dual
correction is computed without explicit projection
operations, penalty continuation, or barrier parameter
scheduling. This formulation offers a direct mechanism
for controlling KKT accuracy and provides a plausible
explanation for the favourable numerical behaviour
observed across the considered problems.

Table~\ref{tab:method_comparison} summarises the main
structural differences between the proposed approach
and the competing methods. The numerical results
indicate that SM--KKT combines accurate recovery of
the reference constrained optima with reliable
attainment of the prescribed residual tolerance
and favourable iteration-based convergence behaviour.
These observations complement the Jacobian
characterisation and LICQ/SOSC-based analysis
developed in the main manuscript. Under the
corresponding regularity, convergence, and
second-order sufficient conditions, a limiting
KKT point satisfying these assumptions is a
strict local minimiser. Overall, the findings
support the proposed framework as a promising
projection-free alternative for nonlinear
constrained optimisation, while further
large-scale testing and computational-time
comparisons remain necessary to establish
broader performance advantages.
\begin{table*}[t]
\centering
\caption{Qualitative comparison with representative methods. The proposed SM–KKT framework uniquely combines projection-free operation, direct KKT residual regulation, low tuning requirements, strong nonconvex capability, and rejection of non-optimal KKT points.}
\rowcolors{2}{gray!15}{white}
\label{tab:method_comparison}
\scriptsize
\setlength{\tabcolsep}{12pt}
\renewcommand{\arraystretch}{1.25}
\begin{tabular}{lcccccc}
\hline
Method & Proj.-free & Direct KKT & Comp. & Tuning & Noncvx & Rejects non-optimal KKT \\
\hline
SM--KKT & Yes & Yes & FB & Low & High & Yes \\
Projected PD & No & Partial & Projection & Low & Medium & No \\
Aug. Lag. & Yes & Indirect & Multiplier & High & High & No \\
Quad. Penalty & Yes & No & Penalty & High & Low & No \\
Log-Barrier & Yes & No & Barrier & High & Medium & No \\
Interior-Point & Yes & Indirect & Barrier & High & High & No \\
SQP & Yes & Approx. & QP-KKT & Medium & High & No \\
ADMM & Yes & No & Splitting & High & Medium & No \\
\hline
\end{tabular}
\end{table*}

The corresponding videos demonstrating the simulation results are available online at \href{https://www.youtube.com/playlist?list=PLHnMyAFhfD94}{[Video demonstrations]} \footnote{In case of difficulty accessing the embedded video link, the complete set of simulation videos is also available at:\\
\url{https://www.youtube.com/playlist?list=PLHnMyAFhfD94}.}.

\section{Conclusion}
\label{sec:conclusion}

This paper proposed a projection-free SM–KKT sliding-mode solver for NCO based on a unified KKT residual formulation. By combining stationarity, feasibility, and complementarity within a single residual system through a regularised Fischer–Burmeister mapping, the proposed framework eliminates the need for multiplier projection and penalty-based feasibility restoration. A residual-driven sliding dynamics together with a structured correction mechanism was developed to directly regulate the KKT residual manifold and guide the primal–dual trajectory towards optimality in finite time to a prescribed residual boundary layer.

From a theoretical perspective, a direct connection was established between the geometry of the residual Jacobian and classical nonlinear programming theory. In particular, it was shown that the residual Jacobian possesses inertia ((n,m+q,0)) if and only if LICQ and SOSC hold. Consequently, convergence of the proposed residual manifold to zero is associated with genuine locally optimal KKT points rather than merely feasible or stationary solutions. This result provides a new interpretation of constrained optimisation as a residual regulation problem and supplies a direct criterion linking the dynamics of the proposed system to classical optimality theory.

The numerical experiments demonstrated the effectiveness of the proposed framework across a diverse set of NCO problems, including nonconvex objectives, Rosenbrock-type valleys, annular feasible regions, and highly nonlinear trigonometric constraints. The proposed SM–KKT solver achieved the smallest final KKT residual in all four benchmarks, recovered the reference optima with essentially zero objective error, and attained KKT residuals below ($10^{-6}$) in every case. Moreover, it was the only method that successfully entered the prescribed residual boundary layer and reached the stopping tolerance in all benchmark problems, doing so substantially earlier than the competing methods. These results demonstrate that direct regulation of the complete KKT residual provides an effective and robust projection-free approach for nonlinear constrained optimization.

\appendix
\subsection{Proof of Lemma~\ref{lem:convexity_sc}} \label{appendix1}
Let $R_i$, $a_i$ and $G_i$ defined by: 
\[
R_i=\sqrt{\lambda_i^2+g_i(x)^2+\rho_c},~
a_i=\nabla g_i(x),
~
G_i=\nabla^2 g_i(x).
\]
The Hessian of \(S_i^c\) with respect to $y_i=
\begin{bmatrix}
x&
\lambda_i
\end{bmatrix}^T$ is
\[
\mathcal H_i^c
=
\begin{bmatrix}
D_i^gG_i+
\dfrac{\lambda_i^2+\rho_c}{R_i^3}a_ia_i^\top
&
-\dfrac{g_i\lambda_i}{R_i^3}a_i
\\[2mm]
-\dfrac{g_i\lambda_i}{R_i^3}a_i^\top
&
\dfrac{g_i^2+\rho_c}{R_i^3}
\end{bmatrix},
\]
where
\[
D_i^g=1+\frac{g_i(x)}{R_i}.
\]
Since \(\rho_c>0\), one has
\[
R_i>|g_i(x)| \implies D_i^g>0.
\]
Moreover,
\[
\dfrac{g_i^2+\rho_c}{R_i^3}>0.
\]

Using the Schur complement with respect to the lower-right block, it is sufficient to show
\[
D_i^gG_i+
\left[
\frac{\lambda_i^2+\rho_c}{R_i^3}
-
\frac{g_i^2\lambda_i^2}{R_i^3(g_i^2+\rho_c)}
\right]
a_ia_i^\top
\succeq0.
\]
The scalar coefficient satisfies
\[
\frac{\lambda_i^2+\rho_c}{R_i^3}
-
\frac{g_i^2\lambda_i^2}{R_i^3(g_i^2+\rho_c)}
=
\frac{\rho_c R_i^2}{R_i^3(g_i^2+\rho_c)}
=
\frac{\rho_c}{R_i(g_i^2+\rho_c)}
>0.
\]
Since \(g_i\) is convex,
\[
G_i=\nabla^2g_i(x)\succeq0.
\]
Thus

\[
D_i^gG_i\succeq0, ~~\frac{\rho_c}{R_i(g_i^2+\rho_c)}
a_ia_i^\top\succeq0.
\]
Hence the Schur complement is positive semidefinite, and therefore
\[
\mathcal H_i^c\succeq0.
\]
Consequently, \(S_i^c(x,\lambda_i)\) is convex in \((x,\lambda_i)\).
\subsection{Proof of Proposition~\ref{prop:kkt_equiv_s_zero}} \label{appendix2}
 Assume first that $S(z^\star)=0$, it follows that
\[
S^x(z^\star)=0 \implies \nabla_x
\mathcal L
(x^\star,\lambda^\star,\mu^\star)
=
0.
\]
Also, 
\[
S^h(z^\star)=0 \implies h(x^\star)=0.
\]

Finally, using the Fischer--Burmeister property in Lemma \ref{def:FB}
\[
S^c(z^\star)=-\phi_{\mathrm{FB}}(\lambda^*,-g^*(x))=0
\]
implies
\[
\lambda^\star\ge0,
\qquad
g(x^\star)\le0,
\qquad
\lambda^\star\odot g(x^\star)=0.
\]

Therefore, all KKT conditions \eqref{eq:kkt} are satisfied. Conversely, suppose that $(x^\star,\lambda^\star,\mu^\star)$
satisfies the KKT conditions \eqref{eq:kkt}. The stationarity condition gives
\[
\nabla_x\mathcal L(x^\star,\lambda^\star,\mu^\star)=0 \implies S^x(z^\star)=0,
\]
the equality constraint gives 
\[
h(x^\star)=0 \implies  S^h(z^\star)=0
\]
and the complementarity conditions imply
\[
S^c(z^\star)=-\phi_{\mathrm{FB}}(\lambda^*,-g(x^*))=0
\]
through the Lemma \ref{def:FB}. Therefore
\[
S(z^\star)=0.
\]

The two implications establish the equivalence.
\subsection{Proof of Theorem~\ref{thm:inertia_characterisation}} \label{appendix3}
At the KKT point $z^\star=(x^\star,\lambda^\star,\mu^\star),$ the residual mapping is
\[
S(z)
=
\begin{bmatrix}
S^x(z)\\
S^c(z)\\
S^h(z)
\end{bmatrix}
=
\begin{bmatrix}
\nabla_x\mathcal L(x,\lambda,\mu)\\
-\phi_{\mathrm{FB}}(\lambda,-g(x))\\
h(x)
\end{bmatrix}.
\]

We first compute the block derivatives of \(S\). The stationarity block satisfies
\[
S^x(z)
=
\nabla_x f(x)
+
J_g(x)^\top\lambda
+
J_h(x)^\top\mu .
\]
Hence
\[
\frac{\partial S^x}{\partial x}(z^\star)
=
H^\star,~
\frac{\partial S^x}{\partial \lambda}(z^\star)
=
J_g(x^\star)^\top,~
\frac{\partial S^x}{\partial \mu}(z^\star)
=
J_h(x^\star)^\top,
\]
where
\[
H^\star
=
\nabla_{xx}^{2}\mathcal L(x^\star,\lambda^\star,\mu^\star).
\]

For the complementarity block, define componentwise
\[
S^c_i(x,\lambda_i)
=
-\lambda_i
+
g_i(x)
+
\sqrt{\lambda_i^2+g_i(x)^2+\rho_c}.
\]
Then
\[
\frac{\partial S^c_i}{\partial x}
=
\left(
1+
\frac{g_i(x)}
{\sqrt{\lambda_i^2+g_i(x)^2+\rho_c}}
\right)
\nabla g_i(x)^\top,
\]
and
\[
\frac{\partial S^c_i}{\partial \lambda_i}
=
-1+
\frac{\lambda_i}
{\sqrt{\lambda_i^2+g_i(x)^2+\rho_c}}.
\]
Therefore, in matrix form,
\[
\frac{\partial S^c}{\partial x}(z^\star)
=
D_\star^g J_g(x^\star),~
\frac{\partial S^c}{\partial \lambda}(z^\star)
=
D_\star^\lambda,
~
\frac{\partial S^c}{\partial \mu}(z^\star)
=
0,
\]
with
\[
D_\star^g
=
\operatorname{diag}
\left(
1+
\frac{g_i(x^\star)}
{\sqrt{(\lambda_i^\star)^2+g_i(x^\star)^2+\rho_c}}
\right),
\]
and
\[
D_\star^\lambda
=
\operatorname{diag}
\left(
-1+
\frac{\lambda_i^\star}
{\sqrt{(\lambda_i^\star)^2+g_i(x^\star)^2+\rho_c}}
\right).
\]

The equality-feasibility block is
\[
S^h(z)=h(x),
\]
so
\[
\frac{\partial S^h}{\partial x}(z^\star)
=
J_h(x^\star),
~~
\frac{\partial S^h}{\partial \lambda}(z^\star)
=
0,~~
\frac{\partial S^h}{\partial \mu}(z^\star)
=
0.
\]

Thus, before separating active and inactive inequalities, the residual Jacobian has the block form
\[
J^\star
=
\begin{bmatrix}
H^\star & J_g(x^\star)^\top & J_h(x^\star)^\top\\
D_\star^gJ_g(x^\star) & D_\star^\lambda & 0\\
J_h(x^\star) & 0 & 0
\end{bmatrix}.
\]

We now exploit the KKT complementarity structure. For an inactive inequality constraint,
\[
i\in\mathcal I,
\qquad
g_i(x^\star)<0,
\qquad
\lambda_i^\star=0.
\]
In the non-regularised Fischer--Burmeister limit, or equivalently at the exact complementarity surface, this gives
\[
\frac{\partial S^c_i}{\partial x}(z^\star)=0,
\qquad
\frac{\partial S^c_i}{\partial \lambda_i}(z^\star)=-1.
\]
Hence each inactive inequality contributes an isolated negative diagonal direction.

For a strongly active inequality constraint,
\[
i\in\mathcal A,
\qquad
g_i(x^\star)=0,
\qquad
\lambda_i^\star>0,
\]
the complementarity block contributes
\[
\frac{\partial S^c_i}{\partial x}(z^\star)
=
\nabla g_i(x^\star)^\top,
\qquad
\frac{\partial S^c_i}{\partial \lambda_i}(z^\star)=0.
\]
Thus the active complementarity block contributes the active constraint Jacobian \(J_A(x^\star)\).

After reordering the variables as $z= \begin{bmatrix}
x &
\lambda_I&
\lambda_A&
\mu
\end{bmatrix}^T,$ the Jacobian becomes
\[
J^\star
=
\begin{bmatrix}
H^\star & 0 & J_A(x^\star)^\top & J_h(x^\star)^\top\\
0 & -I_{r_I} & 0 & 0\\
J_A(x^\star) & 0 & 0 & 0\\
J_h(x^\star) & 0 & 0 & 0
\end{bmatrix}.
\]

Introduce the active constraint Jacobian $C
=
\begin{bmatrix}
J_A(x^\star)&
J_h(x^\star)
\end{bmatrix}^T$. Then \(J^\star\) can be written as
\[
J^\star
=
\begin{bmatrix}
H^\star & 0 & C^\top\\
0 & -I_{r_I} & 0\\
C & 0 & 0
\end{bmatrix}.
\]

This matrix is block diagonal with respect to the inactive multiplier block. Therefore,
\[
\operatorname{In}(J^\star)
=
\operatorname{In}(K)
+
(0,r_I,0),
\]
where
\[
K
=
\begin{bmatrix}
H^\star & C^\top\\
C & 0
\end{bmatrix}.
\]

The matrix \(K\) is the classical symmetric saddle-point matrix associated with the active equality-constrained second-order system. By the saddle-point inertia theorem (see Lemma \ref{lem:kkt_inertia}),
\[
\operatorname{In}(K)
=
(n,r_A+q,0)
\]
if and only if
\[
\operatorname{rank}(C)=r_A+q,
\]
and
\[
Z^\top H^\star Z\succ0,
\qquad
CZ=0.
\]
The first condition is precisely the Linear Independence Constraint Qualification, since it states that the active inequality gradients and the equality constraint gradients are linearly independent. The second condition is precisely the second-order sufficient condition, since it states that the Lagrangian Hessian is positive definite on the tangent space of the active constraints. Therefore, if LICQ and SOSC hold, then
\[
\operatorname{In}(K)
=
(n,r_A+q,0).
\]
Consequently,
\[
\operatorname{In}(J^\star)
=
(n,r_A+q,0)
+
(0,r_I,0).
\]
Since $m=r_A+r_I,$ we obtain
\begin{equation}\label{ini}
    \operatorname{In}(J^\star)
=
(n,m+q,0).
\end{equation}

Conversely, suppose that \eqref{ini}.  The inactive block \(-I_{r_I}\) contributes exactly \(r_I\) negative eigenvalues. Removing this block yields
\[
\operatorname{In}(K)
=
(n,r_A+q,0).
\]
Applying the saddle-point inertia theorem again gives
\[
\operatorname{rank}(C)=r_A+q,~~Z^\top H^\star Z\succ0.
\]
Hence LICQ and SOSC hold. Therefore,
\[
\operatorname{In}(J^\star)=(n,m+q,0)
\]
if and only if the KKT point satisfies LICQ and SOSC. This proves that the residual Jacobian has the required saddle nature exactly at regular local minimisers. In particular, the KKT point is nondegenerate, locally isolated, and satisfies the second-order sufficient optimality condition.
\ifCLASSOPTIONcaptionsoff
  \newpage
\fi



%

\printbibliography[]

@article{kktcon,
  title={On enhanced KKT optimality conditions for smooth nonlinear optimization},
  author={Andreani, Roberto and Schuverdt, Maria L and Secchin, Leonardo D},
  journal={SIAM Journal on Optimization},
  volume={34},
  number={2},
  pages={1515--1539},
  year={2024},
  publisher={SIAM}
}

@article{conco,
  title={Non-convex optimization problems with linear KKT subsystem},
  author={Syed, Mujahid N},
  journal={Journal of Global Optimization},
  volume={94},
  number={2},
  pages={459--480},
  year={2026},
  publisher={Springer}
}

@article{ref2,
  title={On enhanced KKT optimality conditions for smooth nonlinear optimization},
  author={Andreani, Roberto and Schuverdt, Maria L and Secchin, Leonardo D},
  journal={SIAM Journal on Optimization},
  volume={34},
  number={2},
  pages={1515--1539},
  year={2024},
  publisher={SIAM}
}

@book{jahn2007introduction,
  title={Introduction to the theory of nonlinear optimization},
  author={Jahn, Johannes},
  year={2007},
  publisher={Springer}
}

@incollection{andrei2022sequential,
  title={Sequential quadratic programming},
  author={Andrei, Neculai},
  booktitle={Modern Numerical Nonlinear Optimization},
  pages={521--567},
  year={2022},
  publisher={Springer}
}

@article{de2010exact,
  title={Exact penalties for variational inequalities with applications to nonlinear complementarity problems},
  author={de Andr{\'e}, Thiago A and Silva, Paulo JS},
  journal={Computational Optimization and Applications},
  volume={47},
  number={3},
  pages={401--429},
  year={2010},
  publisher={Springer}
}

@incollection{jiang1998global,
  title={Global and local superlinear convergence analysis of Newton-type methods for semismooth equations with smooth least squares},
  author={Jiang, Houyuan and Ralph, Daniel},
  booktitle={Reformulation: nonsmooth, piecewise smooth, semismooth and smoothing methods},
  pages={181--209},
  year={1998},
  publisher={Springer}
}

@article{wright1998ill,
  title={Ill-conditioning and computational error in interior methods for nonlinear programming},
  author={Wright, Margaret H},
  journal={SIAM Journal on Optimization},
  volume={9},
  number={1},
  pages={84--111},
  year={1998},
  publisher={SIAM}
}

@incollection{murray1997some,
  title={Some aspects of sequential quadratic programming methods},
  author={Murray, Walter},
  booktitle={Large-Scale Optimization with Applications: Part II: Optimal Design and Control},
  pages={21--35},
  year={1997},
  publisher={Springer}
}

@article{shen2016stabilized,
  title={A stabilized filter SQP algorithm for nonlinear programming},
  author={Shen, Chungen and Zhang, Lei-Hong and Liu, Wei},
  journal={Journal of Global Optimization},
  volume={65},
  number={4},
  pages={677--708},
  year={2016},
  publisher={Springer}
}

@article{solodov2004sequential,
  title={On the sequential quadratically constrained quadratic programming methods},
  author={Solodov, Mikhail V},
  journal={Mathematics of Operations Research},
  volume={29},
  number={1},
  pages={64--79},
  year={2004},
  publisher={INFORMS}
}

@article{birgin2025global,
  title={On the Global Convergence of a General Class of Augmented Lagrangian Methods: EG Birgin et al.},
  author={Birgin, Ernesto G and Haeser, Gabriel and Maculan, Nelson and Ramirez, Lennin Mallma},
  journal={Journal of Optimization Theory and Applications},
  volume={206},
  number={3},
  pages={57},
  year={2025},
  publisher={Springer}
}

@article{chen2016penalty,
  title={Penalty methods for a class of non-Lipschitz optimization problems},
  author={Chen, Xiaojun and Lu, Zhaosong and Pong, Ting Kei},
  journal={SIAM Journal on Optimization},
  volume={26},
  number={3},
  pages={1465--1492},
  year={2016},
  publisher={SIAM}
}

@article{stegink2018convergence,
  title={Convergence of projected primal-dual dynamics with applications in data centers},
  author={Stegink, Tjerk W and Van Damme, Tobias and De Persis, Claudio},
  journal={IFAC-PapersOnLine},
  volume={51},
  number={23},
  pages={88--93},
  year={2018},
  publisher={Elsevier}
}

@article{zhu2018projected,
  title={Projected primal--dual dynamics for distributed constrained nonsmooth convex optimization},
  author={Zhu, Yanan and Yu, Wenwu and Wen, Guanghui and Chen, Guanrong},
  journal={IEEE transactions on cybernetics},
  volume={50},
  number={4},
  pages={1776--1782},
  year={2018},
  publisher={IEEE}
}

@article{fernandez2012local,
  title={Local convergence of exact and inexact augmented Lagrangian methods under the second-order sufficient optimality condition},
  author={Fern{\'a}ndez, Dami{\'a}n and Solodov, Mikhail V},
  journal={SIAM Journal on Optimization},
  volume={22},
  number={2},
  pages={384--407},
  year={2012},
  publisher={SIAM}
}

@article{kien2018second,
  title={Second-order KKT optimality conditions for multiobjective optimal control problems},
  author={Kien, Bui Trong and Tuyen, Nguyen Van and Yao, Jen-Chih},
  journal={SIAM Journal on Control and Optimization},
  volume={56},
  number={6},
  pages={4069--4097},
  year={2018},
  publisher={SIAM}
}

@article{meng2015first,
  title={First-and second-order necessary conditions via exact penalty functions},
  author={Meng, Kaiwen and Yang, Xiaoqi},
  journal={Journal of Optimization Theory and Applications},
  volume={165},
  number={3},
  pages={720--752},
  year={2015},
  publisher={Springer}
}

@article{voloshinov2018generalization,
  title={A generalization of the Karush--Kuhn--Tucker theorem for approximate solutions of mathematical programming problems based on quadratic approximation},
  author={Voloshinov, Vladimir Vladimirovich},
  journal={Computational Mathematics and Mathematical Physics},
  volume={58},
  number={3},
  pages={364--377},
  year={2018},
  publisher={Springer}
}

@article{syed2026non,
  title={Non-convex optimization problems with linear KKT subsystem},
  author={Syed, Mujahid N},
  journal={Journal of Global Optimization},
  volume={94},
  number={2},
  pages={459--480},
  year={2026},
  publisher={Springer}
}

@article{ref3,
  title={Regularized interior point methods for constrained optimization and control},
  author={De Marchi, Alberto},
  journal={IFAC-PapersOnLine},
  volume={56},
  number={2},
  pages={1247--1252},
  year={2023},
  publisher={Elsevier}
}

@article{ref9,
  title={A penalized Fischer-Burmeister NCP-function},
  author={Chen, Bintong and Chen, Xiaojun and Kanzow, Christian},
  journal={Mathematical Programming},
  volume={88},
  number={1},
  pages={211--216},
  year={2000},
  publisher={Springer}
}

@book{ref11,
  title={Numerical optimization},
  author={Nocedal, Jorge and Wright, Stephen J},
  year={2006},
  publisher={Springer}
}

@article{ref10,
  title={Inertia and rank characterizations of some matrix expressions},
  author={Chu, Delin and Hung, YS and Woerdeman, Hugo J},
  journal={SIAM Journal on Matrix Analysis and Applications},
  volume={31},
  number={3},
  pages={1187--1226},
  year={2010},
  publisher={SIAM}
}

@incollection{ref8,
  title={Sequential quadratic programming methods},
  author={Gill, Philip E and Wong, Elizabeth},
  booktitle={Mixed integer nonlinear programming},
  pages={147--224},
  year={2011},
  publisher={Springer}
}

@article{ref6,
  title={The Augmented Lagrangian Methods: Overview and Recent Advances},
  author={Deng, Kangkang and Wang, Rui and Zhu, Zhenyuan and Zhang, Junyu and Wen, Zaiwen},
  journal={arXiv preprint arXiv:2510.16827},
  year={2025}
}

@article{ref5,
  title={Control-barrier-function-based design of gradient flows for constrained nonlinear programming},
  author={Allibhoy, Ahmed and Cort{\'e}s, Jorge},
  journal={IEEE Transactions on Automatic Control},
  volume={69},
  number={6},
  pages={3499--3514},
  year={2023},
  publisher={IEEE}
}

@article{ref4,
  title={Worst-case iteration bounds for log barrier methods on problems with nonconvex constraints},
  author={Hinder, Oliver and Ye, Yinyu},
  journal={Mathematics of Operations Research},
  volume={49},
  number={4},
  pages={2402--2424},
  year={2024},
  publisher={INFORMS}
}

@article{k,
  title={Efficient implicit constraint handling approaches for constrained optimization problems},
  author={Rahimi, Iman and Gandomi, Amir H and Nikoo, Mohammad Reza and Mousavi, Mohsen and Chen, Fang},
  journal={Scientific Reports},
  volume={14},
  number={1},
  pages={4816},
  year={2024},
  publisher={Nature Publishing Group UK London}
}

@article{ref1,
  title={Complexity of an inexact proximal-point penalty method for constrained smooth non-convex optimization},
  author={Lin, Qihang and Ma, Runchao and Xu, Yangyang},
  journal={Computational optimization and applications},
  volume={82},
  number={1},
  pages={175--224},
  year={2022},
  publisher={Springer}
}

@article{kkt,
  title={First-order methods for nonsmooth nonconvex functional constrained optimization with or without slater points},
  author={Jia, Zhichao and Grimmer, Benjamin},
  journal={SIAM Journal on Optimization},
  volume={35},
  number={2},
  pages={1300--1329},
  year={2025},
  publisher={SIAM}
}

@article{primal2,
  title={A Second-Order Primal--Dual Dynamics for Set Constrained Distributed Optimization Problems},
  author={Tao, Meng and Guo, Luyao and Cao, Jinde and Rutkowski, Leszek},
  journal={IEEE Transactions on Circuits and Systems II: Express Briefs},
  volume={71},
  number={3},
  pages={1316--1320},
  year={2023},
  publisher={IEEE}
}

@article{primal,
  title={Accelerated primal-dual projection neurodynamic approach with time scaling for linear and set constrained convex optimization problems},
  author={Zhao, You and He, Xing and Zhou, Mingliang and Huang, Tingwen},
  journal={IEEE/CAA Journal of Automatica Sinica},
  volume={11},
  number={6},
  pages={1485--1498},
  year={2024},
  publisher={IEEE}
}

@book{realloca,
  title={Modern numerical nonlinear optimization},
  author={Andrei, Neculai},
  volume={195},
  year={2022},
  publisher={Springer}
}

@article{cntr,
  title={Nonlinear programming solvers for unconstrained and constrained optimization problems: a benchmark analysis},
  author={Lavezzi, Giovanni and Guye, Kidus and Ciarci{\`a}, Marco},
  journal={arXiv preprint arXiv:2204.05297},
  year={2022}
}

@article{ML,
  title={Constrained optimization from a control perspective via feedback linearization},
  author={Zhang, Runyu and Raghunathan, Arvind and Shamma, Jeff and Li, Na},
  journal={Advances in Neural Information Processing Systems},
  volume={38},
  pages={55944--55980},
  year={2026}
}

@article{robotics,
  title={Multi-robot formation control and object transport in dynamic environments via constrained optimization},
  author={Alonso-Mora, Javier and Baker, Stuart and Rus, Daniela},
  journal={The International Journal of Robotics Research},
  volume={36},
  number={9},
  pages={1000--1021},
  year={2017},
  publisher={SAGE Publications Sage UK: London, England}
}




\balance


\end{document}